\documentclass[11pt,a4paper]{amsart}
\usepackage{geometry}
 \newenvironment{rcases}
  {\left.\begin{aligned}}
  {\end{aligned}\right\rbrace}

\newcommand*{\triple}[2][.1ex]{%
  \mathrel{\vcenter{\offinterlineskip%
  \hbox{$#2$}\vskip#1\hbox{$#2$}\vskip#1\hbox{$#2$}}}}
  
\newcommand*{\triplerightarrow}{\triple{\rightarrow}}

\usepackage[usenames,dvipsnames]{xcolor}  
\definecolor{color-text}{gray}{0.085}  \definecolor{color-detail}{gray}{0.70}  
\makeatletter
\newcommand{\globalcolor}[1]{\color{#1}\global\let\default@color\current@color}
\makeatother
\AtBeginDocument{\globalcolor{color-text}}
\newcommand\scalemath[2]{\scalebox{#1}{\mbox{\ensuremath{\displaystyle #2}}}}

\usepackage{amssymb}
\usepackage{tikz, pgfplots}
\usetikzlibrary{positioning}

\usepackage[linkcolor=blue,colorlinks=blue]{hyperref} 
\hypersetup{colorlinks=true, citecolor=blue, filecolor=Maroon,
 linkcolor=blue, urlcolor=black!65!purple}
\usepackage[OT2,T1]{fontenc}
\usepackage{helvet}
\DeclareSymbolFont{cyrletters}{OT2}{wncyr}{m}{n}
\DeclareMathSymbol{\E}
{\mathalpha}{cyrletters}{"03}
\definecolor{fondtitre}{RGB}{85,85,85}
\definecolor{fonddeboite}{RGB}{232,232,232}
\definecolor{shadecolor}{RGB}{232,232,232}

\usepackage[flushleft]{threeparttable}
\usepackage{array,booktabs,makecell}
\usepackage[font=small]{caption}

\usepackage{mathabx}
\usepackage{euscript}

\usepackage{slashed}

\newcommand{\Bigwedge}{\scalemath{.90}{\bigwedge}}

\newcommand{\op}{\mathrm{op}}

\newcommand{\PS}{\mathsf{PreStk}}
\newcommand{\Q}{\mathsf{QCoh}}
\newcommand{\IC}{\mathsf{IndCoh}}

\newcommand{\M}{\mathsf{Maps}}
\newcommand{\D}{\scalemath{.90}{\mathcal{D}}}

\newcommand{\J}{\mathsf{Jets}_{dR}^{\infty}}
\newcommand{\JetX}{\mathsf{Jets}_X^{\infty}}

\newcommand{\EQ}{\EuScript{Y}}

\newcommand{\RS}{{\color{white!05!black}\mathbb{R}\underline{\EuScript{S}\mathsf{oL}}}}
\newcommand{\cdga}{{\color{white!10!black}\mathbf{cdga}_{\mathcal{D}_X}^{\leq 0}}}

\newcommand{\DG}
{{\color{white!10!black}\mathbf{dg}_{\mathcal{D}_X}}}

\usepackage{tikz-cd}
\tikzset{%
    symbol/.style={%
        draw=none,
        every to/.append style={%
            edge node={node [sloped, allow upside down, auto=false]{$#1$}}}
    }
}
\usepackage{adjustbox}

\newcommand{\verteq}{\rotatebox{90}{$\,=$}}
\newcommand{\equalto}[2]{\underset{\scriptstyle\overset{\mkern4mu\verteq}{#2}}{#1}}

\usepackage{wrapfig}

\usepackage{amsthm}
\newtheorem{theorem}{Theorem}[section]
\newtheorem{proposition}[theorem]{Proposition}
\newtheorem{corollary}[theorem]{Corollary}

\newtheorem{definition}[theorem]{Definition}%
\newtheorem{example}[theorem]{Example}%
\newtheorem{remark}[theorem]{Remark}%
\newtheorem{observation}{Observation}[section]

\newcommand{\JN}{\mathsf{Jets}_{\mathbb{A}^1}^{1}[n]X}
\newcommand{\Jn}{\mathsf{Jets}_{\mathbb{A}^1}^{\dagger}[n]X}

\newcommand{\EL}{{\scalemath{.92}{\mathrm{E}}{\scalemath{.92}{\mathrm{L}}}}}
\newcommand{\BV}{{\scalemath{.92}{\mathrm{B}}{\scalemath{.92}{\mathrm{V}}}}}
\newcommand{\KT}{{\scalemath{.92}{\mathrm{K}}{\scalemath{.92}{\mathrm{T}}}}}

\usepackage{tikz-cd}
\tikzset{%
    symbol/.style={%
        draw=none,
        every to/.append style={%
            edge node={node [sloped, allow upside down, auto=false]{$#1$}}}
    }
}
\usepackage{adjustbox}

\numberwithin{equation}{section}

\begin{document}
\title[Derived Variational Tricomplex]{Derived Moduli Spaces of Nonlinear PDEs II: Variational Tricomplex and BV Formalism}

\author[J. Kryczka]{Jacob Kryczka}
\address{Jacob Kryczka}
\address{Beijing Institute of Mathematical Sciences and Applications, A4, Room 209
No. 544, Hefangkou Village, Huaibei Town, Huairou District, Beijing 101408,
China}
\email{jkryczka@bimsa.cn}

\author[A. Sheshmani]{Artan Sheshmani}
\address{Artan Sheshmani}
\address{Beijing Institute of Mathematical Sciences and Applications, A6, Room 205
No. 544, Hefangkou Village, Huaibei Town, Huairou District, Beijing 101408,
China}
\address{Massachusetts Institute of Technology (MIT), IAiFi Institute, 182 Memorial Drive, Cambridge, MA 02139, USA}
\address{National Research University, Higher School of Economics, Russian Federation, Laboratory of Mirror Symmetry, NRU HSE, 6 Usacheva Street, Moscow, Russia, 119048}
\email{artan@mit.edu}

\author[S-T. Yau]{Shing-Tung Yau}
\address{Shing-Tung Yau}
\address{Yau Mathematical Sciences Center, Tsinghua University, Haidian District,
Beijing, China}
\address{Beijing Institute of Mathematical Sciences and Applications,
No. 544, Hefangkou Village, Huaibei Town, Huairou District, Beijing 101408,
China}
\email{styau@tsinghua.edu.cn}

\date{\today}


\maketitle

\begin{abstract}
This paper is the second in a series of works dedicated to studying non-linear partial differential equations via derived geometric methods. We study the natural derived enhancement of the de Rham complex of non-linear PDEs via algebro-geometric techniques and examine its consequences for the up-to-homotopy (functional) differential calculus on solution spaces of such systems. Applications to the BV-formalism with and without boundary conditions are discussed.
\end{abstract}

\maketitle
\tableofcontents

\newpage 

\section{Introduction}
Shifted symplectic derived algebraic geometry \cite{PTVV} serves as the natural language for making sense of symplectic or Poisson  structures on singular algebro-geometric moduli spaces e.g. stable sheaves on Calabi-Yau 3,4 folds.
It has seen successful application to many areas of mathematics and physics such as: the study of derived critical loci in field theories, enumerative algebraic geometry e.g. Gromov-Witten, Donaldon-Thomas theory and the theory of PT-stable pairs, as well as to derived intersection theory of derived quotient stacks which describe for example, symplectic reductions. 
\vspace{1.5mm}

As it is well-known, in this setting standard notions such as the property of a differential $2$-form being closed only hold up to homotopy and appear instead as a kind of coherency \emph{data}. What seems to be less known however, is what this story looks like in the algebro-geometric study of moduli spaces of non-linear partial differential equations (${\scalemath{.90}{\mathrm{N}}{\scalemath{.84}{\mathrm{L}}}{\scalemath{.88}{\mathrm{P}}{\scalemath{.84}{\mathrm{DE}}{\scalemath{.70}{\mathrm{S}}}}}}$).

\subsection{What do we achieve?}
\label{ssec: What do we achieve?}
We address this topic by applying a similar homotopical framework to study ${\scalemath{.90}{\mathrm{N}}{\scalemath{.84}{\mathrm{L}}}{\scalemath{.88}{\mathrm{P}}{\scalemath{.84}{\mathrm{DE}}{\scalemath{.70}{\mathrm{S}}}}}}$ and their spaces of solutions \cite{KRY,KSYI}. We introduce a homotopy invariant analog of the Rham complex of a PDE $\EQ$, which in the classical jet-bundle approach to ${\scalemath{.90}{\mathrm{N}}{\scalemath{.84}{\mathrm{L}}}{\scalemath{.88}{\mathrm{P}}{\scalemath{.84}{\mathrm{DE}}{\scalemath{.70}{\mathrm{S}}}}}}$ \cite{KraVinLyc} is called the variational bi-complex \cite{Anderson}. 
\vspace{1.5mm}

Roughly, this enhancement is given by a differentially structured cotangent complex  $\mathbb{L}_{\EQ}$ that we refer to as the \emph{derived variational tricomplex}. We initiate a study of its global features as they pertain to the geometric theory and cohomological analysis \cite{VinCohomPDE} of ${\scalemath{.90}{\mathrm{N}}{\scalemath{.84}{\mathrm{L}}}{\scalemath{.88}{\mathrm{P}}{\scalemath{.84}{\mathrm{DE}}{\scalemath{.70}{\mathrm{S}}}}}}$ from the point of view of $\D$-module theory, which gives a satisfactory setting to study analogs of shifted symplectic structures for ${\scalemath{.90}{\mathrm{N}}{\scalemath{.84}{\mathrm{L}}}{\scalemath{.88}{\mathrm{P}}{\scalemath{.84}{\mathrm{DE}}{\scalemath{.70}{\mathrm{S}}}}}}$ and their moduli spaces of solutions.
\subsection{Background and motivation}
Before stating our results (Theorems A,B as summarized in Sect. \ref{ssec: Main Results} below) we dedicate some space to putting these results in context and providing motivation for them. Specifically, let us highlight some of the main features which we believe illuminate some deep connections
between PDE theory and geometric analysis, physics and algebraic geometry.
\begin{enumerate}
    \item \textcolor{blue}{\textit{Homotopical geometric analysis of PDEs:}} global derived geometry combined with ideas from PDE theory can be seen as a tool for studying non-perturbative constructions of classical field theories while also encoding perturbative phenomena via formal geometry (see Subsect.\ref{ssec: HoPDE and Quantization}).
    
    Roughly speaking, the latter is captured by 
fixing a classical solution $\varphi_0\in \mathsf{Sol}$ and studying the formal neighbourhood in the homotopical space of solutions $\widehat{\mathrm{\mathbb{R}\mathsf{Sol}}}_{\varphi_0},$ which can only see infinitesimal information. Indeed, in the presence of gauge symmetries with gauge group $\mathcal{G}$, for a solution $\varphi_0:\star\rightarrow \mathrm{Sol}/\mathcal{G},$ this translates to the fact that $\varphi_0$ factors canonically through the completion
$\big[\mathsf{Sol}/\mathcal{G}\big]_{\varphi_0}^{\wedge}\hookrightarrow \big[\mathsf{Sol}/\mathcal{G}\big].$ 

Mathematically, such completions are defined by restricting prestacks of solutions to certain categories of nilpotent homotopical thickenings. Such objects are completely determined by a tangent dg-Lie algebra $\mathfrak{g}^{\bullet},$ via general philosophy of derived deformation theory adapted to the setting of $\D$-geometry.\footnote{In the language of \cite{KSYI} one says we have a formal $\mathcal{D}_X$-moduli
problem controlling the infinitesimal neighbourhood of $\varphi_0$ on the moduli space of global solutions to a PDE. Roughly, it corresponds to a dg-Lie algebra (chiral deformation complex) in a pseudo-tensor category of $\D$-modules (see \cite{BeiDri} for what this means).}
\vspace{2mm}
 
     \item \textcolor{blue}{\textit{Finite dimensionality:}} in the derived $\D$-geometrical setting many spaces e.g. $\mathfrak{g}$, can be viewed as finite dimensional in a differential sense. With some care regarding homotopical spaces of operations (Sect. \ref{sec: Geometry of D-Prestacks}) important duality statements hold and one can makes sense  of local (differentially structured) homotopy Poisson brackets on solution spaces. In particular, invariant pairings on $\mathfrak{g}^{\bullet}$ are encoded by such Poisson structure on $CE^*(\mathfrak{g}^{\bullet}),$ the Chevalley-Eilenberg algebra that in turn, completely determines the formal moduli problem as in (1).
\vspace{2mm}

    \item \textcolor{blue}{\textit{Non-Lagrangian theories}:} 
    Working with moduli spaces of solutions to ${\scalemath{.90}{\mathrm{N}}{\scalemath{.84}{\mathrm{L}}}{\scalemath{.88}{\mathrm{P}}{\scalemath{.84}{\mathrm{DE}}{\scalemath{.70}{\mathrm{S}}}}}}$-- as opposed to say, critical loci of functionals -- permits one to treat, in a unified manner, physical theories admitting a Lagrangian formulation as well as those that do not. Ubiquitous examples of the latter include theories with self-dual field strengths, chiral bosons, higher-spin gauge theories and superconformal theories with extended supersymmetry. Other fundamental examples coming from non-linear gauge theories which again, do not admit Lagrangian descriptions are the Seiberg-Witten equations and the Donaldson-Uhlenbeck-Yau system.
    \vspace{2mm}
    
    \item \textcolor{blue}{\textit{Categorical $\D$-Linearity:}} A pleasant feature of global algebraic $\D$-geometry is its available functorialities e.g. $\D$-linearity (as opposed to just $\mathcal{O}$-linearity). This is much better for making universal constructions and becomes a crucial feature when passing to a homotopical setting where the usage of higher categories is unavoidable. In turn it offers a more natural alternative to defining a category of generalized PDEs \cite{CatPDEs} which was originally done by imposing by hand that all morphisms and constructions respect the underlying Cartan distribution on jets. 
\end{enumerate}
\subsubsection{Variational bicomplex.} The variational bicomplex $\Omega^{*,*}\big(\mathrm{Jet}_X^{\infty}(E)\big)$ (shown below) is a sheaf of commutative monoid objects in bi-complexes defined by differential forms on
the infinite jet bundle $\mathrm{Jet}_X^{\infty}(E)$ of any fibered manifold $\pi:E\rightarrow X.$ It is known to play, roughly speaking, the same role in the geometry of PDEs as the usual de Rham complex plays in the geometry of a
manifold or variety $X$.

\begin{wrapfigure}[18]{r}{7.4cm}
    \centering
    \adjustbox{scale=.86}{
\begin{tikzcd}
\label{VarBicomplex}
& \mathbb{R}\arrow[d] & 0\arrow[d] 
\\
\mathcal{O}_X \arrow[d] \arrow[r] & \mathcal{O}(\mathrm{Jet}_X^{\infty}(E)) \arrow[d,"d_h"] \arrow[r,"d^v"] &  \Omega^{0,1}(\pi) \arrow[d,"d_h"] \arrow[r,"d^v"] & \ldots
\\
\Omega_X^1 \arrow[d]\arrow[r,"d^v"]&  \overline{\Omega}^1(\pi)\arrow[d] \arrow[r] & \Omega^{1,1}(\pi)\arrow[d]\arrow[r] & \ldots 
\\
\vdots \arrow[d] & \vdots\arrow[d] & \vdots \arrow[d]
\\
\Omega_X^n\arrow[r] &  \overline{\Omega}^n(\pi)\arrow[d]\arrow[r] & \Omega^{n,1}(\pi) \arrow[d]\arrow[r] & \ldots
\\
& 0 & 0
\end{tikzcd}}
\end{wrapfigure}

Together with an associated spectral sequence $\mathcal{C}E_r^{p,q}$ (Vinogradov's $\mathcal{C}$-spectral sequence \cite{CSpectral2}), it gives a powerful differential geometric tool for studying the calculus of variations, the theory
of conservation laws for systems of differential equations, and the theory of (secondary) characteristic classes (Subsect. \ref{sssec: CharClasses}). These cohomologies represent deformation invariants of ${\scalemath{.90}{\mathrm{N}}{\scalemath{.84}{\mathrm{L}}}{\scalemath{.88}{\mathrm{P}}{\scalemath{.84}{\mathrm{DE}}{\scalemath{.70}{\mathrm{S}}}}}}$ and in the case of the Euler-Lagrange (\EL) system, contain relevant physical information e.g. variational principles, gauge symmetries, charges etc. (Subsect. \ref{ssec: Applications to Field Theory}). 
Another of its many uses is due to the fact that the vertical cohomology $H^{p,q}(\mathrm{Jet}^{\infty}(E),d^v)$ agrees with the $E_1$-term of the Serre spectral sequence for the bundle $E$ with typical fibre $F$ i.e. is isomorphic to $\Omega^p(X)\otimes H^q(F).$

A natural differential calculus as well as analogs of (pre)symplectic and Poisson geometry exist on  $\mathrm{Jet}_X^{\infty}(E)$ and their differentially structured sub-manifolds $\mathcal{E}^{\infty}$ known as Diffieties\footnote{This name, due to A.M.Vinogradov, was chosen to emphasize that \emph{diff}(erential var)\emph{ieties} are to partial differential equations just 
as algebraic varieties are to polynomial equations.}, but such classical notions like symplectic forms and Poisson brackets have a slightly different flavour and appear in a rather different context (see \cite{KraVerGeomJets}).

Indeed, in the geometric theory of ${\scalemath{.90}{\mathrm{N}}{\scalemath{.84}{\mathrm{L}}}{\scalemath{.88}{\mathrm{P}}{\scalemath{.84}{\mathrm{DE}}{\scalemath{.70}{\mathrm{S}}}}}}$ and from the perspective of both integrable systems and classical field theory, the analog of the structure sheaf are the local functionals, that appear as certain cohomologies naturally attributed to PDEs. This idea forms the basis of Vinogradov's Secondary Calculus \cite{SecCalcCohomPhysics,VinCohomPDE} (see also \cite{Vit01}) and provides an appropriate setting for a (cohomological) differential calculus on (infinite-dimensional) spaces of solutions to ${\scalemath{.90}{\mathrm{N}}{\scalemath{.84}{\mathrm{L}}}{\scalemath{.88}{\mathrm{P}}{\scalemath{.84}{\mathrm{DE}}{\scalemath{.70}{\mathrm{S}}}}}}$ based on the variational bicomplex.

\subsubsection{Variational tricomplexes.}
\label{sssec: Variational tricomplexes}
There are two well known constructions of the critical locus $\mathrm{Crit}_X(\mathcal{S})$ of a functional $\mathcal{S}:X\rightarrow k$ on a smooth affine $k$-variety $X$ (\cite{CME} and \cite{CosGwi02,CosGwi01}). They are both of an infinitesimal nature and roughly speaking amount to resolving the structure sheaf $\mathcal{O}_{\mathrm{Crit}_X(\mathcal{S})}:=\mathcal{O}_X/\mathcal{I}$ with $\mathcal{I}=\big<\frac{\partial\mathcal{S}}{\partial x}\big>,$ via standard homological methods i.e. a Koszul-Tate (\KT) resolution 
$$\mathcal{O}_{\mathrm{KT}}:=\mathrm{Sym}_{\mathcal{O}_X}\big(\mathbb{T}_X[1]\oplus \mathfrak{g}_{\mathrm{KT}}[2]\big).$$
The so-called \KT-generators $\mathfrak{g}_{\mathrm{KT}}[2]$ kill the extra terms in cohomological degrees $\leq 0$. The \KT-differential is determined by contraction with $d\mathcal{S}.$
\vspace{2mm}

One of the first appearances of the variational tri-complex appeared in a concrete application of jet-bundle geometry to the context of the \BV-formalism for quantization of gauge systems \cite{BVAnti}. In particular, it gave a jet-bundle formulation of the \KT-resolution by combining the bi-complex construction with a resolution of the quotient $\D$-algebra defined by the equation ideal
$$\mathcal{O}(\mathrm{Jet}^{\infty}(E))\hookrightarrow \underbrace{\mathcal{O}(\mathrm{Jet}^{\infty}(E))\otimes_{\mathcal{O}_X}Sym_{\mathcal{O}}^*(\mathcal{V}^{\bullet})}_{\mathcal{A}_{KT}}\xrightarrow{\simeq}\mathcal{O}(\mathrm{Jet}^{\infty}(E))/\big<\mathsf{P}\big>,\hspace{1mm} \mathsf{P}\in \mathcal{O}(\mathrm{Jet}^{\infty}E).$$

This construction enhances
the cohomological setup of the variational bicomplex associated to possibly degenerate ${\scalemath{.90}{\mathrm{N}}{\scalemath{.84}{\mathrm{L}}}{\scalemath{.88}{\mathrm{P}}{\scalemath{.84}{\mathrm{DE}}{\scalemath{.70}{\mathrm{S}}}}}}$ via pullback from
the free bicomplex (Proposition \ref{prop: MorphismOfVarBiSeq}) to the bicomplex of the surface defined by the equations $\mathsf{P}.$ A physicist would replace $E$ by an appropriately resolved bundle $K$ (with sections anti-ghosts anti-fields etc.) and consider $\Omega^{(*,*)}\big(\mathrm{Jet}^{\infty}(K)\big)$ with an additional grading (ghost number). Characteristic cohomologies of the original PDE $\EQ$ are computed as relative cohomologies of the free tricomplex (of ghost number $0$) i.e. 
$$H^{p,*}\big(d_h,\Omega^{p,*}(\EQ)\big)\simeq H_0^{p,*}\big(d_h|\delta,\Omega_0^{p,*}(\mathrm{Jet}^{\infty}K)\big).$$
The cohomologies of the tri-complex are known to relate various Lie algebras associated with global symmetries to conservation laws of classical gauge systems (see \emph{loc.cit.})
\vspace{1.5mm}

An intrinsic formulation in terms of the homotopy theory of ${\scalemath{.90}{\mathrm{N}}{\scalemath{.84}{\mathrm{L}}}{\scalemath{.88}{\mathrm{P}}{\scalemath{.84}{\mathrm{DE}}{\scalemath{.70}{\mathrm{S}}}}}}$ essentially amounts to describing a derived $\D_X$-scheme
$$\mathcal{X}_{\mathrm{KT}}:=\mathbf{Spec}_{\mathcal{D}_X}(\mathcal{A}_{\mathrm{KT}}):\cdga_{\mathcal{A}/}\rightarrow SSets,$$
as a homotopy class\footnote{Since $ho\big(Comm(\mathsf{Mod}(\D_X)^{\leq 0})_{\mathcal{A}/}\big)\simeq \cdga_{\mathcal{A}/}[qis^{-1}].$} in the $(\infty,1)$-category $$\mathsf{PreStk}_{X}(\D_X)_{\mathcal{A}/}:=\mathsf{Fun}\big(Comm\big(\mathsf{Mod}(\mathcal{D}_X)^{\leq 0}\big)_{\mathcal{A}/},\mathsf{Spc}\big).$$

The derived tri-complex we study in this paper by adapting the usual bi-complex construction to geometric spaces with dg-algebras (derived algebras) of functions e.g. $\mathcal{X}_{\mathrm{KT}},$ therefore gives a very general setting in which the above (underived) variational tri-complex fits naturally.

Incorporating both co-connective and connective derived algebras of functions, as is possible in derived geometry, we can extend
the above problem to treat the full BRST differential (stacky part) and not just its Koszul-Tate (derived) part.
In this way, the methods elaborated in the current paper represent a step towards establishing a \emph{global} formalism for non-perturbative (or higher categorical) \BV-quantization from the perspective of an algebro-geometric and sheaf-theoretic approach to the geometric study of ${\scalemath{.90}{\mathrm{N}}{\scalemath{.84}{\mathrm{L}}}{\scalemath{.88}{\mathrm{P}}{\scalemath{.84}{\mathrm{DE}}{\scalemath{.70}{\mathrm{S}}}}}}$.

\subsection{Homotopical PDE Geometry and (Non) Perturbative BV-Quantization}
\label{ssec: HoPDE and Quantization}
A geometric characterization of the \BV-construction is obtained by considering the derived quotient $\mathbf{BV}_{X,\mathcal{S}}$ of $\mathrm{Crit}_X(\mathcal{S})$ by the Lie algebroid of maximal symmetries.

A model for the \BV complex is
$$\mathcal{O}_{\mathrm{BV}}:=\mathrm{Sym}_{\mathcal{O}_X}\big(\mathbb{T}_X[1]\oplus\mathfrak{g}_{\mathrm{KT}}[2]\oplus \mathcal{L}^{\vee}[-1]\big).$$
It carries a canonical $(-1)$-shifted symplectic structure corresponding to a canonical pairing between $\mathcal{O}_X$ and $\mathbb{T}_X[1]$ as well as $\mathfrak{g}_{\mathrm{KT}}[2]$ and $\mathcal{L}^{\vee}[-1].$ The corresponding shifted Poisson bracket is what physicists call the anti-bracket $\{-,-\},$ and there is a degree zero \BV-charge $\mathcal{S}_{\mathrm{BV}}$ satisfying $\{\mathcal{S}_{\mathrm{BV}},\mathcal{S}_{\mathrm{BV}}\}=0,$ thereby defining a differential $D_{\mathrm{BV}}:=\{\mathcal{S}_{\mathrm{BV}},-\}$, whose corresponding cohomology, denoted $H_{D_{\mathrm{BV}}}^i$ is such that
$$H_{D_{\mathrm{BV}}}^0\big(\mathbf{BV}_{\mathcal{S},X}\big)=\big(\mathcal{O}_X/\mathcal{I}\big)^{\mathcal{L}},\hspace{3mm}\text{ and }\hspace{1mm} H_{D_{\mathrm{BV}}}^{i<0}\big(\mathbf{BV}_{\mathcal{S},X}\big)=0.$$
The construction of \cite{CME} sees only the degree $0$ cohomology, corresponding to taking invariants by the flows of bracket and thus quotienting by maximal symmetries, and its degree $1$ piece.
To obtain the full information, one must work with a derived enhancement -- the Chevalley-Eilenberg algebra.

This perspective is thus summarized in a sequence of affine derived schemes
$$\mathrm{Crit}_X(\mathcal{S})\xrightarrow{\simeq}\mathbf{KTR}_{X,\mathcal{S}}\rightarrow \mathbb{R}\mathrm{Crit}_X(\mathcal{S})\rightarrow X,$$ together with a Lagrangian correspondence
\begin{equation}
    \label{eqn: BV, KTR, RCRIT Lag Corr}
    \begin{tikzcd}
    & \arrow[dl] \mathbf{KTR}_{X,\mathcal{S}}=\mathbf{Spec}_{\mathbf{k}}\big(\mathcal{O}_{\mathrm{KTR}}\big) \arrow[dr] & 
    \\
    \mathbf{BV}_{X,\mathcal{S}} && \mathbb{R}\mathrm{Crit}_X(\mathcal{S})
    \end{tikzcd}
\end{equation}

It is desirable to achieve an understanding of the \BV\hspace{.5mm} construction in a way that is \emph{intrinsic} to the formalism of PDE theory via a symplectic reduction in global algebraic $\D$-geometry and we elaborate on this in Subsect. \ref{sssec: Intrinsic formulation via D-geometry}.

\subsection{Main Results.}
\label{ssec: Main Results}
The functor of points geometry of ${\scalemath{.90}{\mathrm{N}}{\scalemath{.84}{\mathrm{L}}}{\scalemath{.88}{\mathrm{P}}{\scalemath{.84}{\mathrm{DE}}{\scalemath{.70}{\mathrm{S}}}}}}$ provides a natural setting to not only formalize non-perturbative physics but to also give a very general algebro-geometric formulation of the problem of quantization. 

This involves considering the configuration space of points $Ran_X$ defined as the homotopy-colimit of all finite powers of $X$ along diagonal maps
$$Ran_X\simeq colim\big[X\xrightarrow{\Delta}X^2\triplerightarrow X^3\cdots\big]\simeq \underset{I\in fSet^{op}}{colim} X^I.$$
It gives a geometric way to encode formal power series of elements in $\mathcal{O}(\mathbb{R}Sol)$ as well as a convenient setting for the geometry of multi-jets (Subsect. \ref{ssec: Applications to Field Theory}).
\vspace{1.5mm}

The parameterized geometry over $Ran_X$ -- the so-called derived $\D$-geometry of factorization spaces -- allows a convenient formalization of a general quantization problem: if $\EQ$ is a derived $\D_X$-space one may associate to it a derived $\D_{Ran_X}$-space $\EQ^{Ran}$ with an additional factorization property. Then, under mild hypothesis (Sect. \ref{sec: Variational Factorization Homology of PDEs}) its algebra of functions is a commutative factorization algebra in $\D$-modules. The aim is to deform this commutative structure to a non-commutative one, which in most cases, is not possible in the non-derived setting.  
\vspace{1.5mm}

In the geometric theory of ${\scalemath{.90}{\mathrm{N}}{\scalemath{.84}{\mathrm{L}}}{\scalemath{.88}{\mathrm{P}}{\scalemath{.84}{\mathrm{DE}}{\scalemath{.70}{\mathrm{S}}}}}}$ and Secondary Calculus, the interesting objects are sheaves, bundles or spaces living over $\EQ$, so the most natural thing to do is to work directly with the stable tangent $(\infty,1)$-category of differential factorization modules of horizontal multi-jets over $\EQ^{Ran}$ (Subsect. \ref{sssec: Horizontal Jets}). This will have the structure of a factorization category, whose deformations are called categorical (factorization) quantizations \cite{DefQuant}. This gives an optimal setting to approach a general quantization problem in the geometric theory of ${\scalemath{.90}{\mathrm{N}}{\scalemath{.84}{\mathrm{L}}}{\scalemath{.88}{\mathrm{P}}{\scalemath{.84}{\mathrm{DE}}{\scalemath{.70}{\mathrm{S}}}}}}$ via global and categorified techniques coming from derived algebraic geometry.
\vspace{1mm}

We give a dictionary summarizing the algebro-geometric ($\mathcal{D}$-geometry) analogs of objects coming from the differential-geometric (Diffiety geometry) approach for the benefit of the reader (see \cite{KRY}):

\begin{table}[ht]
\centering
\adjustbox{scale=.72}{
\begin{tabular}{|c|c|}
\hline
\textbf{$\mathcal{D}_X$-geometry} & \textbf{Diffiety geometry} \\
\hline

 Affine $\mathcal{D}_X$-scheme $\mathrm{Spec}_{\mathcal{D}_X}(\mathcal{A})$ & (Elementary) Diffiety \\ 
Universal (left) jet-$\mathcal{D}_X$-algebra $\mathcal{A}$, & Functions on the `empy equation' $\mathrm{Jet}^{\infty}(E)$  \\

Volume twisted algebra $\mathcal{A}^{r}$ & Space of Lagrangians $\overline{\Omega}^n(\pi)$ \\
 Central cohomology $h(\mathcal{A})$ & Secondary functions $\overline{H}^n(\pi)$ \\

 Vector Fields $\Theta_{\mathcal{A}}$ (\ref{ssec: Local Calculus}) & Evolutionary derivations $\varkappa(\pi)$\\ 
 Kahler $1$-forms $\Omega_{\mathcal{A}}^1$ (\ref{ssec: Local Calculus}) & Cartan $1$-forms $\mathcal{C}\Omega^1(\pi)$
 \\

 Local multivectors (\ref{ex: Hamiltonian Operators}) & Variational multivectors\\

$\D$-module de Rham algebra $\mathsf{DR}^{\mathcal{D}}(\Omega_{\mathcal{A}}^*)$ (\ref{Variational de Rham complex definition}) & Variational bi-complex \\

 $\mathrm{Lie}^{*}\mathcal{A}$-Algebroids   & Variational Lie Algebroids \\\hline
\end{tabular}}
\end{table}
 
\newpage 

The first result is stated in paraphrased form and describes the homotopical analog of the bi-complex and the coherency data that must be satisfied in order for a differential form to be closed. 
\vspace{2mm}

\noindent\textbf{Theorem A.}\hspace{1mm}
\emph{Let $X$ be a complex analytic manifold of dimension $d_X.$}
\vspace{1mm}

\noindent (i) (Thm. \ref{prop: DRVar is a prestack}, Prop. \ref{Theorem A}): \emph{For a derived $\D_X$-algebra $\mathcal{A}^{\bullet}=(\mathcal{A},\delta),$ there exists a $\mathcal{D}$-prestack of variational ($k$-forms) of degree $n$ defined by a graded mixed object in bi-complexes $\mathcal{V}ar(\mathcal{A}^{\bullet},n).$ 
Furthermore, if $\mathbb{R}\mathbf{Sol}_{\mathcal{D}_X}(\mathcal{A}^{\bullet})$ is homotopically finitely 
$\D$-presented then $\mathcal{V}ar^k$ is equivalent to}
$$\mathsf{Maps}\big(\mathcal{O}_{\mathbb{R}\mathbf{Sol}_{\mathcal{D}_X}(\mathcal{A}^{\bullet})},\mathsf{DR}_{var}(\mathbb{L}_{\mathbb{R}\mathbf{Sol}_{\mathcal{D}_X}(\mathcal{A}^{\bullet})})[n](k)\big).$$
\vspace{.5mm}

\noindent \noindent (ii) (Prop. \ref{prop: Cohom degree N (p,q) form key}):
\emph{Via $\prod$-totalization there is a decomposition by complexes\footnote{Remark that in the classical setting, we only have a vector space worth of such objects.} of closed $(p,q)$-forms
$$\big(\prod_{r\geq p}\prod_{s\geq q}\Omega_{X}^{*+r}\otimes\Bigwedge^{s+*}\mathbb{L}_{\EQ/X}^{\bullet}[-r-s]\big)[p+q],\hspace{1mm} q\geq0,p=1,\ldots, d_X,$$}
\emph{whose elements are formal sums}
$$\vartheta=\big(\theta_{n}^{(p,q)} | \theta_{n-1}^{(p+1,q)}+\theta_{n-1}^{(p,q+1)}|\cdots| \theta_{n-r}^{(p+r,q)}+\ldots+\theta_{n-r}^{(p,q+r)}|\cdots\big).$$
\vspace{1mm}
\noindent (iii) (Prop. \ref{prop: Closed degree N (p,q) form}):
\emph{With respect to the total differential a closed $(p,q)$-form 
 $\vartheta$ is a cocycle satisfying}:
$$D^{Tot}(\vartheta)=0\Leftrightarrow\begin{cases}
    \delta\theta_n^{(p,q)}=0
    \\
    d^v\theta_{n}^{(p,q)}+\delta\theta_{n-1}^{(p,q+1)}=0
    \\
    d_h\theta_n^{(p,q)}+\delta\theta_{n-1}^{(p+1,q)}=0
    \\
    d^v\theta_{n-1}^{(p,q+1)}+\delta\theta_{n-2}^{(p,q+2)}=0
    \\
    d_h\theta_{n-1}^{(p,q+1)}+d^v\theta_{n-1}^{(p+1,q)}+\delta\theta_{n-2}^{(p+1,q+1)}=0
    \\
    d_h\theta_{n-1}^{(p+1,q)}+\delta\theta_{n-2}^{(p+2,q)}=0,
    \\
    \cdots
    \\
\end{cases}
$$
\vspace{1mm}

We would also like to mention that it is possible to state Theorem A in the setting where $X$ is a super-manifold of dimension $d|m,$ by replacing dualizing sheaves with Berezinians and the de Rham complex $\Omega_X^*$ with the complex of integral forms $\mathcal{I}_{X,*}.$ One would like to do this, for instance, in the context of treating super-symmetric field theories.

\subsection{Applications to Field Theory.}
\label{ssec: Applications to Field Theory} 
When $\mathcal{A}$ is the jet algebra on sections of a (graded or super) field bundle $\pi:E\rightarrow X$, certain cohomology classes (\ref{eqn: Variational cohomologies}) in $h^*\big(\mathsf{DR}(\mathcal{A})\big)$ of its de Rham algebra define special functionals on sections $\underline{Sect}(X,E)\subset \underline{\mathrm{Maps}}(X,E).$
 There is a natural integration pairing
 \begin{equation}
     \label{eqn: Integration pairing}
 H_{*,c}(X)\times h^{*-d_X}\big(\mathsf{DR}(\mathcal{A})\big)\rightarrow \mathrm{Hom}\big(\underline{\Gamma}(X,E),\underline{k}\big),\hspace{2mm} (K,\omega)\mapsto \mathcal{S}_{(K,\omega)},
 \end{equation}
for $K$ an element of the singular homology cycles of $X$ with compact support, where $\mathcal{S}_{(K,\omega)}:s\mapsto \int_K (j_{\infty}s)^*\omega,$ for $j_{\infty}s:X\rightarrow \mathrm{Jet}^{\infty}(E)$, the Taylor series of a section $s,$ and $k$ is a field, often taken to be $\mathbb{R},\mathbb{C}$ or even $\mathbb{R}[\![\hbar]\!],$ with $\hbar$ a formal parameter.
\vspace{2mm}

Classes in $h^{*-p}\big(\mathsf{DR}(\mathcal{A})\big),p\leq d_X$ correspond to higher conservation laws, gauge charges etc. Indeed, elements of $h^0(\mathcal{A})$ represent functionals, while conserved currents are locally $(d_X-1)$-forms $J=J^i(x^{\mu},u^{\alpha},u_{\sigma}^{\alpha})d^{d_X-1}x_i,$ with coefficients in $\mathcal{A}$ such that $\mathcal{D}_X\bullet J^i=0$ which is the $\mathcal{D}_X$-algebra action by total derivatives. Therefore elements in $h^1(\mathcal{A})$ are \emph{conservation laws}. Similarly,  $\mathcal{Q}\in h^2$ describe gauge charges. From the perspective of the variational bicomplex they live on the first page of the $\mathcal{C}$-spectral sequence (shown below). Importantly, on this page we have the term (in blue) sending a system of PDEs to associated (Helmholtz) conditions for variationality.

\begin{wrapfigure}[18]{l}{8.5cm}
    \centering
    \adjustbox{scale=.82}{
\begin{tikzcd}
\label{FirstPage}
 0 & 0 & 0 
\\
 \overline{H}^n(\EQ)  \arrow[r,"\delta"] \arrow[u,] & {\color{white!25!blue}\mathcal{C}E_1^{n,1}(\EQ)} \arrow[u,] \arrow[r,]&  {\color{white!25!blue}\mathcal{C}E_1^{n,2}(\EQ)}\arrow[r] \arrow[u]& \cdots
 \\
 \overline{H}^{n-1}(\EQ)\arrow[u,] \arrow[r,"\delta"] & \mathcal{C}E_1^{n-1,1}(\EQ)\arrow[u,] \arrow[r,]& \mathcal{C}E_1^{n-1,2}\arrow[u](\EQ)\arrow[r]& \cdots
 \\
 \vdots \arrow[u,] & \vdots \arrow[u,] & \vdots\arrow[u]
 \\
 \mathcal{O}(\EQ)\arrow[u,]\arrow[r,] & \Omega^{0,1}(\EQ)\arrow[u,]\arrow[r,]& \Omega^{0,2}(\EQ)\arrow[r]\arrow[u] & \cdots
 \\
 0 \arrow[u] & 0\arrow[u] & 0\arrow[u] & 
\end{tikzcd}}
\end{wrapfigure}

Physicists are typically interested in (exponentiated) \emph{integral functionals} and $N$-point correlation functions e.g. 
$\mathcal{S}(\varphi)=\int_{X^N}g(x_1,\ldots,x_N)\prod_{i=1}^{N}\varphi(x_i)dx_i.$ 

One faces an immediate obstruction to defining these \emph{multilocal}-functionals via cohomology classes in $\overline{H}^n(\pi):=\overline{H}^n(\mathrm{Jet}^{\infty}(E))$.
This is exemplified by attempting to define a point-wise algebra structure; via (\ref{eqn: Integration pairing}) a class $\mathcal{S}=\int f(x,\Phi,\Phi_{\sigma})d\text{vol}_X(x)$ in $\overline{H}^n(\pi)$ determines a transformation on sections by
$\mathcal{S}(s):=\int_X j_{\infty}(s)^*f(x,u_{\sigma}^{\alpha})d\text{vol}_X(x)\in k,$ and its product with another functional $\mathcal{G}$ is
 $$\big(\mathcal{S}\cdot \mathcal{G}\big)(s)=\int_X\int_X f(x,u_{\sigma}^{\alpha})g(x',u_{\sigma'}^{\alpha'})d\text{vol}_X(x)\wedge d\text{vol}_X(x').$$
Such products correspond to classes in $\overline{H}^{2n}(\pi\times \pi),$ so are integral functionals for $\pi\times \pi,$ but \emph{not} for $\pi$ itself due to the strict inclusion of $\overline{H}^{2n}(\pi\times \pi)$ in $\overline{H}^n(\pi)\otimes\overline{H}^n(\pi).$ More generally, composites
$\mathcal{S}_1\cdot\ldots\cdot \mathcal{S}_k$ are not integral for $\pi$ but are integral for $\pi\times\ldots\times \pi,$ prompting us (see Subsect. \ref{ssec: HoPDE and Quantization}) to work with multi-jets and multi-local functions. 
Namely, consider a finite set $I$ and the sub-manifold
\begin{equation}
    \label{eqn: Open subset of XI}
X^{(I)}:=\big\{(x_1,\ldots,x_N):=(x^I)\in X^I|x_i\neq x_j, 1\leq i< j\leq N=|I|\big\}\hookrightarrow X^I.
\end{equation}
Associated to the usual $k$-jet bundle $\pi_k:\mathrm{Jet}^k(E)\rightarrow X,$ we may take
\begin{equation}
    \label{eqn: k multi-jet bundle}
    \pi_k^{(I)}:\mathrm{Jet}^k(E)^I:=\mathrm{Jet}^k(E)\times\cdots\times \mathrm{Jet}^k(E)\rightarrow X^I,
    \end{equation}
which defines a sub-manifold
\begin{equation}
    \label{eqn: I fold k-jet bundle}
\mathrm{Jet}^k(E)^{(I)}:=\big(\pi_k^{(I)}\big)^{-1}\big(X^{(I)}\big)\subset \mathrm{Jet}^k(E)^I,
\end{equation}
called the $I$-fold $k$-jet bundle of $\pi.$ A finite order multi-jet bundle is any $I$-fold $k$-jet bundle, for any finite set $I$ and $k\geq 1,$ while the infinite multi-jet bundle is an $I$-fold $k$-jet bundle for $k\rightarrow \infty.$
There are natural analogs of jet prolongations maps
\begin{eqnarray*}
j_k^{(I)}:Sect(X,E)&\rightarrow& Sect\big(X^I,\mathrm{Jet}^k(E)^{(I)}\big), \\
\varphi&\mapsto&\big[(x_1,\ldots,x_N)\mapsto (j_k(\varphi)(x_1),\ldots,j_k(\varphi)(x_N))\big].
\end{eqnarray*}
Unambiguous definitions of \emph{multi-local Lagrangians} and \emph{multi-local functionals} are given by functions
 $L^{(I)}:\mathrm{Jet}^{\infty}(E)^{(I)}\rightarrow \mathbb{R}^I,$ and expressions $\mathcal{S}^{(I)},$ such that 
$$\mathcal{S}^{(I)}\big(x_1^{\mu},\ldots,x_N^{\mu}\big)=\int_{K(I)\subset X^{(I)}}L^{(I)}\circ j_{\infty}^{(I)}\varphi(x_1^{\mu},\ldots,x_N^{\mu}),$$
for $K(I)$ a suitable compact domain, respectively.

Our needs necessitate the use of an $(\infty,1)$-categorical analog of the multi-jet construction for derived factorization $\D$-spaces of ${\scalemath{.90}{\mathrm{N}}{\scalemath{.84}{\mathrm{L}}}{\scalemath{.88}{\mathrm{P}}{\scalemath{.84}{\mathrm{DE}}{\scalemath{.70}{\mathrm{S}}}}}}$.
\vspace{2mm}

\noindent\textbf{Theorem B.}\hspace{1mm} (Theorem \ref{theorem: Theorem B Main claim}, Proposition \ref{proposition: Thm B (ii)} and Corollary \ref{Thm B corollary}).\hspace{1.5mm}
\emph{Let $\EQ$ be an affine $\D$-finitary prestack. Then 
there is a natural morphism}
$$\int_X^{fact}\mathcal{A}^{\ell}\rightarrow \Gamma\big(\RS(\EQ),\mathcal{O}_{\RS(\EQ)}\big),$$
\emph{which is an equivalence of monoids if $\mathcal{A}^{\ell}$ is cohomologically finite \emph{(\ref{ssec: Finiteness})}. If $\EQ$ is the derived $\D$-prestack of solutions $\mathbb{R}\underline{\mathsf{Sol}}_{\mathcal{D}}(\mathcal{I}),$ of a PDE with corresponding multi-jet $\D$-space $\EQ^{Ran}$, then its factorization homology identifies with $\mathbb{R}\underline{\mathsf{Sol}}(\mathcal{I})$ as a usual derived stack.}
\vspace{2mm}

Theorem B can be viewed as a homotopy-theoretic realization of the central idea of Secondary Calculus. In particular, it opens the door to studying other Secondary notions i.e. differential forms etc. via factorization homology with twisted coefficients, but for the lack of space, we defer this investigation to a future work.

We highlight Theorem B with an application to the \BV-formalism in Proposition \ref{Derived Quotient as BV Theorem} in Sect. \ref{sec: Applications 1}.
\vspace{1.5mm}

\noindent\textbf{Acknowledgements.}
The authors warmly thank Nicolai Reshetikhin for giving an interesting and inspiring talk at BIMSA out of which some parts of this work, specifically the inclusion of non-empty boundary, have been motivated. We thank him for interesting discussions.

The first author is supported by the Postdoctoral International Exchange Program of Beijing Municipal Human Resources and Social Security Bureau. The second author is supported by a grant from Beijing Institute of Mathematical Sciences and
Applications (BIMSA), and by the National Program of Overseas High Level Talent.

\subsection{Notations and Conventions.}
\label{Notations and Conventions}
Throughout this paper we treat non-linear PDEs modulo their symmetries in a coordinate free way and we’ll mainly work with $(\infty,1)$-categories to also avoid model-dependent arguments. There are a lot of technicalities which must be stated in order to make subtle arguments and wherever possible, we omit them for simplicity.
\begin{itemize}
    \item $\mathsf{Spc}$ is the symmetric monoidal $(\infty,1)$-category of spaces and for a field $k$ of characteristic zero $\mathsf{Vect}_k$ is the stable symmetric monoidal $(\infty,1)$-category of unbounded cochain complexes.
    \item  $\PS_k$ is the category of prestacks i.e. functors $\mathsf{Fun}\big(\mathsf{CAlg}_k,\mathsf{Spc}),$ where $\mathsf{CAlg}_k:=\mathsf{CAlg}(\mathsf{Vect}_k)$ is the $(\infty,1)$-category of commutative algebras in $\mathsf{Vect}_k.$

    \item For any $(\infty,1)$-category $\mathsf{C}$ and an operad $\mathcal{O},$ we denote by 
$\mathcal{O}-\mathsf{Alg}[\mathsf{C}],$
the category of $\mathcal{O}$-algebras in $\mathsf{C},$ e.g. $\mathsf{CAlg}_{k}:=\mathcal{C}omm-\mathsf{Alg}[\mathsf{Vect}_k].$
\end{itemize}
Via Dold-Kan and weighted negative-cyclic cohomology $NC^w(A/k):=NC^w(\mathbf{DR}(A/k)),$ one defines an étale local space of $n$-shifted closed $p$ forms
$\mathcal{A}_k^{p,cl}(A;n):=\big|NC^w(A/k)[n-p](p)\big|.$ Thus, given a derived stack $X$
there is an equivalence
$$\mathcal{A}_k^p(X;n)\simeq \underset{Spec(A)\in (\mathsf{dStk}_{/X})^{\op}}{holim}\hspace{1mm}\mathcal{A}^p(A;n).$$
The condition of being closed tells us that we have a semi-infinite sequence $\omega=\omega_0+\omega_1+\ldots$ with 
$\omega_i\in \mathbf{DR}(\mathcal{X})(p+i)[n+p]\cong \mathbb{R}\Gamma\big(\Bigwedge^{p+i}\mathbb{L}_{\mathcal{X}}[n-i]\big)$ such that $\delta\omega_0=0$ and such that $d\omega_i=\delta\omega_{i+1}.$
It is convenient to introduce the total complex
$$\underline{\mathbf{DR}}(\mathcal{X}):=\bigg(\prod_{p\geq 0}\mathbb{R}\Gamma\big(\mathrm{Sym}_{\mathcal{O}_{\mathcal{X}}}^p(\mathbb{L}_{\mathcal{X}}[-1])\big),D:=\delta+d\bigg),$$
with $$\underline{\mathbf{DR}(\mathcal{X})}_{\geq p}[n]\simeq \underbrace{\mathbb{R}\Gamma\bigg(\prod_{i\geq 0}\big(\Bigwedge_{\mathcal{O}_{\mathcal{X}}}^{p+i}\mathbb{L}_{\mathcal{X}}\big)[n]\bigg)}_{\text{total degree }n+p+i}.$$
Being closed is equivalently saying $\omega$ is a degree $n+p$
cocycle for the differential $D.$
\vspace{1mm}

\noindent\textit{\textcolor{blue}{The role of factorization:}} We exploit the natural \emph{factorization} structure on objects (see \cite{BeiDri},\cite{FraGai}) essentially as a local to global tool that also captures important phenomena appearing in physics \cite{CosGwi02,CosGwi01}. Factorization objects can be introduced quite generally as in \cite{Butson2}, with examples including: factorization algebra objects (e.g. in $\D$-modules), factorization categories, factorization functors etc. (see \cite{RaskinChiral}).

We mainly work with factorization $\D$-spaces, defined as a derived stack $Z^{Ran}$ over $Ran_{X_{dR}}$, i.e. a family $\{p_I:Z^I\rightarrow X_{dR}^I\},$ for all finite sets $I$, compatible with pull-backs along diagonals and compositions of surjections (up to coherent homotopy), together with homotopy coherent factorization equivalences,
$$j_{\alpha}^*Z^I\xrightarrow{\simeq}j_{\alpha}^*\prod_{j\in J} Z^{I_j},$$
for all surjections $\alpha:I\rightarrow J$ with $I_j:=\alpha^{-1}(j)$ and embedding $j_{\alpha}:U(\alpha):=\{x_i\neq x_j| \alpha(i)\neq \alpha(j)\}\rightarrow X^I.$

\begin{remark}[Notation/Terminology]
Any decoration of an object by a super-script `$Ran$' e.g. $Z^{Ran}$ above, should be taken as a reference to a compatible family of objects over each $X_{dR}^I\in fSet.$ Roughly speaking, for our purposes, this may be taken as synonymous with `factorization,' as we do not consider objects over $Ran_X$ unless to make use of factorization structures. It will therefore be omitted wherever possible from our notation.
\end{remark}
The corresponding $(\infty,1)$-category is denoted by $$\PS_{\mathcal{D}}^{fact}(X)\subseteq \PS_{/Ran_{X_{dR}}}.$$

Such objects admit a notion of a quasi-coherent sheaf of categories \cite{Gai1Aff}. We make free use of this structure, always denoting a quasi-coherent sheaf of categories over a space $Z$ by script-font $\EuScript{C}_Z.$ 
\vspace{2mm}

For the reader’s convenience, a summary of the most basic terms we use is given in the appendix.

\section{Geometry of $\D$-Prestacks}
\label{sec: Geometry of D-Prestacks}
The main tool for geometry on $\D$-spaces are the spaces of local operations, introduced in \cite{BeiDri}. Roughly speaking, they are defined for a finite family of dg-$\D_X^{op}$-modules $(\mathcal{M}_{i}^{\bullet})_{i\in I},\mathcal{N}^{\bullet},$ by
\begin{equation}
    \label{eqn: LocalOps}
\mathbb{R}\mathcal{P}_{I}^*\big(\{\mathcal{M}_i\},\mathcal{N}\big):=\mathbb{R}\mathcal{H}om_{\mathcal{D}_{X^I}}\big(\boxtimes_{i\in I}^{\mathbb{L}}\mathcal{M}_i^{\bullet},\Delta_*\mathcal{N}^{\bullet}\big),
\end{equation}
with $\Delta:X\hookrightarrow X^I$ and $\boxtimes_i(-)\simeq \bigotimes_ip_i^*(-),$ with $p_i:X^I\rightarrow X$ natural projections.
\vspace{1.5mm}

Operations (\ref{eqn: LocalOps}) can also be defined in categories of differential graded $\mathcal{A}\otimes_{\mathcal{O}_X}\mathcal{D}_X$-modules and are denoted by $\mathbb{R}\mathcal{P}_{\mathcal{A},I}^*$. 
The non-trivial point is that extending them to varieties with $dim(X)>1,$ must be done in a homotopy-coherent setting via $(\infty,1)$-categories (see \cite{FraGai,KRY}\footnote{
 As pointed out in \cite{FraGai}, categories of chiral and factorization $\D$-modules, even in the simplest case of a curve $X,$ lack co-products therefore cannot admit a model
category structure.}), rather than by just using abelian and triangulated categories as in \cite{BeiDri}. 
\vspace{1.5mm}

The appropriate homotopy-coherent space of maps consist of factorization compatible morphisms (Subsect. \ref{sssec: Internal Factorization Hom-Algebra $!$-Sheaves}) of factorization $\mathcal{F}(\mathcal{A})\otimes_{\mathcal{O}_{Ran_X}}\mathcal{D}_{Ran_X}$-modules, where $\mathcal{F}(\mathcal{A})$ denotes the commutative factorization algebra in $\D$-modules canonically associated to a commutative $\D$-algebra (see Proposition \ref{Multi-jet r-adjoint proposition}). One obtains a bi-functorial assignment
\begin{equation}
    \label{eqn: Holocals}
    \mathbb{R}\mathsf{Maps}_{\mathcal{A}}^*:\big(\mathcal{F}(\mathcal{A}^{\ell})-\mathsf{Mod}^{fact}(\D_{Ran_X})\big)^{op}\times \mathcal{F}(\mathcal{A}^{\ell})-\mathsf{Mod}^{fact}(\D_{Ran_X})\rightarrow \mathsf{Spc}.
\end{equation}
Compatibility with (\ref{eqn: LocalOps}) follows since, for a family of $\D$-modules on $Ran_X$ supported on the main diagonal $\{\mathcal{M}_i\}_{i=1,\ldots,n},\mathcal{N},$ and the tautological factorization algebra $\mathcal{O}_{Ran_X}$ in $\D_{Ran_X}$-modules, the corresponding space of maps
$\mathbb{R}\mathsf{Maps}_{\mathcal{O}_{Ran_X}}^*$ arising from (\ref{eqn: Holocals}) is equivalent to
$$\mathbb{R}\mathcal{H}om_{\mathcal{D}_{Ran_X}}\big(\mathcal{M}_1\otimes^*\cdots\otimes^*\mathcal{M}_n,\mathcal{N}\big)\simeq \mathbb{R}\mathcal{P}_{n}^*(\{\mathcal{M}_i\},\mathcal{N}),$$
with $\otimes^*$ the (local) tensor product on the $(\infty,1)$-category $\mathsf{Mod}(\D(Ran_X)),$ roughly defined by 
$$\mathcal{M}\otimes^*\mathcal{N}=\cup_*(\mathcal{M}\boxtimes\mathcal{N}),$$ with $\cup:Ran_X\times Ran_X\rightarrow Ran_X$ the operation of union of finite sets. We refer to \cite[Sect. 2.2]{FraGai} for details.
\vspace{1.5mm}

Local operations allow one to work with geometric objects on $\D$-spaces of solutions as well as formulate necessary finiteness conditions and duality statements for them.
The key point is that they induce usual operations in de Rham cohomology and therefore give an algebraic presentation of local functions, differential forms and vector fields on spaces of sections $\underline{Sect}(X,E)$ (Proposition \ref{prop: Integration pairing}). 
\vspace{2mm}

\noindent\textit{\textbf{Basic notions.}} Standard equivalences of categories of left and right $\D$-modules still hold for commutative monoid objects:  
$$(-)^r:\mathrm{CAlg}(\mathcal{D}_X)\simeq\mathrm{CAlg}(\mathcal{D}_X^{\mathrm{op}}):(-)^{\ell},$$ given by $\mathcal{A}^r:=\omega_X\otimes_{\mathcal{O}_X}\mathcal{A}$ and $\mathcal{A}^{\ell}:=\omega_X^{\otimes -1}\otimes_{\mathcal{O}_X}\mathcal{A}$.

In what follows, unless otherwise specified, a `$\mathcal{D}$-algebra' means a left one, while a $\mathcal{D}^{\mathrm{op}}$-algebra means a right $\mathcal{D}$-algebra.

If $\mathcal{A}$ is the infinite jet algebra on sections of some bundle $E\rightarrow X,$ a $\D$-geometric non-linear PDE is algebraically encoded by a sequence 
\begin{equation}
    \label{eqn: SES}
    0\rightarrow \mathcal{I}_X\rightarrow \mathcal{A}\rightarrow \mathcal{B}\rightarrow 0,
\end{equation}
of $\mathcal{D}_X$-algebras with $\mathcal{I}_X$ some $\mathcal{D}_X$-ideal sheaf. A more geometric definition can be given.

\begin{definition}
\normalfont The category of \emph{affine $\mathcal{D}_X$-schemes} over $X$\footnote{We also call them $(X,\mathcal{D}_X)$-affine schemes.} is defined the opposite category $\mathrm{Aff}_X(\mathcal{D}_X):=\mathrm{CAlg}(\mathcal{D}_X)^{\mathrm{op}}.$ A non-linear PDE in the sense of (\ref{eqn: SES}) therefore corresponds to a $\mathcal{D}_X$-subscheme. 
\end{definition}
A generic affine $\mathcal{D}_X$-scheme is denoted by analogy with ordinary algebraic geometry as $\mathrm{Spec}_{\mathcal{D}_X}(\mathcal{A})$ and the functor of points philosophy gives a $\mathcal{D}$-Yoneda embedding
\begin{eqnarray}
    \label{eqn: D-Yoneda}
h_{\mathcal{A}}^{\mathcal{D}}:\mathrm{CAlg}(\mathcal{D})\rightarrow \mathrm{Set},\hspace{3mm}\mathcal{B}\mapsto \mathrm{Hom}_{\mathcal{D}\mathrm{-Alg}}(\mathcal{A},\mathcal{B}).
\end{eqnarray}
 In particular, the space of classical holomorphic solutions is given by
\begin{equation}
    \label{eqn: Classical D-Solutions}
Sol_{\mathcal{D}}(\mathcal{A}):=h_{\mathcal{A}}^{\mathcal{D}}(\mathcal{O}_X)=\mathrm{Hom}_{\mathcal{D}\mathrm{-Alg}}(\mathcal{A},\mathcal{O}_X).
\end{equation}
If $\mathcal{F}$ is another function space with a $\D_X$-algebra structure, denote the space of solutions of $\mathcal{A}$ in $\mathcal{F}$ by $Sol_{\mathcal{D}}(\mathcal{A},\mathcal{F}).$

\begin{definition}
\label{Definition: (Formal) Algebraic D-Spaces}
\normalfont
An \emph{algebraic $\D_X$-space} is a presheaf $\mathcal{X}:\big(\mathrm{CAlg}_{\mathcal{D}_X}^{\mathrm{op}},\tau\big)\rightarrow \mathrm{Set},$ that is a sheaf for a suitable topology $\tau$ (usually the Zariski or étale topologies). Given $\mathcal{R}^{\ell}\in \mathrm{CAlg}_X(\D_X),$ a \emph{$\D_X$-algebra $\mathcal{R}^{\ell}$ point} of an étale local $\D$-space $\EQ$ is a $\D$-space morphism $\mathrm{Spec}_{\mathcal{D}}(\mathcal{R})\rightarrow \EQ$.
\end{definition}

The main example arises via a given sequence (\ref{eqn: SES}). In this case, the space of $\mathcal{R}^{\ell}$-points is a functorial assignment
\begin{eqnarray}
    \label{eqn: Local solution D space}
    \underline{Sol}_{\mathcal{D}}(\mathcal{B}):\mathrm{CAlg}(\D_X)_{\mathcal{A}/}&\rightarrow& \mathrm{Set},
    \\
    \underline{Sol}_{\mathcal{D}}(\mathcal{B})(\mathcal{R}^{\ell})&:=&Sol_{\mathcal{D}}(\mathcal{B},\mathcal{R}^{\ell})=\big\{\psi\in \mathcal{R}^{\ell}| \mathsf{F}(\psi)=0,\forall \mathsf{F}\in\mathcal{I}\big\}.
\end{eqnarray}

An appropriate derived analog of (\ref{eqn: Local solution D space}) is given by 
\begin{equation}
    \label{eqn: Derived Local solution D space}
    \mathbb{R}Sol_{\mathcal{D}_X}(\mathcal{B},-):=\mathbb{R}\mathcal{M}aps_{dg-\D_X-alg}(Q\mathcal{B},-):Ho\big(\cdga_{\mathcal{A}/}\big)\rightarrow Ho(SSet),
    \end{equation}
with $Q\mathcal{B}$ the canonical cofibrant replacement \cite{KRY},\cite{KSYI}.
\subsection{Coefficients}
\label{ssec: DMods with DAlg Coeff}
Given a $\D_X$-algebra $\mathcal{A}$ we study the categories of $\mathcal{A}$-module objects in $\D_X$-modules and complexes of such objects, denoted respectively by 
$$\mathrm{Mod}_{\mathcal{D}_X}(\mathcal{A}):=\mathcal{A}-\mathrm{Mod}(\mathcal{D}_X),\hspace{2mm}\text{and }\hspace{1mm} \DG(\mathcal{A}).$$
Objects (dg or not) are just left $\D_X$-modules $\mathcal{M}$ with a $\D_X$-compatible $\mathcal{A}$-action, $\nu\colon \mathcal{A}\otimes\mathcal{M}\rightarrow \mathcal{M}.$ 

We endow $\DG(\mathcal{A})$ with an evident model structure whose weak-equivalences are quasi-isomorphisms. 
\begin{proposition}
\label{Tautological AD equivalence}
There are equivalences of categories $\mathrm{Mod}_{\mathcal{D}_X}(\mathcal{A})\simeq\mathrm{Mod}\big(\mathcal{A}\otimes_{\mathcal{O}_X}\mathcal{D}_X),$ and
$\mathrm{Mod}_{\mathcal{D}_X^{\mathrm{op}}}(\mathcal{A}^{r})\simeq \mathrm{Mod}_{\mathcal{D}_X^{\mathrm{op}}}(\mathcal{A})$.
\end{proposition}

The $\mathcal{A}$-module structure on $\mathcal{M}$ compatible with the $\D_X$-structure can be repackaged as an action
$\mathcal{A}\otimes_{\mathcal{O}_X}\mathcal{D}_X\otimes\mathcal{M}\rightarrow \mathcal{M},$  by the sheaf of differential operators on $X$ with coefficients in $\mathcal{A},$  $$\mathcal{A}[\mathcal{D}_X]:Opens_X\rightarrow \mathrm{Ring}, U\mapsto  \mathcal{A}[\mathcal{D}_X](U):=\mathcal{A}_{U}\otimes_{\mathcal{O}_X|_{U}}\mathcal{D}_U.$$

The associative unital structure 
$\circ:\mathcal{A}[\mathcal{D}_X]\times \mathcal{A}[\mathcal{D}_X]\rightarrow \mathcal{A}[\mathcal{D}_X],$ is determined by
$$(a\otimes 1_{\mathcal{O}_X})\circ (b\otimes P):=\mu(a,b)\otimes  P,\hspace{1mm}
    (1_{\mathcal{A}}\otimes \theta)\circ (b\otimes P):=\nabla_{\theta}(b)\otimes P+b\otimes(\theta\circ P),$$
for all $a,b\in\mathcal{A},\theta\in\Theta_X,P\in\mathcal{D}_X.$

Moreover, $\mathcal{A}\hookrightarrow \mathcal{A}[\mathcal{D}_X]$ is a unital sub-algebra via $a\mapsto a\otimes 1_{\mathcal{O}_X}$ while $\Theta_X\hookrightarrow \mathcal{A}[\mathcal{D}_X],$ is a Lie sub-algebra via $\theta \mapsto 1_{\mathcal{A}}\otimes \theta.$ 

There is a natural double filtration on $\mathcal{A}[\mathcal{D}_X]$ combining that on $\mathcal{A}$ and the order filtration on $\mathcal{D}_X$. If $\mathcal{A}^{\ell}$ is an infinite-jet algebra the filtration is given for each $k\in \mathbb{N},$ by the space freely generated by $\mathrm{Fl}^k\mathcal{D}_{X}\otimes_{\mathcal{O}_X}\mathcal{O}_E$ i.e. filtration by finite jets.

It is convenient to introduce a total filtration associated to the double filtration,
$$\mathrm{Fl}_{\mathrm{Tot}}^p\big(\mathcal{A}[\mathcal{D}_X](V)\big):=\bigcup_{p=k+\ell}\mathrm{Fl}^k\mathcal{A}(V)\otimes_{\mathcal{O}_X(V)}\mathrm{Fl}^{\ell}\mathcal{D}_{X}(V).$$
Put $\mathrm{Fl}_{\mathrm{Tot}}\mathcal{A}[\mathcal{D}_X]:=\lbrace{\mathrm{Fl}_{k,\ell}\mathcal{A}[\mathcal{D}_X]\rbrace}_{k,\ell\in\mathbb{Z}\times\mathbb{Z}},$ and interpret 
sections of $\mathrm{Fl}^{k,\ell}\mathcal{A}[\mathcal{D}_X]$ as differential operators of order $\leq \ell$ with coefficients in $k$-jets.

Given a morphism $F:\mathcal{A}\rightarrow \mathcal{B}$ of commutative $\mathcal{D}_X$-algebras, we have usual adjoint functors
\begin{equation}
\scalemath{.93}{F^*\colon\mathrm{Mod}\big(\mathcal{B}[\mathcal{D}_X]\big)\rightleftarrows \mathrm{Mod}\big(\mathcal{A}[\mathcal{D}_X]\big)\colon \mathrm{coind}_{F},\hspace{2mm} 
\mathrm{ind}_F\colon\mathrm{Mod}\big(\mathcal{A}[\mathcal{D}_X]\big)\rightleftarrows \mathrm{Mod}\big(\mathcal{B}[\mathcal{D}_X]\big)\colon F^*.}
\end{equation}

The categories $\mathrm{Mod}(\mathcal{A}[\mathcal{D}_X])$ and $\mathrm{Mod}(\mathcal{A}[\mathcal{D}_X^{\mathrm{op}}])$ are symmetric monoidal categories for which $\mathcal{A}$ and $\mathcal{A}^r$ are the respective monoidal units. 
For instance, we have a functor
$$-\otimes_{\mathcal{A}}-:\mathrm{Mod}(\mathcal{A}[\mathcal{D}_X])^{\mathrm{op}}\times\mathrm{Mod}(\mathcal{A}[\mathcal{D}_X])\rightarrow \mathrm{Mod}(\mathcal{A}[\mathcal{D}_X]),$$
on $\mathcal{A}[\mathcal{D}_X]$-modules, whose opposite monoidal structure in the category of right $\mathcal{A}[\mathcal{D}_X]$-modules is given by 
$\mathcal{M}\otimes_{\mathcal{A}}\mathcal{N}:=\big(\mathcal{M}^{\ell}\otimes_{\mathcal{A}}\mathcal{N}^{\ell}\big)^{r},$
with monoidal unit
$\mathcal{A}^r.$ 
It is convenient for later use to define two induction functors in left and right $\mathcal{A}[\mathcal{D}_X]$-modules, $
\mathrm{ind}_{\mathcal{A}[\mathcal{D}_X]}:\mathrm{Mod}(\mathcal{O}_X)\rightarrow \mathrm{Mod}(\mathcal{A}[\mathcal{D}_X]),\mathrm{ind}_{\mathcal{A}[\mathcal{D}_X]}(\mathcal{F}):=\mathcal{A}[\mathcal{D}_X]\otimes_{\mathcal{O}_X}\mathcal{F},$ and similarly 
$\mathrm{ind}_{\mathcal{A}[\mathcal{D}_X]}^r(\mathcal{F}):=\mathcal{F}\otimes_{\mathcal{O}_X} \mathcal{A}[\mathcal{D}_X].$
\begin{proposition}
\label{prop: CDiffs}
If $\mathcal{M}\cong \mathrm{ind}_{\mathcal{A}[\mathcal{D}_X]}(\mathcal{F})$ and $\mathcal{N}\cong \mathrm{ind}_{\mathcal{A}[\mathcal{D}_X]}(\mathcal{G})$ are induced $\mathcal{A}[\mathcal{D}_X]$-modules of finite type i.e. $\mathcal{F},\mathcal{G}$ are finitely generated, the space of internal $\mathcal{A}[\mathcal{D}_X]$-homomorphisms identifies with differential operators in total derivatives 
$$\mathrm{Hom}_{\mathcal{A}[\mathcal{D}_X]}\big(\mathrm{ind}_{\mathcal{A}[\mathcal{D}_X]}(\mathcal{F}),\mathrm{ind}_{\mathcal{A}[\mathcal{D}_X]}(\mathcal{G})\big)\simeq \mathrm{Diff}_{\mathcal{A}}\big(\mathcal{A}\otimes\mathcal{F},\mathcal{A}\otimes\mathcal{G}\big).$$
\end{proposition}
These notions sheafify over $X$ that we denote $\mathcal{H}om_{\mathcal{A}}$ and $\mathcal{H}om_{\mathcal{A}[\mathcal{D}_X]}.$

The duality operation comes from $\mathcal{A}[\mathcal{D}_X]$-linearity\footnote{Note this is \emph{not} $\mathcal{H}\mathrm{om}_{\mathcal{A}}(\mathcal{M},\mathcal{A}).$}; \emph{local/inner duality} is given by the functor
\begin{equation}
    \label{eqn: Local duality}
    (-)^{\circ}:\mathrm{Mod}(\mathcal{A}[\mathcal{D}_X])\rightarrow \mathrm{Mod}(\mathcal{A}[\mathcal{D}_X])^{\mathrm{op}},\hspace{3mm}
    \mathcal{M}^{\circ}:=\mathcal{H}\mathrm{om}_{\mathcal{A}[\mathcal{D}_X]}\big(\mathcal{M},\mathcal{A}[\mathcal{D}_X]\big),
    \end{equation}
but notice this produces a right $\mathcal{A}[\mathcal{D}_X]$-module i.e. an $\mathcal{A}[\mathcal{D}_X^{\op}]$-module.
\begin{definition}
\label{Defin: Inner Verdier Duality}
\normalfont The \emph{inner Verdier duality} is $D_{\mathcal{A}}^{ver}:=(-)^{\ell}\circ (-)^{\circ}.$ 
\end{definition}
Definition \ref{Defin: Inner Verdier Duality} is written by
\begin{equation}
    \label{eqn: Inner duality}
    D_{\mathcal{A}}^{ver}:\mathrm{Mod}\big(\mathcal{A}[\mathcal{D}_X]\big)\rightarrow \mathrm{Mod}\big(\mathcal{A}[\mathcal{D}_X]\big),\hspace{2mm} (\mathcal{M}^{\circ})^{\ell}:=\mathcal{H}\mathrm{om}_{\mathcal{A}[\mathcal{D}_X]}(\mathcal{M},\mathcal{A}[\mathcal{D}_X])^{\ell}.
\end{equation}
Notice that the sheaf $\mathcal{H}\mathrm{om}_{\mathcal{A}[\mathcal{D}_X]}(\mathcal{M},\mathcal{N}\otimes\mathcal{A}[\mathcal{D}_X])$ comes with a universal pairing
$$ev:\Delta^*\big(\mathcal{H}\mathrm{om}_{\mathcal{A}[\mathcal{D}_X]}(\mathcal{M},\mathcal{N}\otimes\mathcal{A}[\mathcal{D}_X])\otimes \mathcal{M}\rightarrow \Delta_*^{(2)}(\mathcal{N}\otimes\mathcal{A}[\mathcal{D}_X]),$$
where $\Delta:X\hookrightarrow X\times X$ is the diagonal embedding.
\vspace{1mm}

 The homotopical analog is called \emph{derived inner Verdier duality} and is a right-derived functor of (\ref{eqn: Inner duality}) defined on the underlying homotopy categories of the model category,
$$\mathbb{D}_{\mathcal{A}}^{ver}(-):Ho(\DG(\mathcal{A}))\rightarrow Ho(\DG(\mathcal{A})),$$
where
\begin{equation}
    \label{eqn: Derived Inner Verdier Dual}
\mathbb{D}_{\mathcal{A}}^{ver}(\mathcal{M}^{\bullet}):=\mathbb{R}\mathcal{H}om_{\mathcal{A}\otimes\mathcal{D}_X-dg}\big(\mathcal{M}^{\bullet},\mathcal{A}\otimes_{\mathcal{O}_X}\mathcal{D}_X)\otimes\omega_X^{\otimes -1}[dim_X].
\end{equation}
A useful derived functor of local solutions of $\mathcal{M}^{\bullet}$ is also defined 
$$\mathbb{R}Sol_{\mathcal{A}}(\mathcal{M},-):Ho(\DG(\mathcal{A}))^{op}\rightarrow Ho(\mathcal{A}^{\ell}-\mathbf{dg}_k),$$
with $\mathcal{A}^{\ell}-\mathbf{dg}_k$ the category of $\mathcal{A}^{\ell}$-modules in complexes of $k$-vector spaces (differential graded $\mathcal{A}^{\ell}$ $k$-vector spaces) by setting 
\begin{equation}
\label{eqn: Derived Local Solutions}
\mathbb{R}Sol_{\mathcal{A}}(\mathcal{M},\mathcal{N}):=\mathbb{R}\mathcal{H}om_{\mathcal{A}^r\otimes\mathcal{D}_X^{op}}(\mathbb{D}_{\mathcal{A}}^{ver}(\mathcal{M}),\mathcal{N}\otimes \mathcal{A}[\mathcal{D}_X]).
\end{equation}
For example, taking the initial $\D_X$-algebra $\mathcal{O}_X$ recovers the usual $\D$-module derived solution functors
$\mathbb{R}Sol_{\mathcal{O}_X}(\mathcal{M},\mathcal{N})=\mathbb{R}\mathcal{H}om_{\D_X^{op}}(\mathbb{D}(\mathcal{M}),\mathcal{N}),$ with derived dual $\D$-module $\mathbb{D}(\mathcal{M})=\mathbb{R}\mathcal{H}om_{\D_X}(\mathcal{M},\D_X).$

\begin{proposition}
Let $\mathcal{E}$ be a coherent $\mathcal{O}_X$-module and suppose that $\mathcal{A}$ is a $\D_X$-smooth commutative $\D$-algebra. Then $$\mathbb{D}_{\mathcal{A}}^{ver}\big(ind_{\mathcal{A}[\mathcal{D}]}^r(\mathcal{E})\big)\simeq \mathcal{A}\otimes \omega_X\otimes_{\mathcal{O}_X}\mathcal{E}^{\vee}\otimes_{\mathcal{O}_X}\mathcal{A}[\mathcal{D}_X][dim_X].$$
\end{proposition}
The main example of $\mathcal{A}^{\ell}[\mathcal{D}_X]$-modules or relevance in PDE geometry are obtained as pull-backs of vector bundles over $X.$ They are described in convenient terms via the language of horizontal jets.
\subsubsection{Example: Horizontal Jets as $\mathcal{D}$-Bundles.}
\label{sssec: Horizontal Jets}
If $\mathcal{N}$ is any $\D_X$-module and $\mathcal{A}\in \mathrm{CAlg}(\mathcal{D}_X),$ there is an equivalence 
$$\mathrm{Hom}_{\mathrm{CAlg}_{\mathcal{D}_X}}\big(Sym_{\mathcal{O}_X}^{\bullet}(\mathcal{N}),\mathcal{A}\big)\cong \mathrm{Hom}_{\mathcal{D}_X}(\mathcal{N},\mathcal{A}),$$
where on the right-hand side we have a vector-space of $\D$-module morphisms. Then, 
$$\mathbb{V}_{\mathcal{D}}(\mathcal{N}):=Spec_{\mathcal{D}_X}\big(Sym_{\mathcal{O}_X}^{*}(\mathcal{N})\big),$$ is a vector $\D_X$-scheme over $X.$
\begin{proposition}
The functor $\mathbb{V}_{\mathcal{D}}$ defines an equivalence of categories between left $\D_X$-modules and vector $\D_X$-schemes, with inverse sending $\EQ$ to the $\D_X$-module $s^*\Omega_{\EQ}^1,$ where $s:X\rightarrow \EQ$ is the zero-section morphism.
\end{proposition}

The category of such vector schemes $\mathrm{Vect}_{X}(\mathcal{D}_X)$ is fully-faithfully embedded into a functor category of those from $\mathrm{CAlg}_X(\mathcal{D}_X)$ taking values in sheaves of $k$-vector spaces on $X$. In other words, we may view a vector $\mathcal{D}$-scheme $\EQ$ as a contravariant functor, $\underline{\EQ}$ defined on $\mathcal{A}\in \mathrm{CAlg}_{\mathcal{D}_X}$ as the  étale sheaf over $X$, which evaluates on an open embedding $j:U\hookrightarrow X$ as:
$$\underline{\EQ}(A)(U):=\EQ|_{U}\big(\mathcal{A}|_U\big)=\EQ\big(j_*j^*\mathcal{A}\big),$$
since for open embeddings $j$, the $\mathcal{D}$-module pull-back coincides with the sheaf theoretic one.

In precisely the same way, we have a notion of a vector $\D_X$-scheme over an affine $\D_X$-scheme $Spec_{\mathcal{D}_X}(\mathcal{A}).$ 
In the dg-setting i.e. model categorically, there is for $\mathcal{A}\in \cdga,$ the standard free-forgetful adjunction
\begin{equation}
    \begin{tikzcd}
    \label{eqn: Free-forget AD-Mod/Alg Adjunction}
\mathrm{Free}_{\mathcal{A}}:\DG(\mathcal{A})\arrow[r,shift left=.5ex] & Comm\big(\DG(\mathcal{A})\big)\simeq \cdga_{\mathcal{A}/}:\arrow[l,shift left=.5ex]\mathrm{For}.
\end{tikzcd}
\end{equation}

Adjunction (\ref{eqn: Free-forget AD-Mod/Alg Adjunction}) induces a similar adjunction of $\infty$-categories, denoted
$$\mathsf{Free}_{\mathcal{A}}:\mathsf{Mod}_{\mathcal{D}_X}(\mathcal{A})\rightleftarrows \mathsf{CAlg}_{\mathcal{A}/}(\mathcal{D}_X):\mathsf{For}.$$
There is an assignment 
$$\mathbb{V}_{\mathcal{A}}(-):\DG(\mathcal{A})\rightarrow \mathbf{dStk}(\mathcal{D}_X)_{/Spec_{\mathcal{D}}(\mathcal{A})},$$
where $\mathbb{V}_{\mathcal{A}}(\mathcal{M})$ defined on $\gamma:Spec_{\mathcal{D}}(\mathcal{B})\rightarrow \EQ,$ by $\mathsf{Maps}_{\mathcal{B}-dg}(\mathcal{B},\gamma^*\mathcal{M}),$ with $\EQ=Spec_{\mathcal{D}}(\mathcal{A}).$
\begin{proposition}
\label{prop: Vector D Stack}
Suppose that $\mathcal{M}$ is a perfect $\mathcal{O}_{\EQ}[\mathcal{D}_X]$-module and is connective. Then the $\mathbb{V}_{\EQ}(\mathcal{M})$ is relatively representable over $\EQ$ i.e. $\mathbb{V}_{\EQ}(\mathcal{M})\simeq Spec_{\EQ,\mathcal{D}}(\mathsf{Free}_{\mathcal{A}}\mathcal{M}^{\vee}).$
\end{proposition}
\begin{proof}
Since $\mathcal{M}$ is perfect it is dualizable as an object of $\DG(\mathcal{O}_{\EQ})$ with dual $\mathcal{M}^{\vee}$ given by (\ref{eqn: Derived Inner Verdier Dual}).
Let $\gamma:Spec_{\mathcal{D}}(\mathcal{B})\rightarrow \EQ$ be an object of $\mathbf{dStk}_X(\mathcal{D}_X)_{/\EQ},$ and note by definition $\mathbb{V}_{\EQ}(\mathcal{M})(\gamma)$ is equivalent to
$$\mathsf{Maps}_{\mathcal{B}[\mathcal{D}]}(\mathcal{B},\gamma^*\mathcal{M})\simeq \mathsf{Maps}_{\mathcal{B}[\mathcal{D}]}(\gamma^*\mathcal{M}^{\vee},\mathcal{B}),$$
where we use the fact $\mathcal{M}$ is dualizable. Then, by (\ref{eqn: Free-forget AD-Mod/Alg Adjunction}) this is equivalent to
$$\mathsf{Maps}_{\cdga_{\mathcal{B}/}}(\mathsf{Free}_{\mathcal{B}}\gamma^*\mathcal{M}^{\vee},\mathcal{B})\simeq \mathsf{Maps}_{\cdga_{\mathcal{O}_{\EQ}/}}(\mathsf{Free}_{\EQ}(\mathcal{M}^{\vee}),\gamma_*\mathcal{O}_{Spec_{\mathcal{D}}(\mathcal{B})}\big).$$
\end{proof}

\subsubsection{}
Consider the affine $\D$-scheme corresponding to the jet construction and put $p_{\infty}:Spec_{\mathcal{D}_X}(\mathcal{A})\rightarrow X$. Denote $\Omega_X^{\bullet}$ the de Rham algebra on $X$ and consider 
$$\mathcal{A}^{\ell}[\mathcal{D}_X](p_{\infty}^*\mathcal{E},p_{\infty}^*\Omega_X^{\bullet}),$$
which we identify with the space of total differential operators acting from sections of the pull-back bundle $p_{\infty}^*\mathcal{E}$ to the horizontal de Rham algebra. In other words,
$$\mathcal{A}^{\ell}[\mathcal{D}_X](p_{\infty}^*\mathcal{E},p_{\infty}^*\Omega_X^{\bullet})\simeq \mathcal{H}om_{\mathcal{A}^{\ell}[\mathcal{D}_X]}\big(ind_{\mathcal{A}^{\ell}[\mathcal{D}_X]}(p_{\infty}^*\mathcal{E}),ind_{\mathcal{A}^{\ell}[\mathcal{D}_X])}p_{\infty}^*\Omega_X^{\bullet}\big).$$
Consider
$$Sym_{\mathcal{A}}^p\big(\mathcal{A}^{\ell}[\mathcal{D}_X](p_{\infty}^*\mathcal{E},p_{\infty}^*\Omega_X^{\bullet})\big),$$
which is the space of symmetric $p$-multilinear differential operators in total derivatives acting from sections 
$$\Gamma(Spec_{\mathcal{D}}(\mathcal{A}),p_{\infty}^*\mathcal{E})\rightarrow \Gamma(Spec_{\mathcal{D}}(\mathcal{A}),p_{\infty}^*\Omega_X^{\bullet}).$$

Moreover, we can consider $Alt_{\mathcal{A}}^p\big(\mathcal{A}^{\ell}[\mathcal{D}_X](p_{\infty}^*\mathcal{E},p_{\infty}^*\Omega_{X}^{\bullet})\big).$
Suppose we have a formally exact sequence of sheaves (see \cite{VinCohomPDE}) of $\mathcal{O}_X$-modules 
$$0\rightarrow F_1\rightarrow F_2\rightarrow \ldots\rightarrow F_{k-1}\rightarrow 0.$$
Pull-back to $Spec_{\mathcal{D}_X}(\mathcal{A}^{\ell})$ via $p_{\infty}$ and using the fact that $Jet^{\infty}(F_i)\times_X Spec_{\mathcal{D}_X}(\mathcal{A}^{\ell})\simeq \overline{Jets}^{\infty}(p_{\infty}^*F_i),$ holds for each $i$, with $\overline{Jets}^{\infty}$ the horizontal jet functor \cite{KSYI}, we have polynomial functions on $Jet^{\infty}(F_{\bullet})\times_X Spec_{\mathcal{D}_X}(\mathcal{A}),$ are given by sections of  $$Sym_{\mathcal{A}^{\ell}}(D'\overline{\mathcal{J}}^{\infty}(p_{\infty}^*F_{\bullet})\big),\hspace{1mm}\text{where } 
D'(-):=\mathcal{H}om_{\mathcal{A}^{\ell}}(-,\mathcal{A}^{\ell}).$$

By the isomorphism $\mathcal{A}^{\ell}\otimes^!\mathcal{D}_X(F_{\bullet},\mathcal{O}_X)\simeq \mathcal{A}^{\ell}[\mathcal{D}_X]\big(p_{\infty}^*F_{\bullet},\mathcal{A}^{\ell}),$ we have that 
$$Sym_{\mathcal{A}^{\ell}}\big(\mathcal{A}^{\ell}\otimes^!\mathcal{D}_X(F_{\bullet},\mathcal{O}_X)\big)\simeq \mathcal{A}^{\ell}\otimes^!Sym_{\mathcal{O}_X}(\mathcal{D}_X(F_{\bullet},\mathcal{O}_X)\big).$$
\begin{proposition}
For any (graded) vector bundle $F_{\bullet},$ denote the Serre-dual bundle $F_{\bullet}^{\vee}\simeq F_{\bullet}^*\otimes \omega_X,$ with $\omega_X$ the dualizing complex.
Then, for all $p>0,$
$$H^i\big(Sym_{\mathcal{O}_X}^p(\mathcal{D}_X(F_{\bullet},\Omega_X^{\bullet})\big)\simeq Sym^{p-1,self}\big(\mathcal{D}_X(F_{\bullet},F_{\bullet}^{\vee})\big),i=n,$$
and is zero otherwise.
Similarly, 
$$H^i\big(Sym_{\mathcal{A}^{\ell}}^p\big(\mathcal{A}^{\ell}[\mathcal{D}_X](p_{\infty}^*F_{\bullet},p_{\infty}^*\Omega_X^{\bullet})\big)\simeq Sym_{\mathcal{A}}^{p-1,self}(\mathcal{A}[\mathcal{D}_X](p_{\infty}^*F_{\bullet},D'(p_{\infty}^*F_{\bullet})\otimes_{\mathcal{A}}(\mathcal{A}\otimes_{\mathcal{O}_X}\omega_X)).$$
    \end{proposition}

In coordinates, represent $\mathbf{P}\in \mathcal{D}_X(F_{\bullet},\mathcal{O}_X)$ as
$$\mathbf{P}=\sum_{\sigma}\big(\mathsf{P}_{\sigma}^1(x)\cdots \mathsf{P}_{\sigma}^{R}(x)\big)\partial^{\sigma},$$
where $R:=\sum_{j=1}^{k-1}rank(F_j).$ For each $j=1,\ldots,k-1$ let us denote sections 
$$v(j)\in \Gamma(Spec_{\mathcal{D}_X}(\mathcal{A}^{\ell}),p_{\infty}^*F_j).$$
Operator $\mathbf{P}$ defines a free $\mathcal{D}_X$-module on generators $v^{\lambda}(j)$ with $\lambda=1,\ldots,rank(F_j)$ for each $j=1,\ldots,k-1.$
In other words, we have a free left $\mathcal{D}_X$-module
$$\bigoplus_{j=1}^{k-1}\bigoplus_{\lambda-1}^{rank(F_j)}\mathcal{D}_X\bullet v^{\lambda}(j).$$
These variables $v^{\lambda}(j)$ are viewed as tuples $v^{\lambda}(j)(x,u_{\sigma}^{\alpha})$ for $\lambda=1,\ldots,rank(F_j),$ and the $\mathcal{D}_X$-action is interpreted as coordinates 
$$\partial^{\tau}\bullet v^{\lambda}(j)=D^{\tau}(v^{\lambda}(j))=v_{\tau}^{\lambda}(j)\in \Gamma(Spec_{\mathcal{D}_X}(\mathcal{A}^{\ell}),\overline{\mathcal{J}}^{\infty}(p_{\infty}^*F_j)\big).$$
\begin{proposition}
    $\overline{\mathcal{J}}^{\infty}(p_{\infty}^*F_{\bullet})$ defines a vector $\mathcal{D}_X$-bundle over $Spec_{\mathcal{D}_X}(\mathcal{A}^{\ell})$ (equiv. a vector $\mathcal{A}^{\ell}[\mathcal{D}_X]$-bundle) by setting 
    $$\mathbb{V}_{\mathcal{A}^{\ell}}(F_{\bullet}):=\mathbb{V}_{\mathcal{A}^{\ell}}\big(\overline{\mathcal{J}}^{\infty}(p_{\infty}^*F_{\bullet})\big):=Spec_{\mathcal{D}_X}\big(Sym_{\mathcal{A}^{\ell}}^{\bullet}(\overline{J}^{\infty}(p_{\infty}^*F_{\bullet})\big),$$ and where
    $\mathcal{O}\big(\overline{\mathcal{J}}^{\infty}(p_{\infty}^*F_{\bullet})\big)$ has the structure of both a commutative $\mathcal{D}_X$-algebra and an $\mathcal{A}^{\ell}$-algebra via the structure map $\gamma:\mathbb{V}_{\mathcal{A}^{\ell}}(F_{\bullet})\rightarrow Spec_{\mathcal{D}_X}(\mathcal{A}^{\ell}).$
\end{proposition}
\begin{proof}
    Note that $f=f\big(x,u_{\sigma}^{\alpha},v_{\tau}^{\lambda}(j)\big)\in \mathcal{O}\big(\overline{\mathcal{J}}^{\infty}(p_{\infty}^*F_{\bullet})\big)$ so the algebra structure 
    $$\mathcal{O}\big(\overline{\mathcal{J}}^{\infty}(p_{\infty}^*F_{\bullet})\big)\otimes \mathcal{O}\big(\overline{\mathcal{J}}^{\infty}(p_{\infty}^*F_{\bullet})\big)\rightarrow \mathcal{O}\big(\overline{\mathcal{J}}^{\infty}(p_{\infty}^*F_{\bullet})\big),$$
    is determined by
   the obvious rule on derivatives of variables i.e. functions of $u_{\sigma}^{\alpha},v_{\tau}^{\lambda}(j)$ and $u_{\sigma'}^{\alpha},v_{\tau'}^{\lambda}(j)$ for each $\alpha,\lambda,j$ give functions of $u_{\sigma+\sigma'}^{\alpha},v_{\tau+\tau'}^{\lambda}(j).$ The $\mathcal{D}_X$-structure is given by total derivatives again, by the extended action,
   $$\mathcal{D}_X\bullet \mathcal{O}\big(\overline{\mathcal{J}}^{\infty}(p_{\infty}^*F_{\bullet})\big)\rightarrow \mathcal{O}\big(\overline{\mathcal{J}}^{\infty}(p_{\infty}^*F_{\bullet})\big),\partial_i\bullet f\mapsto \overline{D}_i(f),$$
   with $\overline{D}_i:=\partial_i+u_{\sigma+i}^{\alpha}\partial_{u_{\sigma}^{\alpha}}+v_{\tau+i}^{\lambda}(j)\partial_{v_{\tau}^{\lambda}(j)}$.
   
\end{proof}

\subsection{Local Calculus.}
\label{ssec: Local Calculus}
For a $\mathcal{D}_X$-algebra $\mathcal{A}$, there is an $\mathcal{A}[\mathcal{D}_X^{\mathrm{op}}]$-module
\begin{equation}
    \label{eqn:Local vector fields}
\Theta_{\mathcal{A}}:=\big(\Omega_{\mathcal{A}}^1\big)^{\circ}=\mathcal{H}\mathrm{om}_{\mathcal{A}[\mathcal{D}_X]}\big(\Omega_{\mathcal{A}}^1,\mathcal{A}[\mathcal{D}_X]\big),
\end{equation}
with $\Omega_{\mathcal{A}}^1$ given by a $\mathcal{D}$-linear Kahler construction. It is isomorphic in for $\mathcal{D}$-smooth algebras (like the jet case) to 
$\mathcal{A}[\mathcal{D}_X]\otimes_{\mathcal{A}}p_{\infty}^*\Omega_{E/X}^1,$ and its $\mathcal{A}[\mathcal{D}_X]$-module rank is equal to the $\mathcal{O}_E$-module rank of $\Omega_{E/X}^1.$

Local vector fields are derivations of $\mathrm{Jet}^{\infty}(\mathcal{O}_E)$ or a more general $\mathcal{A}\in \mathrm{CAlg}_{\mathcal{D}_X},$ as $\mathcal{D}_X$-linear derivations,
\begin{equation}
    \label{eqn: D-Derivation Functor}
\mathrm{Der}_{\mathcal{D}_X}\big(\mathcal{A},-\big):=\mathrm{Der}(\mathcal{A},-)\cap \mathcal{H}\mathrm{om}_{\mathcal{D}_X}\big(\mathcal{A},-\big).
\end{equation}

Put $\mathrm{Der}_{\mathcal{D}_X}(\mathcal{A}):=\mathrm{Der}_{\mathcal{D}_X}(\mathcal{A},\mathcal{A})$ and denote the sheaf associated to the presheaf $\mathrm{Der}_{\mathcal{D}_X}(\mathcal{A})$ by $\mathcal{D}\mathrm{er}_{\mathcal{D}_X}(\mathcal{A}).$ 
For an open subset $U$ of $X$ with embedding $j:U\hookrightarrow X$,
$$\mathcal{D}\mathrm{er}_{\mathcal{D}_X}\big(\mathcal{A}\big)(U):=\big\{\E:j^*\mathcal{A}\rightarrow j^*\mathcal{A} |\E(P\bullet f_U)=P\bullet\E(f_U),\hspace{1mm} f_U\in j^*\mathcal{A},P\in \mathcal{D}_U\big\},$$
where $\E$ moreover satisfy the obvious Leibniz rule and $j^*=j^{-1}$, as usual for $\mathcal{D}$-module pull-backs for open embeddings.

\begin{proposition}
\label{D-linear O-linear relationship Lemma}
If $\mathcal{A}=\mathrm{Jet}^{\infty}(\mathcal{O}_E)$ is an algebraic jet $\D$-algebra there are isomorphisms $\mathrm{Der}_{\mathcal{D}_X-\mathrm{Alg}}(\mathcal{A},\mathcal{A})\cong \mathrm{Der}_{\mathcal{O}_X-\mathrm{Alg}}(\mathcal{O}_E,\mathcal{A}).$ 
Derivations are naturally filtered by
$\mathcal{F}^k\mathcal{D}\mathrm{er}(\mathcal{A}):=\mathcal{D}\mathrm{er}\big(\mathcal{F}^k\mathcal{A}\big)$ for $k\geq 0.$
In particular, if $E=X\times M$ then $\mathrm{Der}_{\mathcal{D}_X-\mathrm{Alg}}(\mathcal{A},\mathcal{A})\cong \Theta_M\otimes_{\mathcal{O}_M}\mathcal{A}.$
\end{proposition}
By Proposition \ref{D-linear O-linear relationship Lemma} one has
$\mathcal{F}^k\mathrm{Der}_{\mathcal{D}_X}\big(\mathcal{A},\mathcal{A}\big)\cong \mathrm{Der}_{\mathcal{O}_X-\mathrm{Alg}}\big(\mathcal{O}_E,\mathcal{F}^k\mathcal{A}\big).$

If $\EQ=\mathrm{Spec}_{\mathcal{D}_X}(\mathcal{A}),$ is an algebraic jet $\mathcal{D}$-scheme with projection 
$\pi_{\infty}:\EQ\rightarrow E$ there is an isomorphism of $\mathcal{A}^r[\mathcal{D}_X^{\mathrm{op}}]$-modules,
$\Theta_{\mathcal{A}}\rightarrow \big(\pi_{\infty,0}^*\Theta_{E/X}\big)\otimes_{\mathcal{A}^r}\mathcal{A}^r[\mathcal{D}_X^{\mathrm{op}}].$

In coordinates, they look like
\begin{equation}
\label{eqn: D-Geometric Generating Section}
\chi:=\sum_{\alpha=1}^{\mathrm{rank}(E)}\chi^{\alpha}(x,u,u_{\sigma})\frac{\partial}{\partial u^{\alpha}}\in \pi_{\infty}^*\Theta_{E/X}, \hspace{2mm} \chi^{\alpha}\in \mathcal{A},
\end{equation}
while $\E_{\xi}=\sum \E_{\sigma}^{\alpha}\frac{\partial}{\partial u_{\sigma}^{\alpha}}\in \mathrm{Der}_{\mathcal{D}_X}\big(\mathcal{A},\mathcal{A}\big).$
\begin{example}
\label{example: VectorFields}
\normalfont Consider the bundle $E=\mathbb{R}\times\mathbb{R}\rightarrow X:=\mathbb{R}$ with fiber coordinates $(t,u)$ and base coordinate $t.$
Then, $\Theta_{E/X}$ is spanned by $\partial_u$ with pull-back $\mathcal{A}$-module $\pi_{\infty}^*\Theta_{E/X}\simeq \mathcal{A}[\partial_u].$ The map
$\pi_{\infty}^*\Theta_{E/X}\rightarrow \Theta_{\mathcal{A}}$ sends $$\xi:=f(x,u,u_{\sigma})\frac{\partial}{\partial u}\mapsto \E_{\xi}:=\sum_{\sigma}\big(D_t^{\sigma}(f)\big)\frac{\partial}{\partial u_{\sigma}}.$$ The $\mathcal{D}$-action is through 
$D_t:=\frac{\partial}{\partial t}+\sum_{n\geq 0}u_{n+1}\frac{\partial}{\partial u_n},$ where $u_n:=\frac{\partial^n u}{\partial t^n}.$
\end{example}
\begin{remark}
Following standard terminology, vector fields $\E_{\xi}$ associated to sections \emph{(\ref{eqn: D-Geometric Generating Section})}
are called evolutionary derivations.
\end{remark}

When $\mathcal{A}$ is $\mathcal{D}_X$-smooth, we have a natural bi-duality isomorphism 
$$\Omega_{\mathcal{A}}^1\rightarrow \mathcal{H}\mathrm{om}_{\mathcal{A}^r[\mathcal{D}_X^{\mathrm{op}}]}\big(\Theta_{\mathcal{A}},\mathcal{A}^r[\mathcal{D}_X^{\mathrm{op}}]\big).$$
Forgetting the $\mathcal{D}_X$-module structure i.e. viewing $\Omega_{\mathcal{A}}^1$ as a $\mathcal{A}^{\ell}$-module, there is a natural $\mathcal{O}$-linear isomorphism of $\mathcal{O}$-modules $\Omega_{\mathcal{A}/\mathcal{O}}^1\rightarrow\Omega_{\mathcal{A}}^1$ and one way to interpret the $\mathcal{D}$-module structure induced on $\Omega_{\mathcal{A}}^1$ is coming from a section of the natural projection $\Omega_{\mathcal{A}/\mathbb{C}}^1\rightarrow \Omega_{\mathcal{A}/X}^1$.\footnote{i.e. as an Ehresmann connection.} Put
$\Omega_{\mathcal{A}}^*:=\Bigwedge_{\mathcal{A}}^*\Omega_{\mathcal{A}}^1,$ taken in the category of $\mathcal{A}^{\ell}[\mathcal{D}_X]$-modules.
In coordinates (c.f. Example \ref{example: VectorFields}) for the trivial rank $m$ bundle, a basis for $\Omega_{\mathcal{A}}^1$ is given by Cartan forms $\omega_{\sigma}^{\alpha}=du_{\sigma}^{\alpha}+\sum_{i=1}^nu_{\sigma+1_i}^{\alpha} dx^i.$

A $\mathcal{D}$-linear calculus can be obtained but we do not elaborate it in full detail here, choosing only to give one example construction. 

Consider an $\mathcal{O}_X$-linear derivation of $\mathcal{O}_E$ with values in $\mathcal{A}^{\ell},$ denoted $Y=V^{\alpha}\otimes f_{\alpha}.$ There is a natural $\mathcal{O}_X$-linear insertion operation $\iota(Y)$ 
\begin{equation}
\label{eqn: Iota(Y)}
\iota(Y):\Omega_{E/X}^*\rightarrow \mathrm{Jet}^{\infty}\big(\Omega_{E/X}^*\big),\hspace{2mm} \iota(Y)(\omega):=(-1)^{V^{\alpha}\dot f_{\alpha}}f_{\alpha}\iota_{V^{\alpha}}\omega.
\end{equation}
The image of (\ref{eqn: Iota(Y)}) lies in $\pi_{\infty}^*\Omega_{E/X}^*,$ and by the universal property of algebraic jet functors gives a $\mathcal{D}_X$-linear derivation $\iota_{\mathcal{D}_X}(Y),$ of $\mathrm{Jet}^{\infty}\big(\Omega_{E/X}^*\big)$, such that 
$$\iota_{\mathcal{D}_X}(Y)\big|_{\Omega_{E/X}^*\subset \mathrm{Jet}^{\infty}(\Omega_{E/X}^*)}=\iota(Y).$$
We obtain a map $\iota_{\mathcal{D}}(Y):\mathrm{Jet}^{\infty}\big(\Omega_{E/X}^*\big)\rightarrow \mathrm{Jet}^{\infty}\big(\Omega_{E/X}^*\big),$ defined by
$$\mathsf{P}(x,\partial_x)\otimes \omega\mapsto (-1)^{\mid \mathsf{P}\mid \cdot\mid Y\mid }\mathsf{P}(x,\partial_x)\otimes\iota(Y)\omega,\hspace{2mm}\mathsf{P}\in \mathcal{D}_X.$$ More generally there are $\mathcal{D}$-linear maps $\iota_{\mathcal{D}}(Y_1,\ldots,Y_k):=\iota_{\mathcal{D}}(Y_1)\circ\ldots\circ \iota_{\mathcal{D}}(Y_k).$
\\

Since $\mathcal{A}$ is in particular an $\mathcal{O}_X$-algebra consider $\Bigwedge_{\mathcal{O}_X}^{\bullet}\Omega_{\mathcal{A}}^1,$ with induced filtration $\mathcal{F}^k\Omega_{\mathcal{A}}^{\bullet}=\Bigwedge_{\mathcal{O}_X}^{\bullet}\Omega_{\mathcal{F}^k\mathcal{A}}^1,$ and consider  $\Omega_{E/X}^{\bullet}:=\Omega^{\bullet}(E)/\pi^*(\Omega_X^{\bullet}),$ and 
the $\mathcal{O}_X$-linear differential $d:\Omega_{E/X}^{\bullet}\rightarrow \Omega_{E/X}^{\bullet+1}.$ It injects into
$\mathrm{Jet}^{\infty}\big(\Omega_{E/X}^{\bullet}\big)$ and by
Proposition \ref{D-linear O-linear relationship Lemma} there exists a $\mathcal{D}_X$-linear derivation 
$$d_{\infty}:\mathrm{Jet}^{\infty}\big(\Omega_{E/X}^{\bullet}\big)\rightarrow 
\mathrm{Jet}^{\infty}\big(\Omega_{E/X}^{\bullet}\big),\hspace{2mm}\text{ with}\hspace{2mm}d_{\infty}|_{\Omega_{E/X}^{\bullet}\subset \mathrm{Jet}^{\infty}(\Omega_{E/X})}=d.$$
Since $d^2=0$ we have $d_{\infty}^2=0,$ and $\mathrm{Jet}^{\infty}\big(\Omega_{E/X}^{\bullet})$ is naturally a differential graded commutative algebra. One may show that
$\Omega_{\mathcal{A}}^{\bullet}\cong \mathrm{Jet}^{\infty}\big(\Omega_{E/X}^{\bullet}),$
is both an isomorphism of dg-algebras and $\mathcal{A}^{\ell}$-modules.
This can be seen, for instance, by using the sequence
$$0\rightarrow p_{\infty}^*\Omega_{E/X}^1\rightarrow \Omega_{Spec_{\mathcal{D}}(\mathcal{A})/X}^1\rightarrow \Omega_{Spec_{\mathcal{D}}(\mathcal{A})/E}^1\rightarrow 0.$$

If $\Omega_{\mathcal{A}}^1$ is locally of finite $\mathcal{A}[\D]$-presentation, for example if $\mathcal{A}$ is $\D$-smooth, there are operations
$$ev_2:\Omega_{\mathcal{A}}^1\boxtimes\Theta_{\mathcal{A}}^{\ell}\rightarrow \Delta_*\mathcal{A},\hspace{1mm} coev_2:\Delta_*\mathcal{A}\rightarrow \Theta_{\mathcal{A}}^{\ell}\boxtimes \Omega_{\mathcal{A}}^1.$$
The local Lie bracket $[-,-]:\Theta_{\mathcal{A}}\boxtimes\Theta_{\mathcal{A}}\rightarrow \Delta_*\Theta_{\mathcal{A}},$ induces a $1$-shifted local Poisson bracket 
$$\{-,-\}:Sym_{\mathcal{A}}^p(\Theta_{\mathcal{A}}[1])\boxtimes Sym_{\mathcal{A}}^q(\Theta_{\mathcal{A}}[1])\rightarrow \Delta_* Sym_{\mathcal{A}}^{p+q-1}(\Theta_{\mathcal{A}}[1]).$$
Finally, there is a natural local Lie derivative operation $\mathcal{L}:\Theta_{\mathcal{A}}\boxtimes\mathcal{A}^r\rightarrow \Delta_*\mathcal{A}^r,$ in $\mathcal{D}_{X\times X}^{op}$-modules. These make $\Theta_{\mathcal{A}}^{\ell}$ into a local Lie algebroid over $\mathcal{A}$ whose bracket is locally completely determined on elements $\varphi=\varphi^{\alpha}\partial_{\alpha}$ and $\psi\in \psi^{\beta}\partial_{\beta}$ by functions
$$[\varphi,\psi]^{\beta}=\sum [\E_{\varphi}]_{\sigma}^{\alpha}\frac{\partial \varphi^{\beta}}{\partial u_{\sigma}^{\alpha}}-(\E_{\psi})_{\sigma}^{\alpha}\frac{\partial \varphi^{\alpha}}{\partial u_{\sigma}^{\alpha}},\hspace{3mm}[\E_{\varphi}]_{\sigma}^{\alpha}:=(\D_X)_{\sigma}\bullet(\varphi^{\alpha}).$$

All constructions carry over to arbitrary algebraic $\D$-spaces $\mathcal{X}$, due to their étale local nature. That is, taking $\Omega_{\mathcal{A}}^1$ as the local model on affine $\D_X$-schemes $U=Spec_{\mathcal{D}_X}(\mathcal{A}),$ then define $1$-forms on any functor $\mathcal{X}$ by specifying for every morphism $x:h_{U_x}^{\mathcal{D}}\rightarrow \mathcal{X}$ a $1$-form $[x^*\omega]\in \Omega_{U_x}^1$ such that for all morphisms $f:x\rightarrow y$ (over $\mathcal{X}$) i.e. commuting diagrams,
\[
\begin{tikzcd}
h_{U_x}^{\mathcal{D}}\arrow[d,"f"] \arrow[r,"x"] & \mathcal{X}
\\
h_{U_y}^{\mathcal{D}}\arrow[ur,"y"] & 
\end{tikzcd}\hspace{2mm}\text{then }\hspace{1mm} f^*[y^*\omega]\cong [x^*\omega].
\]
Proceeding similarly one may define differential $p$-forms on spaces $\mathcal{X},$ for $p\geq 1.$
\subsubsection{The $\D_X$-linear derived de Rham differential}
\label{sssec: D de Rham diff}
The homotopical analogs of $1$-forms as defined in \cite{KRY,KSYI} are obtained via taking derived functors,
$$\mathbb{L}\Omega^1(-):Ho(\cdga_{\mathcal{A}/})\rightarrow Ho(\DG(\mathcal{A})).$$
One may then set 
$\mathbb{T}_{\mathcal{B}}:=\mathbb{D}_{\mathcal{A}}^{ver}(\mathbb{L}\Omega_{\mathcal{B}}^1).$
\vspace{1mm}

The $\D_X$-linear derived de Rham differential can be straightforwardly defined using local operations (\ref{eqn: Holocals}). It is the element in $\mathbb{R}\mathcal{P}_{[1]}^*\big(\mathcal{A},\mathbb{L}\Omega_{\mathcal{A}}^1),$ satisfying the obvious graded Leibniz rule. In particular, it is the universal graded derivation (respecting all internal cohomological differentials) from $\mathcal{A}$ to the dg-$\mathcal{A}[\mathcal{D}]$-module $\mathbb{L}\Omega_{\mathcal{A}}^1$ which is additionally $\mathcal{D}$-linear.
For the case of free $\D$-algebras, it is simple. Namely, consider $\mathbb{L}_{\mathsf{Free}_{\mathcal{D}}(\mathcal{M})}\simeq \mathsf{Free}_{\mathcal{D}}(\mathcal{M})\otimes\mathcal{M}$ as an object of $\mathsf{Free}_{\mathcal{D}}(\mathcal{M})-\DG.$ View $\mathsf{Free}_{\mathcal{D}}^0(\mathcal{M})$ by omitting the $0$-th component and so the corresponding (derived) de Rham differential $d_{dR}:\mathsf{Free}_{\mathcal{D}}^0(\mathcal{M})\rightarrow \mathsf{Free}_{\mathcal{D}}(\mathcal{M})\otimes\mathcal{M}.$

\subsection{De Rham Algebras}
\label{ssec: De Rham Algebras}
We now extend the $\D$-linear construction of Kähler forms, or better yet the $\D$-geometric cotangent complex, to its full (derived $\D$-geometric) de Rham algebra. 

Essentially, the object 
 $\mathsf{DR}^{\mathcal{D}}(\mathcal{A})$ we describe is the de Rham complex with coefficients in the universal complex of integral forms $\mathcal{I}_{*,X}$ ($\mathbb{R}\mathcal{H}om_{\mathcal{D}}$ of the Berezinian complex) if $X$ is a super-variety \cite{PenkovSuperD}, or in the de Rham algebra $\Omega_X^*$ of $X$, in the even case. 
\begin{remark}
 We do not discuss any details about the super-setting, but mention it only due to its physical relevance in modelling super-symmetric field theories via $\D$-geometry. 
 \end{remark}
 
 The variational de Rham complex of $\mathcal{A}^{\ell}$ is the $\D$-de Rham algebra for the underlying $\mathcal{A}[\D_X]$-module $\Omega_{\mathcal{A}}^*=Sym(\Omega_{\mathcal{A}}^1[1])$ (décalage). 

We will omit the explicit pull-back via $p_{\infty}$ for objects living over $X,$ for simplicity.
 \begin{definition}
\label{Variational de Rham complex definition}
    \normalfont Fix a left $\mathcal{A}[\mathcal{D}_X]$-module $\mathcal{M}.$ 
   The \emph{de Rham complex of $\mathcal{M}$} (over $\mathcal{A})$ is
   $$\mathsf{DR}_{\mathcal{A}}^{\mathcal{D}}(\mathcal{M}):=(\mathcal{A}\otimes_{\mathcal{O}_X}\omega_X)\otimes_{\mathcal{A}\otimes_{\mathcal{O}_X}\mathcal{D}_X}^{\mathbb{L}}\mathcal{M},$$ with obvious replacements of canonical sheaf with sheaves of integral forms when $X$ is a super-variety.
In this case the $\D$-\emph{geometric de Rham algebra} of $\mathcal{A}^{\ell}$ is 
\begin{equation}
\label{eqn: de Rham algebra A}
\mathsf{DR}^{\mathcal{D}}(\mathcal{A}):=\mathsf{DR}_{\mathcal{O}_X}^{\mathcal{D}}(\mathcal{A})=\big(\mathcal{I}_{*,X}\otimes_{\mathcal{O}_X}\mathcal{D}_X[n]\big)\otimes_{\mathcal{D}_X}^{\mathbb{L}}\mathcal{A}^{\ell}.
\end{equation}
The \emph{variational de Rham complex of $\mathcal{A}$} is  $\mathsf{DR}^{var}(\mathcal{A}^{\ell}):=\mathsf{DR}_{\mathcal{O}_X}^{\mathcal{D}}(\Omega_{\mathcal{A}}^*).$
\end{definition}
The (central) de Rham cohomology ($H^0$ of de Rham complex) is $$h(\mathcal{M}):=\Omega_X^n\otimes_{\mathcal{D}_X}\mathcal{M},$$
and to $\mathcal{A}[\mathcal{D}_X]$-modules $\mathcal{M}$ is
$$h_{\mathcal{A}}(\mathcal{M}):=H^0\big(\mathsf{DR}_{\mathcal{A}}^{\mathcal{D}}(\mathcal{M})\big).$$

If $\mathcal{A}$ is a left $\D_X$-algebra, we can consider the central cohomology of also its right structure as $h^r(\mathcal{A}):=H^n(\mathcal{A}^r\otimes_{\mathcal{D}_X}\mathcal{O}_X).$

Formation of Kähler differentials and its subsequent wedge-powers are étale local, Definition \ref{Variational de Rham complex definition} defines a functor
$$\mathsf{DR}^{var}(-):Comm(\mathcal{D}_X-mod)_{et}\rightarrow cdga_{\mathcal{D}_X}.$$
The weight $p$-piece of the variational de Rham algebra is
\begin{equation}
\label{eqn: Weight p}
\mathsf{DR}^{var}(\mathcal{A}^{\ell})(p)=\mathsf{DR}^{var}\big(\Omega_{\mathcal{A}}^p\big)=\big(\mathcal{I}_{*,X}\otimes \mathcal{D}_X[n]\big)\otimes_{\mathcal{D}_X}^{\mathbb{L}}\Omega_{\mathcal{A}}^p.
\end{equation}
If $\mathcal{X}$ is an algebraic $\D_X$-space its variational de Rham complex is 
  $$\mathsf{DR}^{var}(\mathcal{X}):=\varprojlim_{\mathrm{Spec}_{\mathcal{D}}(\mathcal{A})\rightarrow \mathcal{X}}\mathsf{DR}^{var}\big(\mathrm{Spec}_{\mathcal{D}}(\mathcal{A}^{\ell})\big).$$
We are interested in the associated cohomology spaces, denoted by
$$h^*\big(\mathsf{DR}(\mathcal{A})\big):=H^*\big(\mathcal{I}_{*,X}\otimes\mathcal{D}_X[n]\otimes^{\mathbb{L}}\mathcal{A}\big),$$
as well as 
\begin{equation}
    \label{eqn: Variational cohomologies}
h^{*,\mathrm{var}}\big(\mathcal{A}^{\ell}\big):=h^*\big(\mathsf{DR}^{var}(\mathcal{A}^{\ell})\big):=H^*\bigg(\big(\mathcal{I}_{*,X}\otimes_{\mathcal{O}_X}\mathcal{D}_X[n]\big)\otimes_{\mathcal{D}_X}^{\mathbb{L}}\Bigwedge_{\mathcal{A}}^*\Omega_{\mathcal{A}}^1\bigg).
\end{equation}
Spaces (\ref{eqn: Variational cohomologies}) are called the \emph{variational cohomologies} of $\mathcal{A}.$
\begin{proposition}
\label{prop: 0,1 var h}
Suppose that $\mathcal{A}^{\ell}:=\mathrm{Jet}^{\infty}\big(\mathcal{O}_E^{\mathrm{alg}}\big)$ for a morphism of locally ringed spaces $(E,\mathcal{O}_E)\rightarrow (X,\mathcal{O}_X).$ Then:
\begin{enumerate}
    \item $h^{0,\mathrm{var}}(\mathcal{A}^{\ell})(0)\cong H^0\big(\omega_X\otimes_{\mathcal{D}_X}^{\mathbb{L}}\mathcal{A}^{\ell}\big)=\omega_X\otimes_{\mathcal{D}_X}\mathcal{A}^{\ell},$ (in degree $n$);
    
    \item $h^{0,\mathrm{var}}(\mathcal{A}^{\ell})(1)\cong h(\Omega_{\mathcal{A}}^1)\cong \mathrm{Hom}_{\mathcal{A}^{\ell}[\mathcal{D}_X]}\big(\omega_X^{-1}\otimes\Theta_{\mathcal{A}},\mathcal{A}\big)$
    
\end{enumerate}
\end{proposition}
The homotopical bi-complex, rather its derived description, is to be denoted by either $\mathbb{L}\mathsf{DR}^{var}(-)$ or simply $\mathcal{V}ar$ (as in Sect. \ref{sec: D-Geometric DerVar}) and is obtained by the functor
$$\mathcal{S}ym_{\mathcal{A}}\circ\mathbb{L}\Omega^1(-):Ho(\cdga_{\mathcal{A}/})\rightarrow Ho(\DG(\mathcal{A}))\rightarrow Ho(\cdga_{\mathcal{A}/}^{gr}),$$
where $\mathcal{S}ym_{\mathcal{A}}$ is the derived symmetric algebra functor, given in the usual way: 
$$\mathcal{S}ym_{\mathcal{A}}^n(-)\simeq (-)_{\Sigma_n}\circ \otimes_{\mathcal{A}}^{\mathbb{L},n}:Ho(\DG(\mathcal{A}))\rightarrow Ho(\DG(\mathcal{A})),$$
e.g. $\mathcal{S}ym_{\mathcal{A}}^2(\mathcal{M}^{\bullet})\simeq (Q\mathcal{M}\otimes_{\mathcal{A}}Q\mathcal{M})_{\Sigma_2},$ with $Q$ the cofibrant replacement functor. 
\begin{proposition}
\label{prop: Sym preserves W}
    For all $p\geq 0,$ the functor $Sym_{\mathcal{A}}^p(-):\DG(\mathcal{A})\rightarrow \DG(\mathcal{A}),$ preserves quasi-isomorphisms between cofibrant $\mathcal{A}[\mathcal{D}_X]$-modules. 
\end{proposition}
The $\infty$-functor arising from Proposition \ref{prop: Sym preserves W} is denoted as usual $Sym_{\mathcal{A}}^*.$  
We have that,
\begin{equation}
\label{eqn: Derived Variational Algebra}
\mathsf{DR}_{\mathcal{B}}(\mathcal{S}ym(\mathbb{L}\Omega_{\mathcal{B}}^1[1])\big)\simeq p_{\infty}^*\omega_X\otimes_{\mathcal{A}\otimes\mathcal{D}_X}^{\mathbb{L}}\mathcal{S}ym(\mathbb{L}\Omega_{\mathcal{B}}^1[1])\big)\simeq p_{\infty}^*\omega_X\otimes_{\mathcal{A}\otimes\mathcal{D}_X}^{\mathbb{L}}\Bigwedge^*\mathbb{L}_{\mathcal{B}},
\end{equation}
where the wedge is necessarily understood in the derived sense and $p_{\infty}:Spec_{\mathcal{D}}(\mathcal{A})\rightarrow X$ is the structure projection.
\vspace{1mm}

We now globally extend (\ref{Definition: (Formal) Algebraic D-Spaces}).
Namely, let $\EQ$ be an algebraic $\D_X$-space flat over $X.$ There is a complex of left $\mathcal{O}_{\EQ}[\mathcal{D}_X]$-modules,
$$\mathsf{DR}(\EQ/X)=\bigoplus_i \mathsf{DR}^i(\EQ/X), \hspace{2mm}\mathsf{DR}^i(\EQ/X):=\Omega_{\EQ/X}^i.$$
Passing to $\D^{op}$-modules produces a complex of sheaves $h\big(\mathsf{DR}(\EQ/X)\big)$ on the étale site of $\EQ.$ 

Viewed as a $\D_X$-module we can take the $\D$-module de Rham complex:
\begin{equation}
    \label{eqn: D-Module BiComplex}
    DR_X\big(\mathsf{DR}(\EQ/X)\big)\simeq\big[\mathcal{O}_{\EQ}\rightarrow \Omega_{\EQ}^1\rightarrow \ldots\rightarrow \Omega_{\EQ}^p\rightarrow \Omega_{\EQ}^{p+1}\rightarrow\cdots\big].
    \end{equation}
    The right-hand side of (\ref{eqn: D-Module BiComplex}) is $\Omega_{\EQ}^*$ and the components of $d_{\EQ}$ corresponding to $DR_X$ and $\mathsf{DR}(\EQ/X)$ are to be denoted $d_H:=d_X,$ and $d^v:=d_{\EQ/X},$ and are referred to as the `horizontal' and `vertical' components, respectively.
    
    The $\D$-connection induces a splitting $\Omega_{\EQ}^1\simeq\Omega_{\EQ/X}^1\oplus p^*\Omega_X^1$ with $\Omega_{\EQ}^n\simeq \bigoplus_{p+q=n}\Omega_{\EQ/X}^p\otimes p^*\Omega_X^q,$ so that
    \[
\begin{tikzcd}
    & \arrow[dl, "d_h",labels=above left] \Omega_{\EQ/X}^p\otimes p^*\Omega_X^q\arrow[dr,"d^v"] & 
    \\
    \Omega_{\EQ/X}^p\otimes p^*\Omega_X^{q+1} & & \Omega_{\EQ/X}^{p+1}\otimes p^*\Omega_X^{q},
\end{tikzcd}
    \]
giving a clear realization of the $\D$-module analog of the bi-complex.
 Fixing a classical solution $\varphi:X\rightarrow Spec_{\mathcal{D}}(\mathcal{A})$ we obtain a complex of $\mathcal{D}_X$-modules $\varphi^*\mathbb{L}_{\mathcal{A}/\mathcal{O}_X}^{\bullet}.$ The corresponding de Rham complex is
$$DR_X(\varphi^*\mathbb{L}_{\mathcal{A}/\mathcal{O}_X}^{\bullet})=\bigoplus_{i=p+q}\Omega_X^{d_X+p}\otimes \varphi^*\mathbb{L}_{\mathcal{A}/\mathcal{O}_X}^{q},\hspace{2mm} D_{dR}:=d_1\pm \delta.$$
It is the total complex of the double complex whose $(p,q)$-term is $\Omega_X^{d_X+p}\otimes\varphi^*\mathbb{L}_{\mathcal{A}/\mathcal{O}_X}^q$ with differentials $d_1^q:=(-1)^{d_X}\nabla^q,$ and $d_2^p:=d_{dR}^p\otimes id.$ 
If $X$ is a super-variety:
$$\mathsf{DR}^{var}(\varphi^*\mathbb{L}_{\mathcal{A}}^{\bullet}):=\big(\mathcal{I}_{n-m-*,X}\otimes_{\mathcal{O}_X}\mathcal{D}_X\big)\otimes_{\mathcal{D}_X}^{\mathbb{L}}\varphi^*\mathbb{L}_{\mathcal{A}}^{\bullet},$$

For example, consider the free derived $\D$-algebra (\ref{sssec: D de Rham diff}). Then $\mathcal{A}\simeq Free_{\mathcal{D}_X}(\mathcal{M}^{\bullet})$ for a coherent $\D$-module $\mathcal{M}$ and we have that 
$$DR_X\big(\varphi^*\mathbb{L}_{Free_{\mathcal{D}_X}(\mathcal{M}^{\bullet})/\mathcal{O}_X}^{\bullet})\simeq \bigoplus_{i=p+q}\Omega_X^{d_X+p}\otimes \varphi^*\big(Free_{\mathcal{D}}(\mathcal{M}^{\bullet})\otimes_{\mathcal{O}_X}\mathcal{M}^{\bullet})^q.$$

\begin{proposition}
Suppose $\EQ$ is an algebraic $\D_X$-space, flat over $X$ for which $\Omega_{\EQ/X}^1$ is a vector $\D$-bundle on $\EQ.$ Then 
$H^0\big(DR_X(\Omega_{\EQ/X}^i)\big)=0,$ for all $i>0$ and $H^0\big(DR_X(\mathcal{O}_{\EQ})\big)=H^0\big(\mathsf{DR}(\EQ)\big)$.
\end{proposition}
\begin{proof}
Note that $\Omega_{\EQ/X}^i$ is a direct summand of induced $\mathcal{D}$-modules for $i>0.$ Then, since $f\in \mathcal{O}_{\EQ}$ such that $d_{\EQ}(f)\subset \Omega_{\EQ/X}^1\subset \Omega_{\EQ}^1,$ with $d_X(d_{\EQ}f)=0,$ we have $d_{\EQ}f\in H^0\big(DR_X(\Omega_{\EQ/X}^1)\big).$ Thus $d_{\EQ}f=0.$
\end{proof}
Truncating our complexes produces a sheaf parameterizing closed Cartan $2$-forms on $\EQ.$ Indeed, consider the complex of sheaves,
 $$E^{\bullet}:=\tau^{\leq 2}\sigma_{\geq 1}h\big(\mathsf{DR}(\EQ/X)\big).$$
 It is a $2$-term subcomplex with terms,
 $h(\Omega_{\EQ/X}^1),$ and $ ker\big(h(d^v):h(\Omega_{\EQ/X}^2)\rightarrow h(\Omega_{\EQ/X}^3)\big).$
 \begin{proposition}
Complex $E^{\bullet}$ is isomorphic to the single sheaf complex of sheaves 
$h(\Omega_{\EQ/X}^1)\rightarrow h(\Omega_{\EQ/X}^2)^{cl}.$ It determines a sheaf of Picard groupoids on the étale site of $\EQ.$ 
 \end{proposition}
 The sheaf $h(\Omega_{\EQ/X}^2)^{cl}$ is the sheaf of closed $2$-forms with respect to the vertical differential. 

\subsubsection{Characteristic Classes}
\label{sssec: CharClasses}
One of the many uses of understanding the global cohomology of $\D$-spaces via the bi-complex is that we may extract interesting cohomological invariants.
\begin{proposition}
Suppose that $\EQ$ is an étale local $\mathcal{D}$-space of finite presentation. Then for any vector $\mathcal{D}_X$-bundle $\mathcal{V}$ on $\EQ$, there is a well-defined sequence of characteristic classes,
$$ch_p^{\mathcal{D}}(\mathcal{V})\in \mathcal{H}^{2p+1}\big(\EQ_{\bullet},h(\mathsf{DR}^{\geq p}(\EQ/X)\big),p\geq 0.$$
\end{proposition}
\begin{proof}
Let $\EQ_{\bullet}\rightarrow\EQ$ be a $\mathcal{D}$-étale hypercover with $U_{\bullet}\rightarrow X$ an étale hypercover of $X$, such that $\EQ_{\beta}\simeq \bigcup \mathcal{U}_{\beta_i},$ with $\mathcal{U}_{\beta_i}=Spec_{\mathcal{D}_X}(\mathcal{A}_{\beta_i}).$

This data consists of morphisms of simplicial $\D_X$-spaces, $\varphi_{\bullet}:U_{\bullet}\rightarrow \EQ_{\bullet}\rightarrow \EQ,$ which, by choosing a chart $\{U_{\alpha}\}$ amounts to: a collection of morphisms of $\mathcal{D}_{U_{\alpha}}$-schemes specified by $(\varphi_{\alpha})$ together with morphisms of $\mathcal{D}_{U_{\alpha}\times_X U_{\beta}}$-schemes $(\psi_{\alpha\beta})$ for pairs of indices such that $\psi_{\alpha\beta}$ connects $\varphi_{\alpha}|_{U_{\alpha}\times_X U_{\beta}}$ with $\varphi_{\beta}|_{U_{\alpha}\times_X U_{\beta}}$ and satisfies cocycle conditions on higher overlaps.

By our assumptions we may represent $\EQ$ via its totalization $\EQ\simeq Tot(\EQ_{\bullet})$ of the simplicial affine $\mathcal{D}$-scheme determined by $\EQ_n=Spec_{\mathcal{D}}(\mathcal{A}_n^{\ell}),$ where $\mathcal{A}_n^{\ell}\simeq Sym_{\mathcal{O}_X}^*(\mathcal{M}_n[1])$ is the $n$-simplex of a simplicial object in smooth commutative $\mathcal{D}$-algebras \cite[Proposition 3.30--3.32]{KSYI}. In particular each $\mathcal{M}_n$ is a coherent $\mathcal{D}_X$-module.

Consider the sub-complex 
$$\mathsf{DR}^{\geq p}(\EQ/X):=\big[\Omega_{\EQ/X}^p\rightarrow \Omega_{\EQ/X}^{p+1}\rightarrow \cdots\big]\hookrightarrow \mathsf{DR}(\EQ/X),$$ and consider the complex $\mathrm{C}_{\EQ/X}^{\bullet}$ of Cech cochains with coefficients in $h(\mathsf{DR}^{\geq 1}(\EQ/X)\big)$. 

For all 
$i>0,$ one has that 
$$\mathcal{H}^p\big(\EQ_{\beta},h(\Omega_{\EQ/X}^i)\big)=0,$$
and that 
$H^{\bullet}(\mathrm{C}_{\EQ/X})=R\Gamma(\EQ_{\bullet},h(\mathsf{DR}^{\geq 1}(\EQ/X)\big).$
Our goal is to assign to the datum of a vector $\mathcal{D}_X$-bundle $\mathcal{V}$ on $\EQ$, a sequence of characteristic classes,
$$ch_p^{\mathcal{D}}(\mathcal{V})\in \mathcal{H}^{2p+1}\big(\EQ_{\bullet},h(\mathsf{DR}^{\geq p}(\EQ/X)\big),p\geq 0.$$
Considering the projection
$$DR_X\big(\mathsf{DR}^{\geq p}(\EQ/X)\big)\rightarrow h\big(\mathsf{DR}^{\geq p}(\EQ/X)\big)[-1],$$ one can see it is a quasi-isomorphism for all $p>0.$ Moreover unless $p=k=0,$ there are identifications
$$\mathcal{H}^k\big(\EQ_{\bullet},DR_X(\mathsf{DR}^{\geq p}(\EQ/X)\big)\simeq \mathcal{H}^{k-1}\big(\EQ,h(\mathsf{DR}^{\geq p}(\EQ/X)\big).$$
If $\mathcal{V}\simeq \mathcal{F}\otimes\mathcal{D}_X$ corresponds to a $\mathcal{O}_{\EQ}[\mathcal{D}_X]$-module with $\mathcal{F}$ a vector bundle on $\EQ,$ the classes $ch_p^{\mathcal{D}}(\mathcal{V})$ admit a simpler description in terms of the usual Chern classes in de Rham cohomology of $\EQ$ for $\mathcal{F}$ i.e. $ch_p(\mathcal{F})\in H^{2p}(\EQ,\mathsf{DR}_{\EQ}^{\geq p}).$ The first observation being that the product of forms endows $H^{2\bullet}(\EQ,\mathsf{DR}_{\EQ}^{\geq \bullet}),$ with a graded commutative algebra structure.\footnote{The so-called `Secondary wedge product' \cite{VinSecCalc}.}
Natural embeddings $\mathsf{DR}_{\EQ}^{\geq p+1}\subset DR_X(\mathsf{DR}_{\EQ/X}^{\geq p})\subset \mathsf{DR}_{\EQ}^{\bullet}$ furnish a map
$$H^{2p+2}\big(\EQ,\mathsf{DR}_{\EQ}^{\geq p+1}\big)\rightarrow H^{2p+2}\big(\EQ,DR_X(\mathsf{DR}_{\EQ/X}^{\geq p})\big)=H^{2p+1}\big(\EQ,h(\mathsf{DR}_{\EQ/X}^{\geq p})\big).$$
Then, the image of $ch_p(\mathcal{F})$ by $H^{2p}(\EQ,\mathsf{DR}_{\EQ}^{\geq p})\rightarrow R^{2p}a_*(\EQ),$ is the desired class, where $a_{\EQ}:\EQ\rightarrow pt.$
\end{proof}

\noindent\textit{\textbf{Aspects of Secondary Calculus in algebraic $\D$-geometry.}}
The $\D_X$-module structures on sheaves of local $1$-forms and vector fields induce, for a $\mathcal{D}_X$-smooth scheme $Spec_{\mathcal{D}}(\mathcal{A})$ a complex of $\mathcal{A}^{\ell}[\D_X]$-modules, for each $s,p\geq 0:$

\begin{equation}
    \label{eqn: Variational de Rham complex for s,p forms}
\mathsf{DR}_{\mathcal{A}}^{\mathcal{D}}\big(\Bigwedge^s\Omega_{\mathcal{A}}^1\otimes_{\mathcal{A}}\Bigwedge^p\Theta_{\mathcal{A}}^{\ell}\big):=(\mathcal{A}\otimes\omega_X)\otimes_{\mathcal{A}\otimes\mathcal{D}_X}^{\mathbb{L}}\big(\Bigwedge^s\Omega_{\mathcal{A}}^1\otimes_{\mathcal{A}^{\ell}}\Bigwedge^p\Theta_{\mathcal{A}}^{\ell}\big),
\end{equation}
where we used notation of Definition \ref{Variational de Rham complex definition}. Elements of the complex 
(\ref{eqn: Variational de Rham complex for s,p forms}) are called $\D_X$-\emph{geometric} (or \emph{local}) \emph{form-valued multivector fields} of \emph{type} $[p,s].$ 

Note that we must take the leftified module $\Theta_{\mathcal{A}}^{\ell}$.
In particular, we may consider the variational de Rham complex with coefficients in multivectors (omitting the explicit pull-back $p_{\infty}^*$).
\begin{definition}
\normalfont The complex of \emph{variational form-valued multivector fields} is 
$$\mathsf{DR}^{var}(\mathcal{A}^{\ell};\wedge^*\Theta_{\mathcal{A}}^{\ell}):=\big(\mathcal{I}_{*,X}\otimes\mathcal{D}_X[n]\big)\otimes_{\mathcal{D}_X}^{\mathbb{L}}\big(\Omega_{\mathcal{A}}^*\otimes_{\mathcal{A}}\wedge^*\Theta_{\mathcal{A}}\big).$$
\end{definition}
Multivector fields play an important role in the Hamiltonian formalism of PDEs and while a more detailed study will be discussed elsewhere we give just one example.
\begin{example}
\label{ex: Hamiltonian Operators}
\normalfont 
Consider the $\D$-algebra of differential polynomials in $m$ variables $u^1,\ldots,u^m.$ That is, polynomials in variables $u_I^{\alpha}$ e.g. of the form 
$$f(u^{\alpha},u_{\sigma}^{\alpha})=\sum_{k\geq 0}f_{\alpha_1\sigma_1,\ldots,\alpha_k\sigma_k}\prod_{j=1}^ku_{\sigma_j}^{\alpha_j}.$$
The space of $k=|I|$-fold $\mathcal{D}_X$-\emph{polyvector fields} is given by 
\begin{equation}
    \label{eqn:I-fold local polyvectors}
    \Theta_{\mathcal{A}}^k:=\mathcal{H}\mathrm{om}_{\mathcal{A}[\mathcal{D}_X]^{\boxtimes |I|}}\big((\Omega_{\mathcal{A}}^{\bullet})^{\otimes |I|},\Delta_*^{(I)}\mathcal{A}\big).
\end{equation}
When $\mathcal{A}=\mathbb{R}[u_{\sigma}^{\alpha}]$ is the free $\mathbb{R}[\partial]$-module of differential polynomials for any $\mathcal{A}[\mathcal{D}_X]$-module $\mathcal{M}$, one has that 
$\Delta_*^{(I)}\mathcal{M}$ is given by $\mathcal{M}\otimes_{\mathbb{R}[\partial]}\mathbb{R}[\partial_1,\ldots,\partial_{|I|}],$ and (\ref{eqn:I-fold local polyvectors}) is 
$$\mathcal{H}\mathrm{om}_{\mathcal{A}\otimes\mathbb{R}[\partial_1,\ldots,\partial_{|I|}]}\big((\Omega_{\mathcal{A}}^{\bullet})^{\otimes |I|},\mathcal{A}\otimes_{\mathbb{R}[\partial]}\mathbb{R}[\partial_1,\ldots,\partial_{|I|}]\big).$$

For instance $\pi\in\Theta_{\mathcal{A}}^{[2]}$ defines a skew-symmetric map $\pi:\Omega_{\mathcal{A}}^1\otimes \Omega_{\mathcal{A}}^1\rightarrow \mathcal{A}\otimes_{\mathbb{R}[\widetilde{\partial}]}\mathbb{R}[\widetilde{\partial_1},\widetilde{\partial_2}],$ which can be evaluated on basis $1$-forms, written as
$$\pi(\omega_{\sigma}^{\alpha},\theta_{\tau}^{\beta})=\pi_{\sigma\tau}^{\alpha\beta}\otimes_{\mathcal{D}_X}f_{I_1I_2}\widetilde{\partial}_1^{I_1}\widetilde{\partial}_2^{I_2},\hspace{5mm}\pi_{\sigma\tau}^{\alpha\beta},f_{I_1I_2}\in\mathcal{A},$$
with $\widetilde{\partial_i}^{I_i}:=\mathcal{C}\big(\partial_i^{I_i}\big)$ given by the $\mathcal{D}_X$-action. Local Lie bracket on vector fields extends to the (variational) Schouten $*$-bracket of (shifted) multi-vectors $$
    [\![-,-]\!]^{*}:\mathrm{Pol}^*\big(\mathcal{A}^{\ell},p\big)\boxtimes  \mathrm{Pol}^*\big(\mathcal{A}^{\ell},p\big)\rightarrow \Delta_*\mathrm{Pol}^*\big(\mathcal{A}^{\ell},p\big)[-1-p].$$
\end{example}
\vspace{1.5mm}

The space (\ref{eqn: Variational de Rham complex for s,p forms}) is naturally bi-graded:
$$\mathsf{DR}^{var}\big(\Bigwedge^*\Omega_{\mathcal{A}}^1\otimes_{\mathcal{A}}\Bigwedge^{\cdot}\Theta_{\mathcal{A}}^{\ell}\big):=\bigoplus_{s\geq 0}\bigoplus_{p\geq 0} \mathsf{DR}^{var}\big(\Bigwedge^s\Omega_{\mathcal{A}}^1\otimes_{\mathcal{A}}\Bigwedge^p\Theta_{\mathcal{A}}^{\ell}\big).$$

We may suitably interpret the variational $p$-multivector valued differential $s$-forms as the weight $s$-component of the variational de Rham complex (\ref{Variational de Rham complex definition}). In other words, put 
\begin{equation}
    \label{eqn: DRVAR weight s}
\mathsf{DR}^{var}(\mathcal{A}^{\ell})(s):=\mathsf{DR}_{\mathcal{O}}^{\mathcal{D}}(\Omega_{\mathcal{A}}^s),
\end{equation}
and with coefficients in the $\mathcal{A}^{\ell}[\mathcal{D}_X]$-module $\Bigwedge^p\Theta_{\mathcal{A}},$ a class in
$$h^{0,\mathrm{var}}(\mathcal{A}^{\ell})(s,p):=h^0\bigg(\mathsf{DR}_{\mathcal{A}}^{\mathcal{D}}\big(\Bigwedge^s\Omega_{\mathcal{A}}^1\otimes_{\mathcal{A}}\Bigwedge^p\Theta_{\mathcal{A}}^{\ell}\big)\bigg),$$
is a variational (also called functional/secondary) $p$-vector valued $s$-form on $Spec_{\mathcal{D}}(\mathcal{A}).$

We are interested further in the entire central cohomology space:
$$h^{\bullet,\mathrm{var}}(\mathcal{A}^{\ell})(*,*):=\bigoplus_{s,p}h^{\bullet,\mathrm{var}}(\mathcal{A}^{\ell})(s,p).$$
A variational vector field is thus an element of $h^{0,\mathrm{var}}(\mathcal{A}^{\ell})(0,1)$.
\begin{proposition}
Suppose that $X$ is an even variety, that is of dimension $n|0.$ 
Then there is an isomorphism of $k_X$-modules 
$h^{0,\mathrm{var}}(\mathcal{A}^{\ell})(0,1)\cong \omega_X\otimes_{\mathcal{D}_X}\Theta_{\mathcal{A}}^{\ell}.$
\end{proposition}
\begin{proof}
It suffices to notice that $
h^{0,\mathrm{var}}(\mathcal{A}^{\ell})(0,1)=H^0\big(\mathsf{DR}^{var}(\mathcal{A}^{\ell})(0)\big)\otimes_{\mathcal{A}}\Theta_{\mathcal{A}}$ 
which by Proposition \ref{prop: 0,1 var h} is given by $\big(\omega_X\otimes_{\mathcal{D}_X}\mathcal{A}^{\ell}\big)[n]\otimes_{\mathcal{A}}\Theta_{\mathcal{A}}
\simeq \omega_X\otimes_{\mathcal{D}_X}\Theta_{\mathcal{A}}^{\ell}.$
\end{proof}

When $\pi:E\rightarrow X$ is a bundle, $H\subset Sect(X,E)$ is a sub-space of solution sections defining a $\mathcal{D}$-algebraic PDE i.e. $\mathcal{D}$-ideal $\mathcal{I}_H\subset \mathrm{Jet}^{\infty}(E),$ whose $\mathcal{D}$-space of solutions is $Spec_{\mathcal{D}}(\mathcal{B}:=\mathcal{O}(\mathrm{Jet}^{\infty}(E))/\mathcal{I}_H),$ it is possible to show that there exists natural pairings on $H_{*,c}(X)\times h^{*}(\mathcal{B}^{\ell})(s,p)$ with values in $\Omega_{H}^p\otimes \wedge^s T_H,$ for $s,p\geq 0.$ Going through their construction is outside the scope of the paper, but can be obtained by following the ideas we  present here for $s=0,p>0$.

\begin{proposition}
\label{prop: Integration pairing}
Let $(\pi,H,\mathcal{B})$ be as above. Then for all $p\geq 0$ there exists a natural pairing
$$H_{*,c}(X)\times h^{*}(\mathcal{B}^{\ell})(p,0)\rightarrow \Omega_H^p,(K,\omega)\mapsto \lambda_{K,\omega},$$
where $\lambda_{K,\omega}$ is defined for each open subset $U$ of $X$ by its values on $U$-parameterized sections $s_U\in Sect_K(X,E)(U)\subseteq Hom(X\times U,E),$ with support in $K$, by 
$\lambda_{K,\omega}(s_U):=\int_{\sigma}(j_{\infty}(\varphi(x,u))^*\omega\in \Omega_U^p.$
\end{proposition}

If $Sol(\mathcal{I})\subset Sect(X,E)$ is the non-local solution space corresponding to a $\D$-smooth algebraic non-linear PDE (\ref{eqn: Local solution D space}) with $Sol(\mathcal{I})=\{s\in Sect(X,E)|j_{\infty}(s)^*u=0,u\in \mathcal{I}\},$ then  
$$H_{*,c}(X)\times h^{*-d_X}(\mathcal{B}^{\ell})\rightarrow \mathsf{Maps}(Sol(\mathcal{I}),\underline{k}),$$
is a well-defined integration pairing defined in the same way c.f.  (\ref{eqn: Integration pairing}).
\vspace{2mm}

\noindent\textit{\textbf{Variational structures on $Ran_X$.}}
For every $I\in \mathrm{fSet}$, with cardinality $|I|$ consider $X^I:=X\times\ldots\times X$ with projection onto the $i$'th factor as $p_i:X^I\rightarrow X.$
Suppose that $\mathrm{vol}_X(x_1)$ is a volume form around a point $x_1^{\mu}=(x_1^1,\ldots,x_1^n)$ in the first copy of $X$ in $X^I.$ Assume the same for $\mathrm{vol}_X(x_i^{\mu})$ for each $i=1,\ldots,|I|,$ and that
$x_i\neq x_j$.

The volume form on $X^I$ is
$$\mathrm{Vol}_{X^I}(x_1,x_2,\ldots,x_{|I|}):=p_1^*\mathrm{vol}_X(x_1)\wedge \ldots\wedge p_{|I|}^*\mathrm{vol}_X(x_{|I|}).$$
For instance, in the case of a complex curve $C$ we have a line bundle $\mathrm{det}(k^I)$, with an $\mathrm{Aut}(I)$-action via the sign character. There is a canonical $\mathrm{Aut}(I)$-invariant isomorphism 
$\omega_C^{\boxtimes I}\otimes\mathrm{det}(k^I)\rightarrow \omega_{C^I}$ given by $$\theta_1\boxtimes\ldots\boxtimes\theta_{|I|}\otimes(e_{i_1}\wedge\ldots\wedge e_{i_{|I|}})\mapsto p_{i_1}^*\theta_{i_1}\wedge \ldots \wedge p_{i_{|I|}}^*\theta_{i_{|I|}},$$
where $\theta_i\in\omega_C$ and $\{i_1,\ldots,i_{|I|}\}$ is a total ordering on $I$ with $\{e_i\}$ the standard basis of $k^I.$ This holds for more general varieties i.e. we may identify $\omega_{X^I}$ with
$\omega_X\boxtimes\ldots\boxtimes\omega_X.$

Suppose that $\mathcal{N}^{Ran}$ is a $\D_{Ran_X}$-module and $\mathcal{M}$ is a $\D_{Ran_X}$-module supported along the main diagonal. In particular, each $\mathcal{M}^{\boxtimes I}$ is a $\D_{X^I}$-module, and one considers its de Rham complex over $Ran_X$
\begin{equation}
    \label{eqn:de Rham for Ran D modules}
   DR_{\mathcal{M}}\big(\mathcal{N}^{Ran}\big):=\big\{\mathcal{N}^{(I)}\otimes_{\mathcal{D}_{X^I}}^{\mathbb{L}}\mathcal{M}^{\boxtimes I}\}_{I\in\mathrm{fSet}}.
\end{equation}
This works in the $(\infty,1)$-categorical setting. That is, (\ref{eqn:de Rham for Ran D modules}) defines an $\infty$-functor
$$\mathbf{DR}_{\mathcal{M}}(-):\mathsf{Mod}(\D_{Ran_X})\rightarrow \mathsf{Shv}(Ran_X),$$
to the $\infty$--category of (co)sheaves of vector spaces on $Ran_X.$

Using (\ref{eqn:de Rham for Ran D modules}) we may straightforwardly extend (\ref{Variational de Rham complex definition}) by setting
$$\mathsf{DR}_{Ran_X}^{\mathcal{D}}(\mathcal{N}^{Ran}):=\big\{(\mathcal{I}_{*,X^I}\otimes_{\mathcal{O}_{X^I}}\mathcal{D}_{X^I}[|I|d_X])\otimes_{\mathcal{D}_{X^I}}^{\mathbb{L}}\mathcal{N}^I\big\}_{I\in \mathrm{fSet}},$$
i.e. a collection of $\{\mathsf{DR}_{\mathcal{O}_{X^I}}^{\mathcal{D}}(\mathcal{N}^I)\}_{I\in fSet}.$
When $\mathcal{M}$ is a $\D_X$-module supported on the main diagonal then
$$\mathsf{DR}_{Ran_X}^{\mathcal{D}}(\mathcal{M})\simeq\bigg\{\big(\mathcal{I}_{*,X}^{\boxtimes |I|}\otimes_{\mathcal{O}_{X^I}}\mathcal{D}_{X^I}\big)[|I|\cdot\mathrm{dim}(X)]\otimes_{\mathcal{D}_{X^I}}^{\mathbb{L}}\mathcal{M}^{\boxtimes I}\bigg\}_{I\in fSet}.$$

\section{de Rham Algebras of Solution Spaces}
In the previous section we discussed de Rham algebras for objects of classical $\D_X$-prestacks $\mathrm{PStk}_X(\D_X)=\mathrm{Fun}(\mathrm{CAlg}_X(\D_X),\mathsf{Spc}).$ In this section we do it homotopy coherently (see \cite[Thm. 3.57]{KSYI}) based on the $(\infty,1)$-categorical equivalence
\begin{equation}
    \label{eqn: MainEquiv}
\mathsf{PreStk}_{X}(\D_X)\simeq \mathsf{PreStk}_{/X_{dR}}.
\end{equation}
It is related to the classical formulation by noticing the obvious fully-faithful embedding 
$$i:\mathrm{PStk}_{X}(\D_X)\rightarrow \mathbf{dPStk}_X(\D_X),$$
into a model category of derived prestacks has a right-adjoint Quillen functor $i^*$ so by considering its right-derived functor
$\mathbb{R}i^*:Ho(\mathbf{dPStk}_X(\D_X))\rightarrow Ho(\mathrm{PStk}_X(\D_X)),$ there is a fully-faithful left adjoint $\mathbb{L}i_!$ (via Kan extension) acting on classical $\D$-prestacks by
$$\mathbb{L}i_!Sol_{\mathcal{D}}(\mathcal{A})\simeq\mathbb{R}Sol_{\mathcal{D}}(i\mathcal{A}),$$
where $i\mathcal{A}$ the $\D$-algebra viewed as a derived object in cohomological degree zero.
\vspace{2mm}

\noindent\textit{\textbf{The `de Rham side' of }} (\ref{eqn: MainEquiv}). 
Let $p_{dR,X}:X\rightarrow X_{dR}$ and consider an $X_{dR}$-stack $\EQ$ such that its cotangent complex exists. Let $\mathbb{L}_{\EQ/X_{dR}}$ be the relative cotangent complex and consider the map $\eta:p_{dR,X}^*(\EQ)\rightarrow X$ and global functions $\mathcal{O}\big(p_{dR,X}^*\EQ\big).$ 
Consider the push-forward sheaf $\eta_*\mathcal{O}\big(p_{dR,X}^*\EQ\big)$. By pull-back, $p_{dR,X}^*\mathbb{L}_{\EQ/X_{dR}}$, is an object of the $(\infty,1)$-category
 $$\eta_*\mathcal{O}\big(p_{dR,X}^*\EQ\big)-\mathsf{Mod}(\D_X)\simeq \mathsf{Mod}(\mathcal{A}_X\otimes\D_X),$$
where we have put 
 $\mathcal{A}_X:=\eta_*\mathcal{O}(p_{dR}^*\EQ).$
 Making the necessary assumption that $\mathbb{L}_{\EQ/X_{dR}}$ is perfect and consequently dualizable, we may take its derived local Verdier dual (c.f. \ref{eqn: Derived Inner Verdier Dual}), given by
\begin{equation}
    \label{eqn: Derived D Tangent}
\mathbb{R}\mathsf{Maps}\big(p_{dR,X}^*\mathbb{L}_{\EQ/X_{dR}},\eta_*\mathcal{O}\big(p_{dR,X}^*\EQ\big)[\D_X]\big)\otimes \omega_X^{\otimes -1}.
\end{equation}
We take (\ref{eqn: Derived D Tangent}) as the definition of the derived
$\D$-geometric tangent complex $\mathbb{T}_{\EQ/X_{dR}}^{\ell}$ of a derived $X_{dR}$-stack $\EQ.$ 

The most important example arises via the derived $\infty$-jet prestack construction: given $\EuScript{E}\in \PS_{/X},$ there is a canonical prestack 
$\J
(\EuScript{E})\in \PS_{/X_{dR}},$ thus equipped with a canonical flat connection, and a pull-back diagram
\begin{equation}
\begin{tikzcd}
    \label{eqn: Jet-pb diagram}
    p_{dR}^*(\J(\EuScript{E})\big)\arrow[d]\arrow[r,"c_{dR}^{\EuScript{E}}"] & \J(\EuScript{E})\arrow[d]
    \\
    X\arrow[r,"p_{dR}"]& X_{dR}.
    \end{tikzcd}
\end{equation}
Such objects are interpreted as homotopy co-free derived ${\scalemath{.90}{\mathrm{N}}{\scalemath{.84}{\mathrm{L}}}{\scalemath{.88}{\mathrm{P}}{\scalemath{.84}{\mathrm{DE}}{\scalemath{.70}{\mathrm{S}}}}}}$. A generic derived non-linear PDE\footnote{Defined in \cite{KRY} and presented in further detail in \cite{KSYI}.}on sections of $E\rightarrow X$ is understood as a suitable (derived) closed sub-stack $\EQ$ of $\J(E)$.

From hereon-out we keep the following standing assumptions referred to by $(\star)$:
\begin{equation*}
(\star):\begin{cases}
\EuScript{E}\in \PS_{/X}^{laft} \textit{ admits deformation theory relative to } X, \textit{and } \RS_X(p_*\EuScript{E}) \textit{ is laft.}
\\
\textit{Every } 
i:\EQ\rightarrow \JetX(\EuScript{E}) \textit{ comes from such }\EuScript{E} \textit{ for which } \RS(\EQ) \textit{ is laft}.
\end{cases}
\end{equation*}

In \cite{KSYI} objects satisfying $(\star)$ together with the assumption that all cotangent complexes exist were said to be $\D$-finitary prestacks. Here existence of cotangents will be implicitly assumed alongside $(\star).$

Then, these conditions tell us that $\mathbb{L}_{\EuScript{E}/X}$ exists and  $p_{dR*}\EuScript{E}$ admits deformation theory relative to $X_{dR},$ (whose deformation theory is obviously trivial). Thus, a relative cotangent complex exists for $\JetX(\EuScript{E})$, as a pull-back of $\mathbb{L}_{p_*\EuScript{E}/X}.$ 
\vspace{1.5mm}

If $\RS_X(p_*\EuScript{E})$ admits deformation theory, and it does for $\D$-finitary ${\scalemath{.90}{\mathrm{N}}{\scalemath{.84}{\mathrm{L}}}{\scalemath{.88}{\mathrm{P}}{\scalemath{.84}{\mathrm{DE}}{\scalemath{.70}{\mathrm{S}}}}}}$, the induced morphism 
$$\RS_X(i):\RS_X(\EQ)\rightarrow \RS_X(\JetX(\EuScript{E})),$$
of laft prestacks possesses a relative cotangent complex,
$cofib\big(\RS_X(i)\big).$

It is also convenient to consider the pull-back diagram in $\PS_{/X}$ formed by the comonad $\JetX,$
\begin{equation}
    \label{eqn: Infinite prolongation}
    \begin{tikzcd}[row sep=large, column sep = large]
        & \EQ_X^{\infty}\arrow[d]\arrow[r,"i_{\infty}"] & \JetX\EuScript{E}\arrow[d,"\alpha"]
        \\
        \EQ_X\arrow[r,"\rho"] & \JetX\EQ\arrow[r,"\JetX(i_X)"] & \JetX\JetX(\EuScript{E})
    \end{tikzcd}
\end{equation}
where $\rho$ is the induced $\JetX$-coalgebra map. The object
$\EQ_X^{\infty}\in \PS_{/X},$ defined by (\ref{eqn: Infinite prolongation}) may be interpreted as a kind of infinite-prolongation.

\subsection{Finiteness Properties}
\label{ssec: Finiteness}
Homotopical finiteness via (eventual) co-connectivity is necessary for some constructions and required for most proofs.  
Denote by $\mathsf{CAlg}_X(\mathcal{D}_X)^{\leq -n},$ the $n$-coconnective derived $\mathcal{D}_X$-algebras. The various forgetful functors we use, summarizing some results discussed in \cite[Sect. 3.5]{KSYI}, is given by the diagram of functors of $(\infty,1)$-categories:
\[
\begin{tikzcd}
    \mathsf{CAlg}_X(\mathcal{D}_X)\arrow[d,"U"]\arrow[r,"\tau_{\mathcal{D}}^{\leq -n}"] & \mathsf{CAlg}_X(\mathcal{D}_X)^{\leq -n}\arrow[d,"U"]
    \\
    \mathsf{Mod}(\mathcal{D}_X)\arrow[d,"For_{\mathcal{D}}"] \arrow[r,"\tau_{\mathcal{D}}^{\leq -n}"] & \mathsf{Mod}(\mathcal{D}_X)^{\leq -n}\arrow[d,"For_{\mathcal{D}}"]
    \\
    \mathsf{Mod}(\mathcal{O}_X)\arrow[r,"\tau^{\leq -n}"]& \mathsf{Mod}(\mathcal{O}_X)^{\leq -n}
\end{tikzcd}
\]
via the truncations and the functors $U$ forgetting algebra structure and $For_{\mathcal{D}}$ forgetting the connection. Here we have put $\mathsf{Mod}(\mathcal{O}_X):=\mathsf{QCoh}(X)^{\leq 0}.$

Moreover the homotopy-coherent algebraic jet functor is the left-adjoint to the forgetful functor $U:\mathsf{CAlg}_X(\mathcal{D}_X)\rightarrow \mathsf{CAlg}_X(\mathcal{O}_X),$ with $\mathsf{CAlg}_X(\mathcal{O}_X):=Comm\big(\mathsf{QCoh}(X)^{\leq 0}\big).$
\subsubsection{Horizontal jet prestacks}
We now adapt the discussion of Subsect. \ref{sssec: Horizontal Jets} to the homotopical setting, making use of Proposition \ref{Multi-jet r-adjoint proposition}.

Namely, given $\EQ\in \PS_{X_{dR}},$ formation of multi-jets gives a factorization $\D$-space of horizontal jets $\overline{\mathsf{Jets}}^{\infty}(\EQ)^{Ran}$ such that for each $I\in fSet$ and $S$-point of $X_{dR}^I$, a map $x_I:S_{red}\rightarrow X^I$ lifting to a point of $\overline{\mathsf{Jets}}^{\infty}(\EQ)^I$ is equivalently the datum of map in
$$\mathsf{Maps}_{/X_{dR}}(\hat{D}_{x_I,dR}\times_{S_{dR}}S,\EQ).$$

We summarize some useful properties of horizontal jet prestacks.
\begin{proposition}
Consider a $\D$-prestack $\EQ\in \PS_{/X_{dR}}.$
Then, the following hold:
\begin{itemize}
\item There is a tautological equivalence $\overline{\mathsf{Jets}}^{\infty}(\EQ)^{[1]}\simeq \EQ$ in $\PS_{/X_{dR}};$
\item If $\EQ\in \PS_{/X_{dR}}$ is of the form $p_{dR*}E$ for some smooth prestack over $X$ then the prestack of horizontal jets is equivalent to $\mathsf{Jets}^{\infty}(E).$
\end{itemize}
By push-forward along $q^I:\overline{\mathsf{Jets}}^{\infty}(\EQ)^I\rightarrow X_{dR}^I$ there is a $\mathcal{D}_{X^I}$-algebra $\mathcal{F}(\EQ)^I:=q_*^I(\mathcal{O}_{\overline{\mathsf{Jets}}^{\infty}(\EQ)^I}).$
\end{proposition}
Finally, note that for all $I$, there exists an `evaluation' map in $\PS_{/X_{dR}^I}:$
$$X_{dR}^I\times \RS(\EQ)\rightarrow \overline{\mathsf{Jets}}^{\infty}(p_{dR*}E)^I.$$ 

\subsection{Tangent and Cotangent Complexes}
\label{ssec: Sectional de Rham}
We will initiate the description of de Rham algebras associated to a derived non-linear PDE $\EQ.$ For this, we need a description of the various tangent and cotangent complexes.

\begin{proposition}
\label{prop: Tangent RSol1}
    Suppose that $\EQ$ satisfies $(\star).$ Then, for all $U$-parameterized solutions $\varphi_U:U\rightarrow \RS_X(\EQ),$ we have that $$\mathbb{T}[\RS_X(\EQ)]_{\varphi_U}\simeq q_{U*}^{\IC}(j_{\infty}(\varphi_U)^!\EuScript{T}_{\EQ}).$$
\end{proposition}
\begin{proof}
Consider a $U$-parameterized solution $\varphi_U:U\rightarrow \RS_X(\EQ),$ which is equivalently a point $\varphi_U':U\times X_{dR} \rightarrow \EQ,$ that we denote by $j_{\infty}(\varphi_U):U\times X_{dR}\rightarrow \EQ.$ Recall it appears in the following diagram 
\begin{equation}
    \label{eqn: Point of RSol}
\begin{tikzcd}
    U\times X_{dR}\arrow[d,"q_U"]\arrow[r,"\varphi_U\times id_{X_{dR}}"] \arrow[rr, bend left, "j_{\infty}(\varphi_U)"] & \RS_X(\EQ)\times X_{dR}\arrow[d,"\rho"] \arrow[r,"ev_{\varphi_U}"] & \EQ
    \\
    U\arrow[r,"\varphi_U"] & \RS_X(\EQ) & 
\end{tikzcd}
\end{equation}
Suppose that $\EQ\rightarrow X_{dR}$ is a $\D$-finitary Artin stack.
One may show
$$\mathbb{L}\big[\RS_X(\EQ)/X_{dR}\big]\simeq D^{loc-ver}\big(\rho_*ev_{\varphi_U}^!\mathbb{T}_{\EQ/X_{dR}}\big).$$
Since $\mathbb{L}_{\EQ/X_{dR}}$ is assumed to be dualizable, for all solutions, $j_{\infty}(\varphi_U)^!\mathbb{L}_{\EQ/X_{dR}}$ is dualizable as well. Base-changing via the diagram gives us
$$(q_U)_*\big(j_{\infty}(\varphi_U)\big)^!\mathbb{T}_{\EQ/X_{dR}}\simeq \varphi_U^!\rho_*ev_{\varphi_U}^!\mathbb{T}_{\EQ/X_{dR}}.$$

From this diagram it is clear that since $\EQ$ is $\D$-finitary, both $\mathbb{L}_{\EQ}$ and $\mathbb{L}\big(\RS_X(\EQ)\big)$ exists as objects of $\mathsf{ProCoh}(\EQ),$ and for every solution we have
$$\mathbb{L}_{\RS_X(\EQ),\varphi_U}:=\varphi_U^!\mathbb{L}_{\RS_X(\EQ)}\simeq \varphi_U^! \rho_! ev_{\varphi_U}^!\mathbb{L}_{\EQ}\in \mathsf{ProCoh}(U).$$

Since $X$ is smooth, these pro-coherent complexes identify with ind-coherent duals e.g. $\mathbb{T}_{\EQ,\varphi_U}$ is given by 
$\mathbb{D}^{loc-ver}(ev_{\varphi_U}^!\mathbb{L}_{\EQ})$ with $\mathbb{D}^{loc-ver}$ the local Verdier duality,
$$\mathsf{ProCoh}(\RS_X(\EQ)\times X_{dR})\simeq \IC(\RS_X(\EQ)\times X_{dR})^{\op}.$$
Consequently, our desired object exists and is given by 
$$\mathbb{T}[\RS_X(\EQ)]_{\varphi_U}\simeq q_{U*}^{\IC}(j_{\infty}(\varphi_U)^!\EuScript{T}_{\EQ}).$$
\end{proof}
By a similar argument, one can describe $\mathbb{L}[\J(E)/X_{dR}]$ by considering the natural diagram
\begin{equation}
\label{eqn: Jets diagram}
\begin{tikzcd}
    U\times X \arrow[d,"(id_U\times p_{dR})"]\arrow[r,"\varphi_U\times id_{X}"] \arrow[rr, bend left, "\beta_{\varphi_U}"] & \J(\EuScript{E})\times_{X_{dR}}X\arrow[d,"(id\times p_{dR})=:\pi"] \arrow[r,"ev_{\varphi_U}"] & \EuScript{E}
    \\
    U\times X_{dR}\arrow[r,"j_{\infty}(\varphi_U)"] & \J(\EuScript{E}) & 
\end{tikzcd}
\end{equation}
Similar reasoning gives that
$$\mathbb{L}\big[\J(\EuScript{E})/X_{dR}\big]\simeq \big(id_{\J(\EuScript{E})}\times p_{dR}\big)_!\circ ev_{\varphi_U}^!\mathbb{L}_{\EuScript{E}/X},$$
as pro-coherent sheaves on $\J(\EuScript{E}).$
We can describe the situation as pro-coherent sheaves on $U\times X_{dR},$ via pull-back against an infinite jet of a solution i.e. a point $j_{\infty}(\varphi_U)\in \RS_X\big(\J(\EuScript{E})\big)(U)$.
\begin{proposition}
    \label{proposition: Pro-cotangent of Jets}
    For all solutions $\varphi_U$ parameterized by a laft affine $U\in \mathsf{dAff}_{/X}^{laft},$ there is an equivalence
$j_{\infty}(\varphi_U)^!\mathbb{L}_{\J(\EuScript{E})}\simeq (id_U\times p_{dR})_!\circ \beta_{U}^!\mathbb{L}_{\EuScript{E}/X}.$
\end{proposition}
By $(\star)$, we have $\RS_X(p_*(\EuScript{E})\big)$ is laft, with $\EuScript{E}$ admitting deformation theory relative to $X_{dR}$ so that $\RS_X(p_*\EuScript{E})$ admits deformation theory.

By our finitary assumptions, $\mathbb{L}_{\J(\EuScript{E})}$ is dualizable, so $j_{\infty}(\varphi_U)^!\mathbb{L}_{\J(\EuScript{E})}$ is dualizable as well.
\begin{proposition}
  The relative cotangent complex $\mathbb{L}_{\JetX(\EuScript{E})/X}$ is equivalent to the dual of $\pi^*\big(\pi_*ev_{\varphi_U}^!\mathbb{T}_{\EuScript{E}}).$
\end{proposition} 
For this, notice that
$$\mathbb{L}_{\JetX(\EuScript{E})}\simeq \mathbb{L}_{p^*\J(\EuScript{E})}\simeq \mathbb{L}_{\J(\EuScript{E})\times_{X_{dR}}X}\simeq \pi^*\big(\pi_*ev_{\varphi_U}^!\mathbb{T}_{\EuScript{E}}\big)^{\vee}.$$
Indeed, since $\JetX(\EuScript{E})$ is obtained via pull-back and diagram (\ref{eqn: Jets diagram}) induces a morphism
$$\alpha:c_{dR}^*\mathbb{L}_{p_*\EuScript{E}}\rightarrow \mathbb{L}_{\JetX(\EuScript{E})},$$
which exhibits $c_{dR}^*\mathbb{L}_{p_*\EuScript{E}}$ as the relative cotangent complex for $\mathbb{L}_{\JetX(\EuScript{E})},$ where $c_{dR}$ arises as in (\ref{eqn: Jet-pb diagram}). Note $c_{dR}:\JetX(\EuScript{E})\rightarrow p_*\EuScript{E}.$ Put $$\mathbb{L}_{\JetX(\EuScript{E})/p_*\EuScript{E}}:=Cone(\alpha).$$

\subsubsection{Functoriality statements.}
Given a morphism $g:U_1\rightarrow U_2$, first of all, set 
$$\mathbb{L}_{\J(\EuScript{E}),\varphi_i}:=j_{\infty}(\varphi_{U_i})^!\mathbb{L}_{\J(\EuScript{E})}\in\mathsf{ProCoh}(U_i\times X_{dR}) \hspace{1mm} i=1,2.$$
From (\ref{eqn: Jets diagram}) we have the induced diagram
\begin{equation}
\label{eqn: Jet-pb functorial cube}
\adjustbox{scale=.85}{
\begin{tikzcd}[row sep=scriptsize, column sep=scriptsize]
& U_1\times X\arrow[dl, "g"] \arrow[rr, "\varphi_{1}"] \arrow[rrr, bend left, "\beta_{\varphi_{U_1}}"] \arrow[dd, "q_1"] & &\J(\EuScript{E})\times_{X_{dR}}X\arrow[dl, ""] \arrow[dd, ""] \arrow[r,"ev_{\varphi_{U_1}}"] & \EuScript{E}\\
U_2\times X\arrow[rr, "\varphi_2"] \arrow[rrr, bend left, "\beta_{\varphi_{U_2}}"] \arrow[dd, "q_2"] & & \J(\EuScript{E})\times_{X_{dR}}X \arrow[r,"ev_{\varphi_{U_2}}"] & \EuScript{E}\\
& U_1\times X_{dR}\arrow[dl, "g "] \arrow[rr, "\hspace{7mm}j_{\infty}(\varphi_{U_1})"] & &  \J(\EuScript{E}) \arrow[dl, "id"] \\
U_2\times X_{dR} \arrow[rr, "j_{\infty}(\varphi_{U_2})"] & & \J(\EuScript{E})\arrow[from=uu, crossing over]\\
\end{tikzcd}}
\end{equation}
where we have set $q_i:=id_{U_i}\times p_{dR}^X:U_i\times X\rightarrow U_i\times X_{dR}$ for $i=1,2$ and put $\varphi_i:=\varphi_{U_i}\times id_X$ etc. for ease of notation.
It is obvious that $g:U_1\rightarrow U_2,$ induces functors 
$$(g\times id_{X_{dR}})^!:\IC(U_2\times X_{dR})\rightarrow \IC(U_1\times X_{dR}),$$ which moreover generate commutative diagrams via duality operations,
\[
\begin{tikzcd}
    \IC(U_2\times X_{dR})^{\op}\arrow[d]\arrow[r,"(g\times id)^{!\op}"] & \IC(U_1\times X_{dR})\arrow[d]
    \\
    \mathsf{ProCoh}(U_2\times X_{dR})\arrow[r,"(g\times id)_{\mathsf{Pro}}^!"]& \mathsf{ProCoh}(U_1\times X_{dR}).
\end{tikzcd}
\]
Together with (\ref{eqn: Jet-pb functorial cube}), we obtain a canonical map 
$$(g\times id)_{\mathsf{Pro}}^!\mathbb{L}_{\J(\EuScript{E}),\varphi_1}\rightarrow \mathbb{L}_{\J(\EuScript{E}),\varphi_2}.$$ 
Taking the limit over all such diagrams, we obtain a global/solution independent object, that we abusively denoted by
$$\mathbb{L}_{\J(\EuScript{E})}^{glob}:=\underset{g\in \mathsf{dAff}_{/\J(\EuScript{E})}^{\op}}{holim}\hspace{1mm} \mathbb{L}_{\J(\EuScript{E})} .$$

Now, since $\underline{\IC}^!$ is lax symmetric monoidal for the tensor product of sheaves $\mathcal{F}_1,\mathcal{F}_2\rightarrow \mathcal{F}_1\boxtimes\mathcal{F}_2$, denote by $\delta_U:U\times X_{dR}\rightarrow (U\times X_{dR})\times (U\times X_{dR}),$ the diagonal embedding.

\begin{proposition}
\label{proposition: 2-forms Jets}
We have an object $\mathbb{L}_
{\J(\EuScript{E}),\varphi_U}\otimes^!\mathbb{L}_{\J(\EuScript{E}),\varphi_U},$ in $\IC(U\times X_{dR})\otimes \IC(U\times X_{dR}),$
given by
$\delta_U^*\big(\mathbb{L}_
{\J(\EuScript{E}),\varphi_U}\boxtimes\mathbb{L}_
{\J(\EuScript{E}),\varphi_U}\big).$
\end{proposition}
By Proposition \ref{proposition: 2-forms Jets}, there exists an object, 
$$\mathbb{L}_{\J(\EuScript{E}),\varphi_U}^{\otimes^! p}\in \IC(U\times X_{dR})^{\otimes p},$$
which has very desirable properties in the following sense. On one hand it is generated by push-forwards of coherent objects in the parameter directions $U.$ Namely, if $i_Z:Z\hookrightarrow U$ is a closed embedding, denote by $j:U\backslash Z\hookrightarrow U$ the affine open compliment. Then there is a canonical pull-back
$$j_{\IC}^!:\IC(U\times X_{dR})\rightarrow \IC(U\backslash Z\times X_{dR}),$$
which is moreover $t$-exact and fully-faithful. It admits a continuous right-adjoint $j_*^{\IC}$ which preserves the sub-categories of compact objects.

On the other hand, it is stable under embedding in the de Rham directions in to smooth ambient spaces. Namely, if $\iota:X\hookrightarrow \overline{X}$ is a closed embedding into a smooth $\overline{X},$
there is a fully-faithful, $t$-exact functor 
$$\iota_*^{\IC}:\IC(U\times X_{dR})\rightarrow \IC(U\times\overline{X}_{dR}),$$
which admits a continuous right-adjoint and an object $\mathcal{M}\in \IC(U\times X_{dR})$ is compact if and only if $i_*^{\IC}(\mathcal{M})$ is compact.\footnote{Making use of Kashiwara's theorem.}

\subsubsection{Comonadic Extensions}
Consider (\ref{eqn: Infinite prolongation}) with all objects satisfying $(\star)$. Set $\rho^{\EQ}:=p^*p_*(i),$ as the comonadic extension of $i:\EQ\rightarrow \JetX(\EuScript{E}).$

Consider the jet-coalgebra comultiplication $\Delta_{\EuScript{E}}$ (\ref{eqn: Infinite prolongation}) and the associated cofiber sequence
$$\Delta_{\EuScript{E}}^!\mathbb{L}_{\JetX(\JetX\EuScript{E})}\rightarrow \mathbb{L}_{\JetX(\EuScript{E})}\rightarrow \mathbb{L}_{\Delta_{\EuScript{E}}}.$$
\begin{proposition}
Considering $\EQ_X^{\infty}$ and the space of maps in ind-coherent sheaves on jets gives an equivalence
$$\M\big(\Bigwedge^p\mathbb{T}_{\JetX(\EuScript{E})/\JetX\JetX(\EuScript{E})},\alpha^{!}\rho_!^{\EQ}\mathcal{O}_{p^*p_*\EQ_X}[n]\big)\simeq\M_{\EQ_X^{\infty}}(i_{\infty}^!\Bigwedge^p\mathbb{T}_{\Delta_{\EuScript{E}}},\mathcal{O}_{\EQ_X^{\infty}}[n]\big),$$
\end{proposition}
\begin{proof}
    By definition, we have by putting $\M_{\EQ_X^{\infty}}$ to denote the mapping space in \emph{pro-coherent} sheaves,
    \begin{eqnarray*}
        \mathcal{A}_{/p^*p_*(\EQ_X)}^p(\EQ_X^{\infty};n)&=&\M_{\EQ_X^{\infty}}\big(\mathcal{O}_{\EQ_X^{\infty}}[-n],\Bigwedge^p\mathbb{L}_{\EQ_X^{\infty}/p^*p_*\EQ_X}\big)
        \\
        &\simeq& \M_{\EQ_X^{\infty}}\big(\mathcal{O}_{\EQ_X^{\infty}}[-n],i_{\infty}^!\Bigwedge^p\mathbb{L}_{\JetX(\EuScript{E})/\JetX(\JetX(\EuScript{E})}\big)
    \end{eqnarray*}
    which then agrees, by using duality, with 
  $$\M_{\JetX(\JetX(\EuScript{E})}\big(\mathcal{O}_{\JetX(\JetX\EuScript{E}}[-n],\rho_!^{\EQ}\pi_!i_{\infty}^!\Bigwedge^p\mathbb{L}_{\JetX(\EuScript{E})/\JetX(\JetX\EuScript{E})}\big).$$
More exactly, since $\EQ_X$ is $\D$-finitary, so is $\EQ_X^{\infty}.$ By dualizing, which acts, for instance by 
\[
\begin{tikzcd}
    \IC(\J(\EuScript{E}))^{\op}\arrow[d]\arrow[r] & \IC(\EQ)^{\op}
    \arrow[d]
    \\
    \mathsf{ProCoh}(\J(\EuScript{E}))\arrow[r] & \mathsf{ProCoh}(\EQ).
\end{tikzcd}
\]
we obtain the result, noting that we then pass from pro-coherent to ind-coherent sheaves and the above is equivalently given via base-change by a morphism
$$\Bigwedge^p\mathbb{T}_{\Delta_{\EuScript{E}}}\rightarrow i_*^{\infty}\pi_{\infty}^!\mathcal{O}_{\EQ_X^{\infty}}[n]\simeq \Delta_{\EuScript{E}}^!\rho_!\mathcal{O}_{\EQ_X^{\infty}}[n].$$
\end{proof}

\subsubsection{}
A more delicate question concerns when $\EQ\in \PS_{/X_{dR}}$ is obtained from two (derived) differential operators acting on prestacks i.e. morphisms 
\begin{equation}
    \label{eqn: DiffMorphisms}
\mathsf{F}_1:p_{dR}^*p_{dR*}E_1\rightarrow Q,\hspace{2mm} \mathsf{F}_2:p_{dR}^*p_{dR*}E_2\rightarrow Q,
\end{equation}
with $E_1,E_2,Q\in \PS_{/X}.$ 
\begin{proposition}
\label{prop: Constrained tangents}
    Consider \emph{(\ref{eqn: DiffMorphisms})} and suppose that all objects satisfy $(\star).$ Set 
    $$\EQ:=\J E_1\times_{\J Q}\J E_2\rightarrow X_{dR}.$$ Then for all $U$-parameterised solutions $\varphi_U=(\varphi_1,\varphi_2)_U$ with $\varphi_i$ the component of the section into $\J(E_i),i=1,2,$ we have that the cotangent complex of $\RS(\EQ)$ is equivalent to  
    $$ j_{\infty}(\varphi_1,\varphi_2)_U^!\big(pr_1^!\big(\big(id_{\J(\EuScript{E})}\times p_{dR}\big)_!\circ ev_{\varphi_1}^!\mathbb{L}_{E_1/X}\big)\oplus pr_2^!\big(\big(id_{\J(E_2)}\times p_{dR}\big)_!\circ ev_{\varphi_2}^!\mathbb{L}_{E_2/X}\big)\big).$$
\end{proposition}
\begin{proof}
 First note that $\EQ$ is obtained by the $(p_{dR}^*,p_{dR*})$-adjunction on morphisms (\ref{eqn: DiffMorphisms}). That is,  we obtain maps 
$p_*\mathsf{F}_i:p_*E_i\rightarrow p_* Q,i=1,2$ and may consider the homotopy fiber product in $\PS_{/X_{dR}}:$
\[
\begin{tikzcd}
    \EQ\simeq \J E_1\times_{\J Q}\J E_2\arrow[d,"pr_1"] \arrow[r,"pr_2"]\arrow[dr,"\eta"] & \J(E_2)\arrow[d,"p_*\mathsf{F}_2"]
    \\
    \J(E_1)\arrow[r,"p_*\mathsf{F}_1"] & \J(Q),
\end{tikzcd}
\]
and with $\pi:=q\circ \eta:\EQ\rightarrow \J(Q)\xrightarrow{q}X_{dR}.$ Then, there is a diagram,
\[
\begin{tikzcd}
    \mathbb{T}_{\EQ/X_{dR}}\arrow[d]\arrow[r] & pr_1^!\mathbb{T}_{\J(E_1)/X_{dR}}\oplus pr_2^!\mathbb{T}_{\J(E_2)}\arrow[d]\arrow[r] & \eta^!\mathbb{T}_{\J(Q)/X_{dR}}\arrow[d]
    \\
    \mathbb{T}_{\EQ}\arrow[d]\arrow[r] & pr_1^*\mathbb{L}_{\J(E_1)}\oplus pr_2^*\mathbb{T}_{\J(E_2)} \arrow[d]\arrow[r] & \eta^*\mathbb{T}_{\J(Q)}\arrow[d]
    \\
    \pi^*\mathbb{T}_{X}\arrow[r] & \pi^*\mathbb{T}_X\oplus \pi^*\mathbb{T}_X\arrow[r] & \pi^*\mathbb{T}_X
\end{tikzcd}
\]
where we may view relative to $X_{dR}$ over $X$ instead, as it follows from the fact that $0\simeq p_{X,dR}^*\mathbb{L}_{X_{dR}}\rightarrow \mathbb{L}_{X}\rightarrow \mathbb{L}_{X/X_{dR}},$ so $\mathbb{L}_{X}\simeq \mathbb{L}_{X/X_{dR}}.$ Then, one may observe the result follows from analog diagram (\ref{eqn: Point of RSol}) via an application of Proposition \ref{prop: Tangent RSol1}.  
Namely, consider a $U$-point of $$\RS_X(\EQ)\simeq \mathsf{Maps}_{/X_{dR}}(X_{dR},\J(E_1)\times_{\J(Q)}\J(E_2)),$$ with $\rho:\RS_X(\EQ)\times X_{dR}\rightarrow \RS_X(\EQ),$ one has $\varphi_U^! \rho_! ev_{\varphi_U}^!\mathbb{L}_{\EQ},$ from which we may see
\begin{eqnarray*}
\mathbb{L}(\RS_X(\EQ))&\simeq & j_{\infty}^!(\varphi_U)\mathbb{L}(\J E_1\times_{\J Q}\J E_2)
\\
&\simeq& j_{\infty}(\varphi_U)^!\big(pr_1^!\mathbb{L}_{\J E_1}\oplus pr_2^!\mathbb{L}_{\J E_2}\big).
\end{eqnarray*}
\end{proof}
The question of whether $\EQ$ inherits a local shifted structure will therefore depend on $E_i,i=1,2$ and $Q$ as well as properties of the morphisms $\mathsf{F}_i,$ for instance, if they were Lagrangian in a suitable differential sense. 
We illustrate this situation with an example. 
\subsubsection{Framework for Hamilton-Jacobi systems.}
Fix a derived Artin stack $X$ and put $\gamma:E:=X\times \mathbb{A}^1\rightarrow \mathbb{A}^1,$ where $\mathbb{A}^1$ is viewed as a derived stack with $\mathcal{O}(\mathbb{A}^1)=k[t].$
Consider the $n$-shifted cotangent stack $\mathsf{T}^*[n]X$, together with its usual $n$-shifted (Liouville) $1$-form $\lambda$ and non-degenerate shifted symplectic form $\omega=d_{dR}\lambda.$
Consider two pull-back diagrams in derived stacks:
\begin{equation}
\label{eqn: 1-jets}
\begin{tikzcd}
    \mathsf{Jets}_{\mathbb{A}^1}^{1}[n]X\arrow[d,"p_1"] \arrow[r,"p_2"] & \mathbb{A}^1[n]\arrow[d]
    \\
    \mathsf{T}^*[n]X\arrow[r] & *
\end{tikzcd}\hspace{2mm},\begin{tikzcd}
     \mathsf{Jets}_{\mathbb{A}^1}^{\dagger}[n]X\arrow[d,"q_1"] \arrow[r,"q_2"] &\mathbb{A}^1[n]\arrow[d]
    \\
    \mathsf{T}[n]X\arrow[r] & *
\end{tikzcd}
\end{equation}
From (\ref{eqn: 1-jets}) we have natural cofiber sequences 
$$p_1^*\mathbb{T}_{\mathsf{T}^*[n]X}\rightarrow \mathbb{T}_{\JN}\rightarrow p_2^*\mathbb{T}_{\mathbb{A}^1[n]},\hspace{1mm}\text{ and }\hspace{.2mm} q_1^*\mathbb{T}_{\mathsf{T}[n]X}\rightarrow \mathbb{T}_{\Jn}\rightarrow q_2^*\mathbb{T}_{\mathbb{A}^1[n]}.$$
Thus it follows that
$$\mathbb{L}_{\JN}\simeq p_1^*\mathbb{L}_{\mathsf{T}^*[n]X}\oplus p_2^*\mathbb{L}_{\mathbb{A}^1[n]},\hspace{1mm}\text{ and }\mathbb{T}_{\JN}\simeq p_1^*\mathbb{T}_{\mathsf{T}^*[n]X}\oplus p_2^*\mathbb{T}_{\mathbb{A}^1[n]}.$$
Similarly,
$$\mathbb{L}_{\Jn}\simeq q_1^*\mathbb{L}_{\mathsf{T}[n]X}\oplus q_2^*\mathbb{L}_{\mathbb{A}^1[n]},\hspace{1mm}\text{ and }\hspace{.2mm}\mathbb{T}_{\Jn}\simeq q_1^*\mathbb{T}_{\mathsf{T}[n]X}\oplus q_2^*\mathbb{T}_{\mathbb{A}^1[n]}.$$
From these considerations it is clear that the $n$-shifted $1$-form $\lambda$ induces a canonical $n$-shifted $1$-form i.e. a section $\vartheta\in\Gamma\big(\JN,\mathbb{L}_{\JN}\big),$ given locally by 
$\vartheta=p_1^*\lambda -d_{dR}(p_2^*t),$ where $t$ is of degree $n$ as the coordinate on the shifted affine line.

Consider the pull-back diagram in derived stacks:
\begin{equation}
\label{eqn: Legendre Diagram}
\begin{tikzcd}
    \JN\times_{E}^h\Jn\arrow[d,"pr_1"] \arrow[r,"pr_2"] & \Jn\arrow[d,"\eta"]
    \\
    \JN\arrow[r,"\rho"] & E[n]  
\end{tikzcd}
\end{equation}
Composition with standard maps $\pi:\mathsf{T}^*[n]X\rightarrow X$ and $\tau:\mathsf{T}[n]X\rightarrow X,$ gives
$\JN\rightarrow \mathsf{T}^*[n]X\rightarrow X$ and $\Jn\rightarrow \mathsf{T}[n]X\rightarrow X.$

Take mapping spaces of sections in derived stacks of diagram (\ref{eqn: Legendre Diagram}), a section $$s=(w,v)\in \mathsf{Maps}(E,\JN\times_E\Jn),$$ corresponds to $w\in \mathsf{Maps}(E,\JN)$ and $v\in \mathsf{Maps}(E,\Jn)$.
For every $L\in \mathcal{O}(\Jn),$ a correspondence in derived stacks:
\[
\begin{tikzcd}
    & \arrow[dl] \Gamma_L\arrow[dr] & 
    \\
    \Jn & & \JN
\end{tikzcd}
\]
together with a map $\iota_L:\Gamma_L\rightarrow \JN\times\Jn,$ subject to additional properties (e.g. isotropic with additional non-degeneracy conditions corresponding to the shifted symplectic structure on $1$-jet stacks) gives a general framework to pose generalized Hamilton-Jacobi problems. Namely, such a correspondence acts like a kind of Legendre transform. Indeed, classically one might take for example $\Gamma_L=\{(t,x,x_1,\partial_{x_1}L)\}$ with $x_1=\partial_tx.$ 
A section $s=(v,w)$ as above solves the HJ-problem if there is a homotopy $i_v s^*\lambda \simeq s^*d_{dR}H_L,$ for some $H_L\in \mathcal{O}(\JN),$ for which $s$ factors through $\iota_L$ i.e. $s$ determines a point in $\mathsf{Maps}(E,\Gamma_L).$

The question of non-degeneracy then relies on the existence of underlying shifted symplectic or Poissons structures. 
In the current setting, this may be seen by applying the functor $p_{dR*}$ to (\ref{eqn: Legendre Diagram}) puts us in the setting of Proposition \ref{prop: Constrained tangents} and if $X$ has a relative $d$-orientation (see \cite{PTVV}), then the relative (to $\Jn$) mapping space of solutions inherits a $(d-1)$-shifted symplectic structure, by pull-push $\vartheta$ along the span:
    \[
    \begin{tikzcd}
& \arrow[dl,"\pi"] \RS_X(\EQ)\times X_{dR} \arrow[dr,"ev"] & 
\\
\RS_X(\EQ)& & \EQ.
\end{tikzcd}
\]

This is a particular instance of a more general $\D$-geometric transgression procedure, that we do not go into detail with in this work. We conclude with one observation in this direction.
\begin{observation}
\normalfont
Consider $\EQ\rightarrow X_{dR}\in \mathsf{PreStk}_X(\mathcal{D}_X),$ satisfying $(\star)$ with moduli of solutions $\RS_X(\EQ).$ If $\EQ$ is endowed with a shifted form $\omega$ then there is an induced form $\omega_{\RS}$ via transgression. 
For instance, representing
$$\mathbb{T}_{\varphi_U}\RS_X(\EQ)\simeq Ra_*\big(U\times X_{dR},j_{\infty}(\varphi_U)^!\mathbb{T}_{\EQ}\big),$$
if $\EQ$ has an $n$-shifted $2$-form $\theta_{\EQ}$, it induces a pairing
$$\Theta_{\EQ}:\mathbb{T}_{\EQ}\wedge \mathbb{T}_{\EQ}\rightarrow \mathcal{O}_{\EQ}[n],$$
which in turn defines a morphism 
$$Ra_*(U\times X_{dR},j_{\infty}(\varphi_U)^*\mathbb{T}_{\EQ})\wedge Ra_*(U\times X_{dR},j_{\infty}(\varphi_U)^*\mathbb{T}_{\EQ})\rightarrow Ra_*(U\times X_{dR},\mathcal{O}_{U\times X_{dR}}[n]).$$
Roughly speaking, due to the equivalence of functors,
\begin{eqnarray*}
Ra_*(U\times X_{dR},-)&\simeq& \mathsf{Maps}_{\IC(U\times X_{dR})}(\mathcal{O}_{U\times X_{dR}},-)
\\
&\simeq& \M_{\IC(U)\times \IC(X_{dR})}(\mathcal{O}_U\boxtimes \mathcal{O}_{X_{dR}}^r,-)
\\
&\simeq& \M_{\Q(U)\otimes \IC(X_{dR})}(\mathcal{O}_U\boxtimes \omega_X,-)
\\
&\simeq& R\Gamma^{\IC}(U,-)\boxtimes R\Gamma_{dR}(X,-),
\end{eqnarray*}
where the second factor is the $\mathcal{D}$-module de Rham cohomlogy i.e .
$$\Gamma_{dR}(X,\mathcal{M}^{\ell})=\Gamma(X_{dR},\mathcal{M}^{\ell})\simeq \M_{\Q(X_{dR})}(\mathcal{O}_X,\mathcal{M}^{\ell}),$$ which by passing to $\D^{op}$-modules, is of the form $\mathsf{Maps}_{\IC(X_{dR})}(\omega_X,\mathcal{N}),$ where $\mathcal{M}^{\ell}$ is obtained from $\mathcal{N}$ by (the inverse of) a natural functor\footnote{In \cite{GaiRozCrys}, it is denoted $\Upsilon_{X_{dR}}.$} $\IC(X_{dR})\rightarrow \Q(X_{dR})$.
\end{observation}

 \begin{example}
 \normalfont 
Let $C$ be an elliptic curve and $X$ a $d$-oriented derived Artin stack with $\pi:E=\mathsf{T}^*X\times C\rightarrow C$ the natural projection and consider 
    $\EQ:=\J(\mathsf{T}^*X).$ Then $\EQ$ is equipped with a Poisson structure, determined by a bi-vector field
$\Pi\in \mathbb{T}_{\EQ}\wedge^{\star}\mathbb{T}_{\EQ},$ on $RanX.$
It follows that the corresponding de Rham cohomology $\Gamma_{dR}(X,ev_X^*\mathbb{T}_{\EQ}),$ identifies this Poisson structure with the symplectic form on $\M(C,\mathsf{T}^*X),$ by the description of
$\mathbb{T}[\underline{\M}(C,\mathsf{T}^*X)]$ (see \cite{PTVV}).
\end{example}

\section{$\mathcal{D}$-Geometric Derived Variational de Rham Algebra}
\label{sec: D-Geometric DerVar}
In this section we exploit (\ref{eqn: MainEquiv}) to give an more explicit presentation of the de Rham algebras over derived affine $\mathcal{D}$-spaces $Spec_{\mathcal{D}_X}(\mathcal{A}).$

To this end, note the fact that if $\EQ\rightarrow X$ is affine, then so is $\mathrm{Jets}_{dR}^{\infty}(\EQ)\rightarrow X_{dR}$ and there is an identification $\EQ\simeq Spec_{\mathcal{D}}(\mathcal{A}^{\ell}),$ with $\mathcal{A}^{\ell}\in \mathsf{CAlg}_X(\D_X)$ some eventually co-connective \cite{KSYI} derived $\D$-algebra.

\subsection{Shifted Symplectic $\mathcal{D}$-Geometry.}
\label{ssec: Shifted Symplectic Derived D Geometry}
The construction of $\mathbb{L}_{\mathcal{A}^{\ell}}$ is compatible with localization of $X$ and for $j\in \mathbb{Z}$, one has $\mathcal{H}_{\mathcal{D}}^j(\mathbb{L}_{\mathcal{A}^{\ell}})\in \mathrm{Mod}(\mathcal{H}_{\mathcal{D}}^0(\mathcal{A}^{\ell})[\mathcal{D}_X]).$ They are also compatible with étale localization of $Spec_{\mathcal{D}}(\mathcal{A}^{\ell}).$  
\begin{remark}
    If $X$ is quasi-projective one may use a global semi-free resolution of $\mathcal{A}^{\ell}$ to define $\mathbb{L}_{\mathcal{A}^{\ell}}\in \mathrm{D}(\mathcal{A}^{\ell}[\mathcal{D}_X]),$ i.e. as an object of the derived category and exploit naturally defined local operations (c.f. \ref{eqn: LocalOps}).
\end{remark}
We may consider the usual derived de Rham algebra construction adapted to our $\mathcal{D}$-module setting to take values in the complex of integral forms.
\begin{theorem}[\textcolor{blue}{Theorem A (i)}]
\label{prop: DRVar is a prestack}
Let $X$ be smooth of dimension $d_X.$ Assigning a graded-mixed algebra object in bi-complexes of $\D_X$-modules to a derived $\D_X$-algebra $\mathcal{A}=(\mathcal{A}^{\bullet},\delta)$ via
$$\mathcal{A}^{\bullet}\mapsto \mathbf{DR}_{var}^*(\mathcal{A}^{\bullet})[n-p]\simeq \underbrace{\bigoplus_{q\geq 0}\mathcal{I}_{X,*}\otimes\mathcal{D}_X[d_X]\otimes_{\mathcal{D}}^{\mathbb{L}}\big(\Bigwedge_{\mathcal{A}}^q\mathbb{L}_{\mathcal{A}}\big)[q-p+n]}_{\mathcal{V}ar^q(\mathcal{A}^{\bullet},n)},$$
is local on $X,$ satisfies $\mathcal{D}$-étale descent, and is compatible with truncation i.e. the natural morphism $$\mathcal{V}ar^q(\mathcal{A}^{\bullet},n)\rightarrow holim_n\hspace{.3mm}\mathcal{V}ar^q(\tau_{\mathcal{D}}^{\leq -n}\mathcal{A}^{\bullet}),$$ with $\tau_{\mathcal{D}}^{\leq -n}$ as in \emph{(\ref{ssec: Finiteness})} is an equivalence for every $q\geq 0.$
\end{theorem}
\begin{proof}
    To see that this assignment defines a sheaf of spaces on the $\infty$-site of affine derived $\D$-schemes with respect to the $\D$-étale topology
    it is enough to see remark that $\mathcal{U}\rightarrow \Bigwedge^p\mathbb{L}_{\mathcal{U}}$ satisfies $\D$-étale descent.
It indeed defines an $\infty$-functor by Proposition \ref{prop: Sym preserves W} since the underlying bi-graded algebra of $\mathbf{DR}_{var}(\mathcal{A}^{\bullet})$ is $DR_X\big(Sym_{\mathcal{A}}(\mathbb{L}_{\mathcal{A}}[-1])(-1)\big),$ 
and put $\mathcal{V}ar^q(\mathcal{A})\simeq DR_X\big(Sym_{\mathcal{A}}^q(\mathbb{L}_{\mathcal{A}}[-1])\big).$

For each $n$ we set $f_n:\mathcal{A}\rightarrow \tau_{\mathcal{D}}^{\leq -n}\mathcal{A},$ and must prove that the canonical map
$$\nu_q:\mathcal{V}ar^q(\mathcal{A})\rightarrow \underset{n}{holim}\hspace{.5mm} DR_X\big(Sym_{\tau_{\mathcal{D}}^{\leq -n}\mathcal{A}}^q\big(\mathbb{L}_{\tau_{\mathcal{D}}^{\leq -n}\mathcal{A}}[-1]\big)\big),$$
is an equivalence for all $q\geq 0.$
For this, use the natural sequence associated to $f_n,$ 
$$\mathbb{L}_{\tau_{\mathcal{D}}^{\leq -n}\mathcal{A}/\mathcal{A}}[-1]\rightarrow f_n^*\mathbb{L}_{\mathcal{A}}\rightarrow \mathbb{L}_{\tau_{\mathcal{D}}^{\leq -n}\mathcal{A}},$$
obtained from the usual cofiber sequence in $\mathcal{A}^{\bullet}\otimes \mathcal{D}_X$-modules. 
There is an induced sequence
$$\mathcal{V}ar^q(\mathcal{A})\rightarrow \underset{n}{holim}\hspace{.5mm}DR_X\big(Sym_{\tau_{\mathcal{D}}^{\leq -n}\mathcal{A}}^q\big(f_n^*\mathbb{L}_{\mathcal{A}}[-1]\big)\big)\rightarrow\underset{n}{holim}\hspace{.5mm} \mathcal{V}ar^q(\tau_{\mathcal{D}}^{\leq -n}\mathcal{A}),$$
such that $DR_X\big(Sym_{\tau_{\mathcal{D}}^{\leq -n}\mathcal{A}}^p\big(f_n^*\mathbb{L}_{\mathcal{A}}[-1]\big)\big)$ inherits a natural filtration whose associated graded pieces are isomorphic to 
$$DR_X\big(Sym_{\tau_{\mathcal{D}}^{\leq -n}\mathcal{A}}^q\big(\mathbb{L}_{\tau_{\mathcal{D}}^{\leq -n}\mathcal{A}/\mathcal{A}}[-2]\big)\otimes Sym_{\tau_{\mathcal{D}}^{\leq -n}\mathcal{A}}^q(\mathbb{L}_{\tau_{\mathcal{D}}^{\leq -n}\mathcal{A}}[-1])\big).$$
We will have succeeded in proving the claim if we can show that 
$$\underset{n}{holim}\hspace{.5mm}\bigg(DR_X\big(Sym_{\tau_{\mathcal{D}}^{\leq -n}\mathcal{A}}^q\big(f_n^*\mathbb{L}_{\mathcal{A}}[-1]\big)\big)\rightarrow \mathcal{V}ar^q(\tau_{\mathcal{D}}^{\leq -n}\mathcal{A})\bigg)\simeq 0.$$
But since we have equivalences 
$$\underset{n}{holim}\hspace{.5mm} \mathcal{V}ar^q(\tau_{\mathcal{D}}^{\leq -n}\mathcal{A})\simeq DR_X\big(\underset{n}{holim}\hspace{.5mm} Sym_{\tau_{\mathcal{D}}^{\leq -n}\mathcal{A}}^q(\mathbb{L}_{\tau_{\mathcal{D}}^{\leq -n}\mathcal{A}}[-1])\big),$$
for each $q\geq 0,$ it is enough to prove 
$$\underset{n}{holim}\bigg(Sym_{\tau_{\mathcal{D}}^{\leq -n}\mathcal{A}}^q\big(f_n^*\mathbb{L}_{\mathcal{A}}[-1]\big)\rightarrow Sym_{\tau_{\mathcal{D}}^{\leq -n}\mathcal{A}}^q\big(\mathbb{L}_{\tau_{\mathcal{D}}^{\leq -n}\mathcal{A}}[-1]\big)\bigg)\simeq 0.$$
This follows from a standard connectivity argument.

    \end{proof}
Following as in \cite{PTVV}, taking realizations of the objects in Proposition \ref{prop: DRVar is a prestack}, we obtain a simplicial set $\mathcal{V}ar_{\Delta}^p(\mathcal{A}).$
\begin{proposition}
    \label{prop: SimplicialVar}
   There is a well-defined sheaf of spaces given by the functor
\begin{eqnarray*}
    \mathsf{Var}^p(n)(-):(\mathsf{dAff}_X(\mathcal{D}_X))^{op}&\rightarrow& \mathsf{Spc},
    \\
   Spec_{\mathcal{D}}(\mathcal{A}^{\bullet})&\mapsto& \mathcal{V}ar_{\Delta}^p(\mathcal{A}^{\bullet},n):=\mathsf{Var}^p(n)(Spec_{\mathcal{D}}(\mathcal{A}^{\bullet})).
   \end{eqnarray*} 
\end{proposition}
In what follows we will choose not distinguish with the notations $\mathcal{V}ar^p,\mathcal{V}ar_{\Delta}^p,$ hoping the context will be clear.
This $\D$-stack modulates $n$-shifted $p$-forms, via Proposition \ref{prop: DRVar is a prestack} and Proposition \ref{prop: SimplicialVar} we may put
$$\mathcal{V}ar^p(\EQ,n)=\mathsf{Maps}_{\mathsf{DStk}_X(\mathcal{D}_X)}\big(\EQ,\mathsf{Var}^p(n)\big),$$
for a general $\EQ$ in $\mathsf{PreStk}_X(\mathcal{D}_X).$ Analogous to \cite{PTVV}, one has an equivalence
\begin{equation}
    \label{eqn: Variational n-shifted p-forms}
\mathcal{V}\mathrm{ar}^p\big(\EQ,n\big)\simeq \mathsf{Maps}\big(\mathcal{O}_{\EQ},\mathbf{DR}^{\mathrm{var}}(\mathcal{O}_{\EQ})[n](p)\big).
\end{equation}
The space (\ref{eqn: Variational n-shifted p-forms}) is called the classifying space of variational $n$-shifted $p$-forms.
Note that if $\{f_{\alpha}:\bigcup_{\alpha}Spec_{\mathcal{D}}(\mathcal{B}_{\alpha})\rightarrow \EQ\}$ is an atlas for $\EQ$, we have
$$\mathbb{L}_{\EQ}\otimes_{\mathcal{O}_{\EQ}}^{\mathbb{L}}\mathcal{B}_{\alpha}\rightarrow \mathbb{L}_{\mathcal{B}_{\alpha}}\rightarrow \mathbb{L}_{\mathcal{B}_{\alpha}/\mathcal{O}_{\EQ}},$$
and since $\mathsf{Maps}$ in $\mathsf{DStk}_X(\D_X)$ is defined by right Kan extension from affine $\D$-spaces, one has
\begin{eqnarray*}
\mathsf{Maps}(\EQ,\mathsf{Var}^p(n)\big)&\simeq& \M(\underset{Spec_{\mathcal{D}}(\mathcal{B}_{\alpha})\rightarrow \EQ}{hocolim}Spec_{\mathcal{D}}(\mathcal{B}_{\alpha}),\mathsf{Var}^p(n)\big)
\\
&\simeq& \underset{Spec_{\mathcal{D}}(\mathcal{B}_{\alpha})\rightarrow \EQ}{holim}\M(Spec_{\mathcal{D}}(\mathcal{B}_{\alpha}),\mathsf{Var}^p(n)\big)
\\
&\simeq& \underset{Spec_{\mathcal{D}}(\mathcal{B}_{\alpha})\rightarrow \EQ}{holim} \mathcal{V}ar^p(Spec_{\mathcal{D}}(\mathcal{B}_{\alpha}),n).
\end{eqnarray*}
\begin{proposition}
\label{Theorem A}
Suppose that $\EQ=Spec_{\mathcal{D}_X}(\mathcal{A}^{\bullet})$ is an affine derived $\D_X$-prestack which is homotopically finitely $\D$-presented. Then the space \emph{(\ref{eqn: Variational n-shifted p-forms})} is equivalent to the mapping space in $\mathcal{A}^{\ell}[\mathcal{D}]$-modules:
$$\mathsf{Maps}\big(\mathcal{A}^{\ell},\big(\mathcal{I}_{X,*}\otimes\mathcal{D}_X[d_X]\big)\otimes_{\mathcal{D}_X}\Bigwedge_{\mathcal{A}^{\ell}}^p\mathbb{L}_{\mathcal{A}^{\ell}}[n]\big).$$
\end{proposition}
\begin{proof}
Follows from a standard description of the mapping space of sections $\mathsf{Maps}(\EQ,\mathbb{V}_{\EQ}(\mathcal{M})),$ of any vector $\mathcal{D}_X$-space over $\EQ$ as in \ref{sssec: Horizontal Jets} and Proposition \ref{prop: Vector D Stack}.
Indeed, observe that for any $\pi:\mathbb{V}_{\EQ}(\mathcal{M})\rightarrow \EQ$ with $\mathcal{M}$ perfect and therefore dualizable as an $\mathcal{O}_{\EQ}[\mathcal{D}_X]$-module, the mapping space in derived $\mathcal{D}_X$-stacks over $\EQ$ is described for any $f:\EuScript{X}\rightarrow \mathbb{V}_{\EQ}(\mathcal{M}),$ 
by putting $\gamma=\pi\circ f,$ as
$$\mathsf{Maps}_{\mathsf{DStk}_X(\mathcal{D}_X)_{/\EQ}}(\EuScript{X},\mathbb{V}_{\EQ}(\mathcal{M}))\simeq \mathsf{Maps}_{\mathcal{O}_{\EuScript{X}}}(\mathcal{O}_{\EuScript{X}},\gamma^*\mathcal{M}).$$
Indeed, this can be seen from the equivalence
$$\mathsf{Maps}_{\mathsf{DStk}_X(\mathcal{D}_X)_{/\EQ}}(\EuScript{X},\mathbb{V}_{\EQ}(\mathcal{M}))\simeq \mathsf{Maps}_{\mathsf{CAlg}(\mathcal{D}_X)_{\mathcal{O}_{\EQ}/}}(Sym_{\mathcal{O}_{\EQ}}(\mathcal{M}^{\vee}),\gamma_*\mathcal{O}_{\EuScript{X}}),$$
and using standard adjunctions (\ref{eqn: Free-forget AD-Mod/Alg Adjunction}).
\end{proof}

We may define closed forms via weighted negative cyclic homology of graded mixed complexes $E^{\bullet}(\star)$ with their natural map $NC^w(E^{\bullet}(\star))\rightarrow E^{\bullet}(\star),$ via projection onto the first component \cite{PTVV}. We define them as recalled in Subsect.\ref{Notations and Conventions} as cocycles with respect to the corresponding total differential.

The difference for us is rather than considering $\prod_{i\geq 0}\Bigwedge^{p+i}\mathbb{L}_A[n-i]$ as the complex whose underlying graded vector space is the product over all $i\geq 0$ of the graded vector spaces $\Bigwedge^{p+i}\mathbb{L}_A[n-i]$ with differential $d_{dR}+d_A,$ we have a \emph{bi-complex} structure. On $k$-forms of cohomological degree $n$, it reads as
\begin{equation}
    \label{eqn: Tri-complex decomposition}
    \mathcal{V}ar^k(\EQ,n)\simeq\bigoplus_{r,s}\mathcal{V}ar^{(r,s)}(\EQ,n)\simeq \prod_{r\geq 0,s\geq 0}\Gamma\big(\EQ,p_{\infty}^*\Omega_{X}^{p+r}\otimes\Bigwedge^{s+q}\mathbb{L}_{\EQ/X}^{\bullet}[n-(s+r)]\big).
\end{equation}
For simplicity, we record (\ref{eqn: Tri-complex decomposition}) together with its total differential $D=d_h+d_v+\delta,$ with $\delta$ the cohomological differential as a definition. 
Recall that the weight $s$-piece of the variational de Rham algebra (c.f. \ref{eqn: DRVAR weight s}) applied to Proposition \ref{prop: DRVar is a prestack} is $\mathbf{DR}_{var}^{*}(\mathcal{O}_{\EQ})(s).$

\begin{definition}
    \label{definition: n-shifted (p,q)-forms}
    \normalfont The complex of \emph{variational $(p,q)$-forms on $\EQ$} is 
    $$\mathcal{V}ar^{(p,q)}(\EQ)\simeq \big(\prod_{r\geq p,s\geq q}\mathbf{DR}_{var}^{r}(\mathcal{O}_{\EQ})(s)[-r-s]\big)[p+q].$$
   A \emph{closed variational $(p,q)$-form of degree $n$} is an $n$-coycle in $\mathcal{V}ar^{(p,q)}$, with respect to the total differential,
   $$\mathcal{V}ar^{(p,q),cl}(\EQ,n):=Z^n\big(\mathcal{V}ar^{(p,q)}(\EQ),D=d_h+d_v+\delta\big).$$
\end{definition}
Before we make the content of Definition \ref{definition: n-shifted (p,q)-forms} explicit, note there are canonical projections 
 \begin{equation}
 \label{eqn: p,q-projection}
\eta_{p,q}:\mathcal{V}ar^{k}(\EQ)\rightarrow \mathcal{V}ar^{(p,q)}(\EQ),
\end{equation}
to the complex of differential $(p,q)$-forms as well as tautological 
 projections onto purely horizontal and vertical complexes
\[
\begin{tikzcd}
    & \arrow[dl] \mathcal{V}ar^{k}(\EQ)\arrow[dr]
    \\
    \mathcal{V}ar^{(k,0)}(\EQ) && \mathcal{V}ar^{(0,k)}(\EQ).
\end{tikzcd}
\]

Supposing that $\mathcal{A}^{\bullet}$ is a cofibrant $\D$-algebra, arising for example from a generically irregular $\D$-algebra $\mathcal{B}\simeq \mathcal{A}/\mathcal{I}$ of functions (c.f. \ref{eqn: SES}), as $\mathcal{A}^{\bullet}=Q\mathcal{B},$ its cofibrant replacement. Then the action of the three differentials on the term consisting of $(p,q)$-forms of degree $k$ (omitting all other arrows) is:
\begin{equation}
\label{eqn: PreliminaryTriComplex}
\adjustbox{scale=.83}{
\begin{tikzcd}[row sep=scriptsize, column sep=scriptsize]
& \overbrace{p_{\infty}^*\Omega_X^{p}\otimes(\Bigwedge^{q}\Omega_{Q\mathcal{B}}^1)^{k}[q]}^{\textcolor{blue}{(p,q)-\text{forms of degree }k}}\arrow[dl, "\delta"] \arrow[rr, "d_h"] \arrow[dd, "d_v"] & &p_{\infty}^*\Omega_X^{p+1}\otimes(\Bigwedge^{q}\Omega_{Q\mathcal{B}}^1)^{k}[q] \\
p_{\infty}^*\Omega_X^{p}\otimes(\Bigwedge^{q+1}\Omega_{Q\mathcal{B}}^1)^{k+1}[q+1]& &\\
& p_{\infty}^*\Omega_X^{p}\otimes(\Bigwedge^{q+1}\Omega_{Q\mathcal{B}}^1)^{k}[q+1] & &    \\
 & & \\
\end{tikzcd}}
\end{equation}
From (\ref{eqn: PreliminaryTriComplex}) it is clear the natural bi-graded algebra multiplication on the cohomologically graded object of variational $(p,q)$-forms:
$$\mathcal{V}ar^{(p_1,q_1)}(\EQ)\times \mathcal{V}ar^{(p_2,q_2)}(\EQ)\rightarrow \mathcal{V}ar^{(p_1+p_2,q_1+q_2)}(\EQ),$$
acts on homogeneous graded components of cohomological degreees, say $k_1,k_2$ for a cofibrant $\D$-algebra $\mathcal{A}$ by
$$DR_X^{p_1}(\wedge^{q_1}\Omega_{\mathcal{A}}^1)^{k_1}[q_1])\times DR_X^{p_2}(\wedge^{q_2}\Omega_{\mathcal{A}}^1)^{k_2}[q_2])\rightarrow DR_{X}^{p_1+p_2}(\wedge^{q_1+q_2}\Omega_{\mathcal{A}}^1)^{k_1+k_2}[q_1+q_2].$$

 We can now state the second half of Theorem A, pertaining to the description of the `key' of forms and the homotopy-coherent data necessary for their closure in the totalized complex.

 In the following, we also will use the notation $\mathbb{L}\Omega^{(p,q)}(\EQ)_{\mathbf{k}}$ to mean $DR_X^{p}(\wedge^q\mathbb{L}\Omega_{\mathcal{A}}^1)^{k}[q]),$ for a not necessarilly cofibrant $\D$-algebra, to emphasize this is a derived analog of the $(p,q)$-component of the usual variational bi-complex.
 
\begin{proposition}[\textcolor{blue}{Theorem A (ii)}]
\label{prop: Cohom degree N (p,q) form key}
A cohomological degree $N$ differential $(p,q)$-form $$\vartheta=(\theta_N^{(p,q)}|\cdots) \in \mathcal{V}ar^{(p,q)}(\EQ),$$ is a formal sequence of elements $\theta_{M}^{(r,s)}\in \mathbb{L}\Omega^{(r,s)}(\EQ)_{\mathbf{M}},$ with leading term $\theta_N^{(p,q)}\in \mathbb{L}\Omega^{(p,q)}(\EQ)_{\mathbf{N}}$ with all others coming from:
$$\scalemath{.95}{\mathbb{L}\Omega^{(p+1,q)}(\EQ)_{\mathbf{N-1}}\times \mathbb{L}\Omega^{(p,q+1)}(\EQ)_{\mathbf{N-1}},\mathbb{L}\Omega^{(p+2,q)}(\EQ)_{\mathbf{N-2}}\times \mathbb{L}\Omega^{(p+1,q+1)}(\EQ)_{\mathbf{N-2}}\times \mathbb{L}\Omega^{(p,q+2)}(\EQ)_{\mathbf{N-2}}},$$
such that $\theta^{(p,q)}\neq 0$ for possible infinitely many $\ell=p+q.$
A generic generic degree $\mathbf{k}$-component comes from
$$\scalemath{.95}{\mathbb{L}\Omega^{(p+k,q)}(\EQ)_{\mathbf{N-k}}\times\mathbb{L}\Omega^{(p+k-1,q+1)}(\EQ)_{\mathbf{N-k}}\times\cdots\times \mathbb{L}\Omega^{(p+1,q+k-1)}(\EQ)_{\mathbf{N-k}}\times\mathbb{L}\Omega^{(p,q+k)}(\EQ)_{\mathbf{N-k}}}.
$$
\end{proposition}
In other words, a cohomological degree $\mathbf{N}$ differential $n$-form on $\EQ$, understood as an element of 
$H^N\big(\mathcal{V}ar^n(\EQ)\big)$ amounts to a semi-infinite sequence given by a formal sum of cohomological $(p,q)$-forms for which $p+q=n.$
\vspace{1mm}

Two $(p,q)$-forms of cohomological degree $N\leq 0$, say $\vartheta_{N}^{(p,q)},\widetilde{\vartheta}_{N}^{(p,q)}\in \Omega_{var,N}^{(p,q)}$ with $\delta \vartheta_N^{(p,q)}=\delta\widetilde{\vartheta}_N^{(p,q)}=0$ in $\Omega_{var,N+1}^{(p,q)}$ are naive-equivalent if there exists some $\alpha_{N-1}^{(p,q)}\in \Omega_{var,N-1}^{(p,q)}$ such that 
$$\vartheta_{N}^{(p,q)}-\widetilde{\vartheta_{N}}^{(p,q)}=\delta \alpha_{N-1}^{(p,q)}.$$
\begin{proposition}[\textcolor{blue}{Theorem A (iii)}]
    \label{prop: Closed degree N (p,q) form}
   Consider Proposition \emph{(\ref{prop: Cohom degree N (p,q) form key})}. Then a cohomological degree $N$ differential $(p,q)$-form $\vartheta$ is closed in the sense of Definition \emph{\ref{definition: n-shifted (p,q)-forms}} i.e. $$\vartheta=(\theta_N^{(p,q)}|\cdots)\in Z^N\big(\mathcal{V}ar^{(p,q)}(\EQ),D=d_h+d_v+\delta\big), p=1,\ldots,d_X,q\geq 0,$$ if it satisfies $$\delta(\theta_N^{(p,q)})=0\in \mathbb{L}\Omega^{(p,q)}(\EQ)_{\mathbf{N+1}},$$ and the further homotopy-coherent system of relations:
\begin{itemize}
\item $d^v\theta_{N}^{(p,q)}+\delta\theta_{N-1}^{(p,q+1)}=0$ in $\mathbb{L}\Omega^{(p,q+1)}(\EQ)_{\mathbf{N}},$
\item $d_h\theta_N^{(p,q)}+\delta\theta_{N-1}^{(p+1,q)}=0$ in $\mathbb{L}\Omega^{(p+1,q)}(\EQ)_{\mathbf{N}};$

\item $d^v\theta_{N-1}^{(p,q+1)}+\delta\theta_{N-2}^{(p,q+2)}=0$ in $\mathbb{L}\Omega^{(p,q+2)}(\EQ)_{\mathbf{N-1}},$

\item $d_h\theta_{N-1}^{(p,q+1)}+d^v\theta_{N-1}^{(p+1,q)}+\delta\theta_{N-2}^{(p+1,q+1)}=0$ in $\mathbb{L}\Omega^{(p+1,q+1)}(\EQ)_{\mathbf{N-1}},$
\item $d_h\theta_{N-1}^{(p+1,q)}+\delta\theta_{N-2}^{(p+2,q)}=0,$ in $\mathbb{L}\Omega^{(p+2,q)}(\EQ)_{\mathbf{N-1}};$

\item etc.

\end{itemize}

and so on.
\end{proposition}
Suppose that $\mathcal{A}$ is a cofibrant differential graded $\D_X$-algebra.
Taking connected components of the spaces of shifted forms gives
$$H^0\big(\prod_{i\geq 0}\Omega_{\mathcal{A}}^{p+i}[n-i],D_{dR}\big)\cong H^n\big(\prod_{i\geq 0}\Omega_{\mathcal{A}}^{p+i}[-i],D_{dR}\big).$$
One may say that a \emph{closed Cartan form} of \emph{cohomological degree} $n$, with $p\geq 0,n\leq 0$ is thus a sequence $\omega=\big(\omega^0,\omega^1,\omega^2,...\big)$ of forms, where $\omega^i\in \Omega^{p+i}(\mathcal{A})^{n-i},$ for each $i=0,1,2,...$ such that:
$
d\omega^0=0\in \Omega^p(\mathcal{A})^{n+1},
d_{dR}\omega^0+d\omega^1=0$ in $\Omega^{p+1}(\mathcal{A})^{n}$ and similarly that 
$d_{dR}\omega^1+d\omega^2=0$ in $\Omega^{p+1+1}(\mathcal{A})^{n-1}$ and so on.
The homotopy coherent data for $\omega$ to be closed as an expression for $D_{dR}$-closedeness is written as $D_{dR}(\omega)=0$, i.e. $d_{dR}\omega^i+d\omega^{i+1}=0,\hspace{1mm} \text{in } \Omega^{p+i+1}(\mathcal{A})^{n-i},\hspace{1mm} i\geq 0.$
Furthermore, there exists the usual Hodge filtration, this time on bi-complexes:
$$\mathcal{F}_{\mathrm{Hodge}}^p\mathbf{DR}^{\mathcal{D}}\big(\mathcal{A}\big):=\mathbf{DR}^{\mathcal{D}}(\mathcal{A})^{\geq p}[-p]=\big[\Omega_{\mathcal{A}}^p\rightarrow \Omega_{\mathcal{A}}^{p+1}\rightarrow\ldots\big],$$ where the shift indicates we have placed $\Omega_{\mathcal{A}}^p$ in degree $0.$
This is compatible with our homotopy-theoretic notions of equivalence. Indeed, if $f:\mathcal{A}\rightarrow \mathcal{B}$ is a weak-equivalence of cofibrant derived $\mathcal{D}_X$-algebras, then $\Omega_{\mathcal{A}}^p\rightarrow \Omega_{\mathcal{B}}^p$ is a weak-equivalence for all $p$ and so there is an induced weak-equivalence between de Rham algebras.

One then has the complex of $n$-shifted pre-symplectic forms on $Spec_{\mathcal{D}}(\mathcal{A}^{\bullet})$ is a truncation of the $\prod$-totalization 
$$\mathrm{Tot}^{\prod}\big(\mathcal{F}^p\mathbb{L}\Omega_{\mathcal{A}}^{*}[n+2]\big).$$

A $\D$-geometric local symplectic form on a cofibrant replacement $\D$-algebra $\mathcal{A}^{\bullet}:=Q\mathcal{B}\rightarrow \mathcal{B}$ associated to a sequence (\ref{eqn: SES}) can be defined as a $2$-form $\vartheta^2$ of degree $N$ as in Proposition (\ref{prop: Closed degree N (p,q) form}), corresponding to a homotopy class of $$H^{N+2}(Tot F^2\mathbb{L}\Omega_{\mathcal{B}}^*)\simeq H^{N+2}\big(Tot \Omega_{Q\mathcal{B}}^{\geq p}[p]),$$ which is non-degenerate in the sense that the map 
$$\mathrm{Ext}_{Q\mathcal{B}\otimes \mathcal{D}_X}^i(\Omega_{Q\mathcal{B}}^1,Q\mathcal{B}\otimes \mathcal{D}_X)^{\ell}\rightarrow H^{-i-N}(\Omega_{Q\mathcal{B}}^1),$$
induced by the contraction with the leading term $\theta_N^{2}\in H^{-N}(\Omega_{Q\mathcal{B}}^2),$ is an equivalence. 
We have taken homotopy groups of mapping spaces in the $(\infty,1)$-category $\mathsf{Mod}(\mathcal{A}[\mathcal{D}_X])$   i.e. 
$$\pi_i\big(R\mathsf{Maps}_{\mathcal{A}[\mathcal{D}_X]}(\mathbb{L}_{\mathcal{A}^{\bullet}},-)\big)\simeq \mathrm{Ext}_{\mathcal{A}[\mathcal{D}_X]}^{-i}(\mathbb{L}_{\mathcal{A}^{\bullet}},-)\in Ho(\DG(\mathcal{A})).$$
There is a corresponding sub-simplicial set of $\mathcal{V}ar^2$ consisting of non-degenerate $2$-forms, $\mathcal{V}ar^{2,nd}(\mathcal{A}^{\bullet},N)$ but a more relevant object obtained from (\ref{eqn: p,q-projection}), is 
$\mathcal{V}ar^{(d_X,2),nd}.$

\subsubsection{Secondary differential forms as homotopy groups}
We have an interpretation of Secondary differential forms on the space of solutions to a classical $\D_X$-PDE as in Proposition \label{prop: Integration pairing} in terms of spaces \ref{definition: n-shifted (p,q)-forms}. Namely, if
If $\EQ$ is a classical $\D_X$-prestack corresponding to the jet-construction, considering Proposition \ref{Theorem A} we compute homotopy groups (c.f. \cite{PTVV}) to see the $\D$-prestacks $\mathcal{V}ar^{p,cl}(\mathcal{A},n)$ for every $n\geq 0$ are $n$-truncated\footnote{This means as a $\D$-prestack they take values in the full-subcategory of spaces $\mathsf{Spc}_{\leq n}\subset \mathsf{Spc}$ i.e. spaces for which each connected component $S_0$ is such that $\pi_{\ell>n}(S_0)=0,$ for every $\ell>n.$}, and 
$$\big(\mathcal{V}\mathrm{ar}^{p,cl}(Spec_{\mathcal{D}_X}(\mathcal{A}),0\big)\simeq \mathrm{Hom}(\mathcal{A},\Omega_{\mathcal{A}}^{p,cl})\equiv \Gamma(Spec_{\mathcal{D}_X}(\mathcal{A}),\Omega_{\mathcal{A}}^{p,cl}),$$ 
as the `usual' closed $p$-forms on $Spec_{\mathcal{D}}(\mathcal{A}).$

Furthermore, if $\EQ$ is affine $\D$-smooth and classical (structure $\D$-algebra is in degree zero) then for all $0\leq i\leq n-1$ we have $$
\pi_i\big(\mathcal{V}\mathrm{ar}^{p,cl}(\EQ,n)\big)\simeq  h^{p+n-i}\big(\mathbf{DR}^{\mathrm{var}}(\mathcal{A}^{\ell})\big)\simeq H^{p+n-i}\big(\Omega_X^*\otimes\Bigwedge^p\Omega_{\mathcal{A}}^1\big),$$
recovering $\pi_{\infty}^*\Omega_X^{d_X}\otimes_{\mathcal{D}_X}\Bigwedge^p\Omega_{\mathcal{A}}^1$ in degree $p+n-i$.

Proposition \ref{prop: Closed degree N (p,q) form} is consistent with well-known facts coming from geometry of jet-spaces.  
\begin{proposition}
    \label{prop: Jets is pre-symplectic}
Let $E\rightarrow X$ be vector bundle over a manifold $X$ with $dim(X)=d_X.$ Then, $\JetX(E)$ carries a natural $0$-shifted $(d_X-1,2)$-form, for every choice of element in $\mathcal{O}(\JetX(E))^r.$
\end{proposition}
\begin{proof}
 This is evident via the usual prescription of the pre-symplectic form on the jet-space of any fibered manifold (see \cite{Anderson} for a classical statement), that is easily adapted to our situation.
 \end{proof}

\subsection{Derived Variational de Rham Algebra.}
\label{ssec: Derived Variational de Rham Algebra} We use the tool of negative cyclic homology to define a totalization of the derived tri-complex, whose three directions corresponding to the differentials $d_h,d^v$ and the internal cohomological differential $d_{\mathcal{A}}$, corresponding to a triple grading $(p,q;n)\in\mathbb{N}\times\mathbb{Z}_+\times\mathbb{Z}_-$ corresponding to: a horizontal i.e. $p$-direction,  concentrated in degrees $[0,d_X],$
the vertical/jet/Cartan direction i.e. $q$-direction concentrated in degrees $[0,\infty],$ and the internal/derived direction $n$, concentrated in degrees\footnote{We often impose some connectivity/finiteness conditions e.g. eventually co-connectivity so there exists some $n$ for which all cohomologies in lower degrees vanish i.e. we only have concentration in $[-n,0].$ } $[-\infty,0]$. In stacky situations, there are also connective degrees i.e. $[-n,m]$ for some $m\in \mathbb{Z}_+.$

\begin{equation}
    \label{eqn: VariationalTriComplex}
\adjustbox{scale=.75}{
\begin{tikzcd}[column sep=7pt, row sep= 7pt]
 && && 
\\
& \vdots  \arrow[ddd] && \vdots \arrow[ddd, swap, near start]  && \vdots \arrow[ddd,swap, near start] 
\\
&& \vdots\arrow[ddd,"d_v"]  &&
 \vdots \arrow[ddd,"d_v"] && \vdots\arrow[ddd,"d_v"]
  \\
  \cdots\arrow[dr, "\delta", swap] &&\cdots\arrow[dr,"\delta", swap]&& \cdots\arrow[dr,"\delta", swap]
\\
& \cdots\rightarrow\mathcal{V}ar_{\mathbf{-1}}^{(p,q-1)} \arrow[rr,"d_h"] \arrow[dr,swap,"\delta"] \arrow[ddd,swap,"d_v"]  &&
  \mathcal{V}ar_{\mathbf{-1}}^{(p+1,q-1)} \arrow[ddd,swap,"d_h" near start] \arrow[dr,"\delta"] \arrow[rr,"d_h"] && \mathcal{V}ar_{\mathbf{-1}}^{(p+2,q-1)}\rightarrow\cdots \arrow[ddd,swap,"\delta_h" near start] \arrow[dr,"\delta"] \\
&& \cdots\rightarrow \mathcal{V}ar_{\mathbf{0}}^{(p,q-1)} \arrow[ddd,swap] \arrow[rr,crossing over,"\delta" near start] &&
  \mathcal{V}ar_{\mathbf{0}}^{(p+1,q-1)} \arrow[ddd,"d_v"] \arrow[rr,crossing over,"\delta_h" near start] && \mathcal{V}ar_{\mathbf{0}}^{(p+2,q-1)} \rightarrow \cdots\arrow[ddd,"d_v"]\\
  \cdots\arrow[dr,swap] &&\cdots\arrow[dr,swap]&& \cdots\arrow[dr,swap]
\\
&\cdots\rightarrow \mathcal{V}ar_{\mathbf{-1}}^{(p,q)} \arrow[rr,"d_h" near end] \arrow[dr,swap,"\delta"] && \mathcal{V}ar_{\mathbf{-1}}^{(p+1,q)} \arrow[dr,swap,"d"] \arrow[rr,"d_h" near end] && \mathcal{V}ar_{\mathbf{-1}}^{(p
+2,q)}\rightarrow \cdots \arrow[dr,swap,"\delta"] \\
&& \cdots\rightarrow \mathcal{V}ar_{\mathbf{0}}^{(p,q)} \arrow[rr,"d_h"] \arrow[uu,<-,crossing over,"d_v" near end]&&\mathcal{V}ar_{\mathbf{0}}^{(p+1,q)}\arrow[uu,<-,crossing over,"d_v" near end]\arrow[rr,"d_h"] && \mathcal{V}ar_{\mathbf{0}}^{(p+2,q)}\rightarrow \cdots
\end{tikzcd}}
\end{equation}

The variational tri-complex (\ref{eqn: VariationalTriComplex}) is shown such that each face is itself a variational \emph{bi}complex of a fixed cohomological degree (degrees $\cdots,\mathbf{-1},\mathbf{0}$ are shown), where the `derived directions' come out of the page. If the derived $\D$-algebra is concentrated in non-positive degrees, then it terminates in degree $0.$ Moreover, the horizontal arrows continue until we hit the dimension of $X$ i.e. $p=d_X,$ and the vertical arrows are coming from $q=0,$ and continue downward. The so-called `jet-direction' is unbounded e.g. $$d_v:\mathcal{V}ar_{\mathbf{n}}^{(*,*)}\rightarrow \mathcal{V}ar_{\mathbf{n}}^{(*,q+1)}, q\geq 0.$$

In other words we have a sequence of variational bicomplexes
\begin{equation}
    \label{eqn: VarBicomplexSeq}
\cdots\rightarrow  \mathbb{L}\Omega^{(*,*)}(\EQ)_{\mathbf{-2}}\xrightarrow{d_{\mathcal{A}}} \mathbb{L}\Omega^{(*,*)}(\EQ)_{\mathbf{-1}}\xrightarrow{d_{\mathcal{A}}}\mathbb{L}\Omega^{(*,*)}(\EQ)_{\mathbf{0}}\rightarrow 0,
\end{equation}
induced by the internal differential, where the cohomological degree $-i$ term is explicitly decomposed via
$$\mathbb{L}\Omega^{(*,*)}(\EQ)_{\mathbf{-i}}\simeq \bigoplus_{q\geq 0,p=1}^{p=d_X}\wedge^q\mathbb{L}_{\EQ/X}^{-i}\otimes^!p_{\infty}^*\Omega_X^p, \hspace{1mm}\mathbf{i}\in \mathbb{Z}_+.$$
Given a non-linear PDE $\iota:\EQ\hookrightarrow p^*p_*E,$ the sequence of variational bicomplexes (\ref{eqn: VarBicomplexSeq}) recieves a morphism from the sequence of `free' variational bicomplexes. 
\begin{proposition}
\label{prop: MorphismOfVarBiSeq}
    There is a morphisms of complexes of variational bicomplexes:
    \[
    \begin{tikzcd}
\cdots \arrow[r]& \iota^*\mathbb{L}\Omega^{(*,*)}(p^*p_*E)_{\mathbf{-2}}\arrow[d,"\iota^*"] \arrow[r,"d_{\mathcal{A}}"] & \iota^*\mathbb{L}\Omega^{(*,*)}(p^*p_*E)_{\mathbf{-1}}\arrow[d,"\iota^*"] \arrow[r,"d_{\mathcal{A}}"] & \iota^*\mathbb{L}\Omega^{(*,*)}(p^*p_*E)_{\mathbf{0}}\arrow[d,"\iota^*"]
\\
\cdots \arrow[r] & \mathbb{L}\Omega^{(*,*)}(\EQ)_{\mathbf{-2}}\arrow[r,"d_{\mathcal{B}}"] & \mathbb{L}\Omega^{(*,*)}(\EQ)_{\mathbf{-1}}\arrow[r,"d_{\mathcal{B}}"] & \mathbb{L}\Omega^{(*,*)}(\EQ)_{\mathbf{0}}
    \end{tikzcd}
    \]
\end{proposition}
Proposition \ref{prop: MorphismOfVarBiSeq} is obtained via the obvious morphism induced by $\iota^*\mathbb{L}_{p^*p_*E}\rightarrow \mathbb{L}_{\EQ},$ denoted abusively by $\iota^*.$

\subsubsection{Example: $k=1$ cohomological degree $N=0$ forms.} 
An $N$-shifted variational $k$-form is an element of the weighted negative cyclic complex:
$$\mathcal{A}_{var}^k(\EQ;N)=NC^w\big(\mathbf{DR}_{var}(\EQ)[n-k]\big)(k).$$
 In particular, it is obtained from the negative cyclic complex, which in degree $n$ and weight $k$ is
$$NC(\mathbf{DR}_{\EQ})^n(k)\simeq \bigoplus_{i\geq 0}\mathbf{DR}_{\EQ}^{n-2i}(k+i)\simeq \mathbf{DR}_{\EQ}^n(k)\times \mathbf{DR}_{\EQ}^{n-2}(k+1)\times \mathbf{DR}_{\EQ}^{n-4}(k+2)\times\cdots,$$
where 
$$\mathbf{DR}_{\EQ}^{n-2i}(k+i)\simeq DR_X\big(\mathbf{DR}_{\EQ/X}^{n-2i})(k+i)\simeq\bigoplus_{r+s=k+i}DR_X^r\big(\mathbf{DR}_{\EQ/X}^{n-2i}\big)(s).$$
To describe the cocycles in $NC^w\big(\mathbf{DR}_{var}(\EQ)[-1]\big)(1),$ with respect to $D:=\delta+d_{dR}\simeq \delta+d_h+d_v,$ note that this complex is arranged as
\[
\adjustbox{scale=.80}{
\begin{tikzcd}
& \vdots\arrow[d] & \Omega_{\EQ,-2}^2\arrow[d]
\\
\vdots \arrow[d] & \Omega_{\EQ,-2}^1\arrow[d]\arrow[ur] & \Omega_{\EQ,-1}^2\arrow[d]
\\
\mathcal{O}_{\EQ}^{-2}
    \arrow[d,"\delta"] \arrow[ur] & \Omega_{\EQ,-1}^1\arrow[d]\arrow[ur]& \Omega_{\EQ,0}^2\arrow[d] 
    \\
\mathcal{O}_{\EQ}^{-1}\arrow[d] \arrow[ur] & \Omega_{\EQ,0}^1\arrow[d]\arrow[ur] & 0\arrow[d]
\\
\mathcal{O}_{\EQ}^{0}\arrow[ur] & 0 & 0
\end{tikzcd}}
\]

In between degrees $N=0$ and $N=1$, we have the differential $D$ mapping:
$$\big(\Omega_{\EQ,0}^1\times\Omega_{\EQ,-1}^2\times\Omega_{\EQ,-2}^3\times\cdots\big)\xrightarrow{D}\big(0\times \Omega_{\EQ,0}^2\times\Omega_{\EQ,-2}^3\times\cdots\big),$$
where we see that $\omega^1:=\big(\omega_0^1,\omega_{-1}^2,\omega_{-2}^3,\cdots),$ with $\omega_{-i}^j$ a $j$-form in cohomological degree $i.$ So, the datum of being closed amounts to:
$$\delta(\omega_0^1)=0,\hspace{2mm} d_{dR}(\omega_0^1)+\delta(\omega_{-1}^2)=0,\hspace{2mm} d_{dR}(\omega_{-1}^2)+\delta(\omega_{-2}^3)=0,\cdots.$$
Using the \emph{bi}-complex structure i.e. $\omega^j\simeq \sum_{(r,s), r+s=j}\vartheta^{(r,s)}$ we see:
\begin{itemize}
    \item In de Rham weight $1$, two contributions: $(1,0)+(0,1)$ and thus $\delta(\omega_0^1)=0$ reads as $\delta\vartheta_0^{(1,0)}+\delta\vartheta_0^{(0,1)}=0,$ which upon equating the bi-complex degrees can be written as two separate equations:
    $$\delta(\vartheta_0^{(1,0)})=0,\hspace{2mm} \delta(\vartheta_0^{(0,1)})=0.$$

    \item In de Rham weight $2$: we get three contributions $(2,0)+(1,1)+(0,2),$ and therefore the second relation yields
    $$
\begin{cases}
    d_h(\vartheta_0^{(1,0)})+\delta(\vartheta_{-1}^{(2,0)})=0
    \\
    d_h(\vartheta_0^{(0,1)})+d_v(\vartheta_0^{(1,0)})+\delta(\vartheta_{-1}^{(1,1)})=0,
    \\
    d_v(\vartheta_0^{(0,1)})+\delta(\vartheta_{-1}^{(0,2)})=0,
\end{cases}
$$
where again we have separated the relation by equating the components which have identical bi-complex weights.

\item In de Rham weight $3$, the contributions come from $(3,0)+(2,1)+(1,2)+(0,3),$ and we obtain:
$$
\begin{cases}
    d_h(\vartheta_{-1}^{(2,0)})+\delta(\vartheta_{-2}^{(3,0)})=0
    \\
    d_h(\vartheta_{-1}^{(1,1)})+d_v(\vartheta_{-1}^{(2,0)})+\delta(\vartheta_{-2}^{(2,1)})=0,
    \\
    d_h(\vartheta_{-1}^{(0,2)})+d_v(\vartheta_{-1}^{(1,1)})+\delta(\vartheta_{-2}^{(1,2)})=0,
    \\
    d_v(\vartheta_{-1}^{(0,2)})+\delta(\vartheta_{-2}^{(0,3)})=0.
\end{cases}$$

\end{itemize}
The other homotopy coherent data arising from the condition 
$d_{dR}(\omega_{-i}^j)+\delta(\omega_{-i-1}^{j+1})=0,$ can be written down in a similar way.

\subsubsection{Example: Quasi-Smoothness}
Let $\EQ$ be a quasi-smooth derived algebraic non-linear PDE with $X$ an even variety, so that in particular, 
$\mathbb{L}_{\EQ}$ is concentrated in degrees $[-1,0]:$
$$\mathbb{L}_{\EQ}\simeq\big[L_{\EQ}^{-1}\rightarrow L_{\EQ}^0\big].$$
Furthermore, suppose that $L_{\EQ}^{-1},L_{\EQ}^0$ are given by vector $\D_X$-bundles over $\EQ$ i.e. $X$-locally finitely projective $\mathcal{O}_{\EQ}\otimes\D_X$-modules, flat over $X.$

Considering Proposition \ref{prop: DRVar is a prestack},
one has that 
$$\mathbb{L}\Omega^{(*,*)}(\EQ)=DR_X\big(\mathcal{S}ym_{\mathcal{O}_{\EQ}}(\mathbb{L}^{\bullet}[1])\big)\simeq \bigoplus_{\mathbf{-i}}\mathcal{V}ar_{\EQ,\mathbf{-i}}^{(*,*)}.$$
In this case, restricting to cohomological degree $\mathbf{-i}$ component of the $q$-th weight piece e.g. $\mathcal{V}ar_{\EQ,-\mathbf{i}}^{(*,q)},$ we see that it is given 
$$\mathcal{V}ar_{\EQ}^{(*,q)}\simeq \bigoplus_{-i}p_{\infty}^*\Omega_{X}^{d_X-*}\otimes \Bigwedge^q\mathbb{L}_{\EQ}[q].$$
By our assumptions, for every $q\geq 0$ the complex $\Bigwedge^q\mathbb{L}_{\EQ}[q]$ is a perfect $\mathcal{O}_{\EQ}[\mathcal{D}_X]$-module, so that $H_{\mathcal{D}}^i(\Bigwedge^q\mathbb{L}_{\EQ})$ are finitely presented $H_{\mathcal{D}}^0(\mathcal{O}_{\EQ})\otimes_{\mathcal{O}_X}\mathcal{D}_X$-modules. Therefore, pull-back a long a solution gives a coherent $\D$-module. Quasi-smoothness of $\EQ$ leads to a comparatively simple description of the variational tri-complex; in each cohomological degree for fixed jet weights it looks like
\begin{equation}
\label{eqn: Tricomplex decompose}
\mathcal{V}ar_{\EQ,-\mathbf{i}}^{(*,q)}\simeq p_{\infty}^*\Omega_X^{d_X-*}\otimes\big(Sym_{\mathcal{O}_{\EQ}}^{q-i}(L_{\EQ}^{0})\otimes_{\mathcal{O}_{\EQ}}\Bigwedge^iL_{\EQ}^{-1}\big).
\end{equation}

\begin{figure}[h]
\label{figure: QuasiSmoothTriComplex}
    \centering
\adjustbox{scale=.80}{
\begin{tikzcd}[column sep=6pt, row sep= 6pt]
{\color{black!50!white}\mathcal{V}ar_{\mathbf{-1}}^{(0,q-2)}}\arrow[rr,"d_h", black!50!white] \arrow[dr, swap, "\delta"] \arrow[dd, "d_v", black!50!white] && {\color{black!50!white}\mathcal{V}ar_{\mathbf{-1}}^{(1,q-2)}} \arrow[dd, "d_v", swap, near start, black!50!white]\arrow[dr,"\delta"] \arrow[rr,"d_h", black!50!white] && {\color{black!50!white}\mathcal{V}ar_{\mathbf{-1}}^{(2,q-2)}}\arrow[dd, "d_v", swap, near start, black!50!white]\arrow[dr, "\delta"] 
\\
& {\color{white!25!blue}\mathcal{V}ar_{\mathbf{0}}^{(0,q-2)}}\arrow[dd,"d_v", white!25!blue] \arrow[rr,crossing over,"d_h" near start, white!25!blue] &&
  {\color{white!25!blue}\mathcal{V}ar_{\mathbf{0}}^{(1,q-1)}} \arrow[dd,"d_v", white!25!blue] \arrow[rr,crossing over,"d_h" near start, white!25!blue] && {\color{white!25!blue}\mathcal{V}ar_{\mathbf{0}}^{(2,q-1)}} \arrow[dd,"d_v",white!25!blue]
  \\
{\color{black!50!white}\mathcal{V}ar_{\mathbf{-1}}^{(0,q-1)} }\arrow[rr,"d_h", black!50!white] \arrow[dr,swap,"\delta"] \arrow[dd,swap,"d_v",black!50!white]  &&
 {\color{black!50!white}\mathcal{V}ar_{\mathbf{-1}}^{(1,q-1)}} \arrow[dd,swap,"d_v" near start, black!50!white] \arrow[dr,"\delta"] \arrow[rr,"d_h", black!50!white] && {\color{black!50!white}\mathcal{V}ar_{\mathbf{-1}}^{(2,q-1)}} \arrow[dd,swap,"d_v" near start, black!50!white] \arrow[dr,"\delta"] \\
& {\color{white!25!blue}\mathcal{V}ar_{\mathbf{0}}^{(0,q-1)}} \arrow[rr,crossing over,"d_h" near start, white!25!blue] &&
  {\color{white!25!blue}\mathcal{V}ar_{\mathbf{0}}^{(1,q-1)}} \arrow[dd,"d_v",white!25!blue] \arrow[rr,crossing over,"d_h" near start, white!25!blue] && {\color{white!25!blue}\mathcal{V}ar_{\mathbf{0}}^{(2,q-1)}} \arrow[dd,"d_v", white!25!blue]\\
{\color{black!50!white}\mathcal{V}ar_{\mathbf{-1}}^{(0,q)}} \arrow[rr,"d_h" near end, black!50!white] \arrow[dr,swap,"\delta"] && {\color{black!50!white}\mathcal{V}ar_{\mathbf{-1}}^{(1,q)}} \arrow[dr,swap,"\delta"] \arrow[rr,"d_h" near end,black!50!white] && {\color{black!50!white}\mathcal{V}ar_{\mathbf{-1}}^{(2,q)}} \arrow[dr,swap,"\delta"] \\
& {\color{white!25!blue}\mathcal{V}ar_{\mathbf{0}}^{(0,q)}} \arrow[rr,"d_h", white!25!blue] \arrow[uu,<-,crossing over,"d_v" near end, white!25!blue]&&{\color{white!25!blue}\mathcal{V}ar_{\mathbf{0}}^{(1,q)}}\arrow[uu,<-,crossing over,"d_v" near end, white!25!blue]\arrow[rr,"d_h", white!25!blue] && {\color{white!25!blue}\mathcal{V}ar_{\mathbf{0}}^{(2,q)}}
\end{tikzcd}}
\caption{\emph{A truncated quasi-smooth variational tri-complex. Each coloured face is a bi-complex in a fixed cohomological degree.}}
\end{figure}

For instance, taking naive truncations in the $q,p$ directions the (colour-coated) tri-complex in cohomological degrees $0,1$ as shown in (Figure. \textcolor{blue}{$1$}.) corresponds to:
$$\tau^{\leq q-2}\sigma_{>2}\mathbb{L}\Omega^{(*,*)}(\EQ)_{\mathbf{-1}}\overset{\delta}{\longrightarrow} \tau^{\leq q-2}\sigma_{>2}\mathbb{L}\Omega^{(*,*)}(\EQ)_{\mathbf{0}},$$
or via (\ref{eqn: Tricomplex decompose}),
$$\bigoplus_{0\leq q'\leq q-2}p_{\infty}^*\Omega_X^{p\leq 2}\otimes (Sym^{q'-1}(L_{\EQ}^0)\otimes L_{\EQ}^{-1}[-1])\rightarrow \bigoplus_{0\leq q'\leq q-2}p_{\infty}^*\Omega_X
^{p\leq 2}\otimes Sym^{q'}(L_{\EQ}^0).$$

\subsubsection{}
We give one result now without proof that makes use of Proposition \ref{prop: Cohom degree N (p,q) form key} and \ref{prop: Closed degree N (p,q) form}, and should be interpreted as the realization of the central idea of Secondary Calculus as a homotopy-theoretic analog of Proposition \ref{prop: Integration pairing}. Roughly, it states that cohomologically shifted variational forms on an affine $\D$-prestack induce forms on certain mapping prestacks of sections (a kind of PDE-theoretic transgression).\footnote{Analogous to the usual AKSZ-PTVV derived transgression.} In particular, it holds for simplicial sets of closed forms.
\begin{proposition}
  \label{prop: Transgression}
Suppose that $X$ is a compact $d_X$-dimensional manifold. Let $\EQ\hookrightarrow \JetX(E)$ be a $\D$-finitary derived non-linear PDE (see $(\star)$) imposed on some sub-space of sections $\mathbb{R}Sect_0(X,E).$ 
    Assume it is affine i.e. $\EQ\simeq \mathbb{R}Spec_{\mathcal{D}}(\mathcal{B}^{\bullet}),$ then for a fixed cohomological degree $N$, and every $0\leq p\leq d_X$, there is a morphism of simplicial sets 
    $$\mathcal{V}ar_{cl}^{(d_X-p,q)}\big(\mathbb{R}Spec_{\mathcal{D}}(\mathcal{B}^{\bullet});N)\big)\rightarrow \mathcal{A}_{cl}^q(\mathbb{R}Sect_0(X,E),N-p).$$
\end{proposition}
Recall the following essential property of jet functors in algebraic-geometry. 

Let $\mathbb{D}_x$ be the formal disk and consider $\mathcal{A}^{\ell}\in \mathrm{CAlg}_X(\mathcal{D}_X)$ and $x\in X$ with $\mathcal{A}_x^{\ell}$ the $\mathcal{D}$-module fiber. Consider the set $\mathrm{Spec}(\mathcal{A}_x)$ of $\mathcal{O}_x$-algebra homomorphisms $\mathcal{A}^{\ell}\rightarrow \mathbf{k}_x$ to the residue field at $x\in X.$
Via the canonical isomorphism 
$$\mathrm{Hom}_{\mathcal{O}_X-\mathrm{CAlg}}(\mathcal{A}^{\ell},\mathbf{k}_x)=\mathrm{Hom}_{\mathcal{D}_X\mathrm{-CAlg}}(\mathcal{A}^{\ell},\widehat{\EuScript{O}}_x),$$
we have that $\mathrm{Spec}(\mathcal{A}_x)$ consists of all flat sections of $\mathrm{Spec}(\mathcal{A}^{\ell})$ over $\mathbb{D}_x.$ 
In particular, if $\mathcal{A}^{\ell}=\mathcal{O}\big(\mathrm{Jet}^{\infty}(\mathcal{O}_{E})\big)$, by the universal property of jet-schemes, there is an isomorphism
$$\mathrm{Spec}\big(\mathcal{O}(\mathrm{Jet}^{\infty}(\mathcal{O}_E))_x\big)\simeq \Gamma(\mathbb{D}_x,\mathcal{O}_E).$$

By analogous universal properties of algebraic infinite jet functors in derived geometry, by considering the `homotopy co-free PDE' i.e. taking $\JetX(E)$ gives forms on mapping spaces out of infinitesimal formal disk bundles and Proposition \ref{prop: Transgression} gives general transgression maps lifting to closed forms for all $0\leq k <d_X.$
    
If $f:Z\rightarrow X$ is a sub-variety, of codimension $1$, consider the formal completion $X_Z^{\wedge}$ and $f^*(X_{Z}^{\wedge})$ (see \ref{eqn: Formal Normal}) and Proposition \ref{proposition: Formal mapping restrictions}. 

Namely, by Proposition \ref{prop: Transgression} and \ref{prop: Jets is pre-symplectic}, one has 
    $\omega:=\int_{Z}\mathsf{Maps}(Z,\Omega)\in \Omega^2\big(\mathsf{Maps}(f^*X_Z^{\wedge},\EuScript{E}\big).$

This gives a general hint to define transgression functors
$\int_{Z\hookrightarrow X}\mathsf{Maps}_{/X}(f^*X_Z^{\wedge},-),$ for higher codimensional sub-manifolds. It encodes in a homotopical way the expected transgression of variational forms to some $Z\hookrightarrow X$ which classicaly appears as maps from the bi-complex:
$\tau_{Z}:\Omega^{*,*}(\mathrm{Jet}^{\infty}(E))\rightarrow \Omega^*\big(Sect(N_X(Z),E)\big).$

\section{Variational Factorization Homology of PDEs}
\label{sec: Variational Factorization Homology of PDEs}
Recall the fact that if $\mathcal{A}^{\ell}\in \mathsf{CAlg}_X(\D_X),$ via Proposition \ref{Multi-jet r-adjoint proposition}, there is an object $\mathcal{F}(\mathcal{A}^{\ell})\in \mathsf{FAlg}_{\D}(X),$ whose underlying $\D$-module is supported along the main diagonal $X\subset Ran_X,$ by Kashiwara's equivalence. 

Throughout this subsection we tacitly invoke a standard argument that allows us to reduce our considerations to simple cases. That is, beginning with a homotopically finitely $\D_X$-presented prestack $\EQ,$ we write $\mathcal{O}_{\EQ}$ as a colimit over compact $\D$-algebras which are retracts of finite cellular $\D$-algebras \cite{KSYI}. Reducing further, we consider $\mathcal{A}^{\ell}\simeq Sym^!(\mathcal{M}[1])$ (using standard notation for $\otimes^!$-tensor product) for some compact $\D$-module $\mathcal{M}$.

\subsubsection{Linearization}
Consider two lax-monoidal functors $\mathrm{F},\mathrm{G}:\mathrm{PreStk}^{\mathrm{op}}\rightarrow \mathrm{Cat}^{(\infty,1)}$ and a lax monoidal transformation 
$\eta:\mathrm{F}\Rightarrow \mathrm{G}.$
\begin{proposition}
\label{prop: NatTransformOnFactorizations}
The transformation $\eta$ induces a functor of (unital) factorization objects, denoted $\mathsf{F}^{\eta}(X_{dR}):\mathsf{F}^{\mathrm{F}}(X_{dR})\rightarrow \mathsf{F}^{\mathrm{G}}(X_{dR}).$ Moreover, there is an induced symmetric monoidal functor
$$\EuScript{G}_{\eta}^{\mathrm{corr}}:\EuScript{G}_{\mathrm{F}}^{\mathrm{corr}}\rightarrow \EuScript{G}_{\mathrm{G}}^{\mathrm{corr}},$$
compatible with projections to $ \mathrm{Corr}\big(\mathrm{PreStk}\big).$
\end{proposition}
\begin{proof}
To prove the claim we note that is enough to show this on affines $S$ and right-Kan extend to prestacks. The second claim follows from the first i.e. immediately from the non-correspondence version and the observation that there is a homotopy commutative diagram
\[
\begin{tikzcd}
\EuScript{G}_{\mathrm{F}}^{\mathrm{corr}}\arrow[dr] \arrow[rr] & & \arrow[dl] \EuScript{G}_{\mathrm{G}}^{\mathrm{corr}}
\\
& \mathrm{Corr}\big(\mathrm{PreStk}\big) &
\end{tikzcd}
\]
forgetting the algebra structure composing with  $Comm\big( \mathrm{Corr}\big(\mathrm{PreStk}\big)\big)\rightarrow  \mathrm{Corr}\big(\mathrm{PreStk}\big).$
\end{proof}
Letting $\mathrm{F}:=\mathrm{PreStk}_{/(-)}$ and $\mathrm{G}=\mathrm{QCoh},$ the natural transformation is specified on affines $S$, by 
$\eta(S):\mathrm{PreStk}_{/S}\rightarrow \mathrm{QCoh}(S)^{\mathrm{op}},$
which sends a relative prestack $\EQ\xrightarrow{p}S$ to the quasi-coherent sheaf $p_{\bullet}^{\mathrm{QCoh}}(\EuScript{O}_{\EQ})$ on $S$. By Proposition \ref{prop: NatTransformOnFactorizations} there is an induced functor 
$\mathrm{FactSpc}_{S}\rightarrow \mathrm{FactQCoh}_{S},$
from factorization spaces over $S$ to factorization quasi-coherent sheaves.
Our interest is in the case where
$S$ is taken to be $X_{dR}.$ In this situation we obtain a functor from factorization $\mathcal{D}$-spaces over $X$ to factorization $\mathcal{D}$-modules
$$\mathrm{Lin}_{X}:\mathrm{FactSpc}_{\mathcal{D}}(X)\rightarrow \mathsf{FAlg}_{\mathcal{D}}(X).$$

To be able to understand the global categorified nature of sheaves on factorization $\mathcal{D}$-spaces, we use the language of factorization sheaves of categories. 
\vspace{2mm}

\noindent\textbf{Sheaves on Factorization $\D$-Spaces.}
Let $f:\EQ^{Ran}\rightarrow Ran_{X_{dR}},$ be a derived factorization $\D$-space.
Let us describe its associated $(\infty,1)$-category
$\EuScript{QC}\mathrm{oh}_{\EQ^{Ran}},$ as a quasi-coherent sheaf of categories on $Ran_{X_{dR}}$ \cite{Butson2}. Roughly speaking, it is comprised of the datum:
\begin{itemize}
    \item for every finite set $I$ we have a sheaf of categories on $X_{dR}^I$ given by $f_{I,*}\EuScript{QC}\mathrm{oh}_{\EQ^{(I)}},$ 
    
    \item For all $\alpha:I\rightarrow J$ we have a functor
    $$f_{J,*}(\varphi_{\alpha}^{\bullet}):f_{J,*}\EuScript{QC}\mathrm{oh}_{\EQ^{(J)}}\rightarrow f_{J,*}\varphi_{\alpha,*}\varphi_{\alpha}^*\EuScript{QC}\mathrm{oh}_{\EQ^{(J)}}\xrightarrow{\simeq}\Delta(\alpha)^*f_{I,*}\EuScript{QC}\mathrm{oh}_{\EQ^{(I)}},$$
    
    \item Homotopy coherent compatibilities for compositions
    
    \item Factorization equivalences, $j(\alpha)^*f_{I,*}\EuScript{QC}\mathrm{oh}_{\EQ^{(I)}}\xrightarrow{\simeq} j(\alpha)^*\bigboxtimes_{j\in J}f_{I_j,*}\EuScript{QC}\mathrm{oh}_{\EQ^{(I_j)}}.$
\end{itemize}

This is equivalent to requiring that the existence of functors
$$\Delta(\alpha)^*f_{I,*}\EuScript{QC}\mathrm{oh}_{\EQ^{(I)}}\xrightarrow{\simeq}\widetilde{f}_{I,*}\EuScript{QC}\mathrm{oh}_{X^J\times_{X^I}\EQ^{(I)}}\rightarrow f_{J,*}\EuScript{QC}\mathrm{oh}_{\EQ^{(J)}},$$
with $\widetilde{f}_{I}:X^J\times_{X^I}\EQ^I\rightarrow X^I$ the induced maps.

Supposing that the structure map
$f_{\EQ}:\EQ\rightarrow Ran_{X_{dR}}$ induces a corresponding factorization functor
$$f_{\EQ}^*:\EuScript{QC}oh_{Ran_{X_{dR}}}\rightarrow \EuScript{QC}oh_{\EQ}.$$
Then, one has that $\EuScript{O}_{\EQ}:=f_{\EQ}^*(\mathcal{O}_{Ran_{X_{dR}}})$ is a factorization-compatible quasi-coherent sheaf on $\EQ.$

A \emph{factorizable linearization functor} for $\EQ$ is a factorization functor
$$\Phi:\EuScript{QC}oh_{\EQ}\rightarrow \EuScript{QC}oh_{Ran_{X_{dR}}}.$$

\begin{proposition}
\label{proposition: PhiTilde}
Given any factorizable linearization functor for $\EQ$, there is an induced functor $\widetilde{\Phi}$ on the categories of internal factorization algebra objects. 
\end{proposition}

Given a factorization $\D$-prestack with structure map $\mathcal{X}^{Ran}\xrightarrow{\rho^{Ran}}Ran_{X_{dR}}$ which is moreover pseudo-proper, it produces a $\mathcal{D}$-factorization $!$-sheaf object
$$\big(\rho^{Ran}\big)_!^{\mathrm{Shv}}\mathcal{O}_{\mathcal{X}^{Ran}}:=\big\{(\rho^{(I)}\big)_!^{\mathrm{Shv}}\big(\mathcal{O}_{\mathcal{X}^{(I)}}\big)\big\}_{I\in \mathrm{fSet}},$$
and the latter objects are nothing but $!$-sheaf variants of factorization $\mathcal{D}$-modules.
Moreover, due to the assumptions on the factorization $\mathcal{D}$-space we have that $\big(\rho^{(I)}\big)_!^{\mathrm{Shv}}$ agrees with the usual $\mathcal{D}$-module pushforward (de Rham pushforward).

\begin{example}
\normalfont 
If our factorization $\mathcal{D}$-space is a colimit of affine spaces corresponding to relative prestacks $\mathcal{X}^{(I)}\rightarrow X_{dR}^I$ then each $\mathcal{X}^{(I)}$ is representable and the de Rham pushforward of its quasi-coherent sheaf $\mathcal{O}_{\mathcal{X}^{(I)}}$ is a $\mathcal{D}_{X^I}$-module. It determines a representable $\mathcal{D}_{X^I}$-space $\mathrm{Spec}_{\mathcal{D}}(\mathcal{A}^{(I)})$ for the push-forward commutative  $\mathcal{D}_{X}^I$-algebra $\mathcal{A}^{(I)}:=\rho^{(I)}_*\big(\mathcal{O}_{\mathcal{X}^{(I)}}\big).$
\end{example}

\subsubsection{Internal Algebras with Factorization and Connections.}
A factorization compatible sheaf on a factorization $\mathcal{D}$-space $\mathcal{X}^{Ran}$ is, by definition, a functor
$$\mathrm{A}_{\mathcal{X}^{Ran}}\in \mathsf{Fun}_{X_{dR}}^{fact}\big(\EuScript{S}hv_{Ran_{X_{dR}}},\EuScript{S}hv_{\mathcal{X}^{Ran}}\big),$$ which satisfies the property of factorization (a `factorization functor').

The factorization categories $\EuScript{S}hv_{Ran_X}$ and $\EuScript{S}hv_{Ran_{X_{dR}}}$ are determined for each $I$ by $\EuScript{S}hv_{X^I}$ and $\EuScript{S}hv_{X_{dR}^I},$ which using the global section functor for sheaves of categories, tautologically give 
$$\mathbf{\Gamma}\big(X^I,\EuScript{S}hv_{X^I}\big)\simeq \mathrm{Shv}^!(X^I),\hspace{5mm} \mathbf{\Gamma}\big(X_{dR}^I,\EuScript{S}hv_{X_{dR}^I}\big)\simeq \mathrm{Shv}^!(X_{dR}^I)\simeq \mathsf{Mod}(\D_{X^I}).$$
The factorization algebra objects internal to each category are given by usual factorization sheaves on $X$ and factorization $\D$-module objects respectively. By Proposition \ref{proposition: PhiTilde} there is an induced functor between them.

\begin{definition}
\normalfont 
Let $\EQ$ be a unital factorization $\mathcal{D}$-space with unital factorization category $\mathrm{Shv}_{\EQ}^{Ran}.$ A \emph{unital factorization compatible $!$-sheaf on $\EQ^{Ran}$} is a $\mathcal{D}$-factorization algebra object internal to $\mathrm{Shv}_{\EQ}^{Ran}.$
\end{definition}
We denote such an object in the following form:
$$\EuScript{A}_{\mathcal{X}}^{Ran}\in \mathsf{F}\mathsf{Alg}^{\mathrm{un}}\big[\mathrm{Shv}_{\mathcal{X}}^{Ran}\big],$$ and is determined by the following data:

\begin{itemize}
    \item For any finite set $I$, there is $\EuScript{A}_{\mathcal{X}}^{(I)}\in \mathrm{Shv}^!\big(\mathcal{X}^{(I)}\big),$ 
    
    \item For all $\alpha:I\rightarrow J$, there is a morphism
    $$\varkappa^{\alpha}:\nu_{\alpha\bullet}\mu_{\alpha}^{\bullet}\widetilde{\Delta}_{\alpha}^{\bullet}\EuScript{A}_{\mathcal{X}}^{(I)}\rightarrow \EuScript{A}_{\mathcal{X}}^{(J)}\hspace{3mm} \text{in }\hspace{1mm} \mathrm{Shv}^!(\mathcal{X}^{(J)},$$ induced by the maps (\ref{eqn: Unital Fact Cat Structure Maps}) acting by push-forward and pull-back on objects of the sheave of categories over $\mathcal{X}^{(I)};$
    
    \item There is a factorization equivalence:
    $$\widetilde{j}(\alpha)^!\EuScript{A}_{\mathcal{X}^{(I)}}\xrightarrow{\simeq}\widetilde{j}'(\alpha)^!\bigboxtimes_{j\in J}\EuScript{A}_{\mathcal{X}^{(I_j)}},
    \hspace{1mm} \text{in }\mathrm{Shv}_{U(\alpha)\times_{X_{dR}^I}\mathcal{X}^{(I)}}\simeq\mathrm{Shv}_{U(\alpha)\times_{X_{dR}^I}\prod_j\mathcal{X}^{(I_j)}}.$$ 

\end{itemize}
 The maps $\widetilde{\Delta}_{\alpha}:X_{dR}^J\times_{X_{dR}^I}\mathcal{X}^{(I)}\rightarrow \mathcal{X}^{(J)},$ are the canonical ones covering $\Delta_{\alpha}:X_{dR}^J\rightarrow X_{dR}^I$ and $\widetilde{j}(\alpha):U(\alpha)\times_{X_{dR}^I}\mathcal{X}^{(I)}\rightarrow \mathcal{X}^{(I)}$ is the `lift' of $j(\alpha):U(\alpha)\hookrightarrow X_{dR}^I.$ 
 
Moreover if $\alpha$ is a partition, we require $\varkappa^{\alpha}$ is an isomorphism.

Finally, we require for all injections $\alpha:I\hookrightarrow J,$ there are morphisms (abuse of notation, as here we omit the reference to the image of the functors),
$$\EuScript{A}_{\mathcal{X}^{(I)}}\boxtimes\mathcal{O}_{X_{dR}^{I_{\alpha}}}\rightarrow \EuScript{A}_{\mathcal{X}^{(J)}}.$$

\begin{proposition}
\label{Factorization Push-Forward Preserves Factorization Algebras}
Suppose that $f:\mathcal{X}\rightarrow \EQ$ is a morphism of factorization $\mathcal{D}$-spaces admitting  factorization push-forward functors $f_{*}^{Ran}$. Then there is an induced functor on factorization compatible $!$-sheaves
$$f_{\mathrm{alg},*}^{Ran}:\mathsf{F}\mathsf{Alg}\big[\mathrm{Shv}_{\mathcal{X}}^{Ran}\big]\rightarrow \mathsf{F}\mathsf{Alg}\big[\mathrm{Shv}_{\EQ}^{Ran}\big].$$
\end{proposition}

Consider the diagram
\[
\begin{tikzcd}
& U(\alpha)\arrow[dd,"j(\alpha)"]& 
\\
U(\alpha)\times_{X_{dR}^I}\prod_{j\in J}\EQ^{(I_j)}\arrow[ur,"\prod_j\rho^{(I_j)''}"] \arrow[d,"j^{''}"] \arrow[rr,"\mathbf{fact}_{\alpha}"] & & \arrow[ul,"\rho^{(I)'}"] U(\alpha)\times_{X_{dR}^I}\EQ^{(I)}\arrow[d,"j"] 
\\
\prod_{j\in J}\EQ^{(I_j)}\arrow[r,"\prod_{j}\rho^{(I_j)}"] & X_{dR}^I & \arrow[l,"\rho^{(I)}"] \EQ^{(I)}
\end{tikzcd}
\]
A vector $\mathcal{D}^!$-bundle on $\EQ^{Ran}$ is a family $\mathcal{M}_{\EQ}^{Ran}:=\{\mathcal{M}_{\mathcal{X}}^{(I)}\in \mathcal{D}(\EQ^{(I)})\}$ with each $\mathcal{M}_{\EQ^I}$ an $\EQ^I$-locally finitely generated and projective $\mathcal{O}_{\EQ^I}[\mathcal{D}_X^I]$-modules which are moreover factorization compatible i.e. there are natural equivalences for all $\alpha\in \mathrm{Part}_J(I)$,
\begin{equation}
    \label{eqn: Factorization Compats for D-modules on FSpc}
    (j')^*\mathcal{M}_{\EQ}^{(I)}\xrightarrow{\simeq }\big(\mathbf{fact}_{\alpha}\big)_*(j'')^*\bigboxtimes_{j\in J}\mathcal{M}_{\EQ}^{(I_j)}.
\end{equation}
Equivalences (\ref{eqn: Factorization Compats for D-modules on FSpc}) are compatible with composition of $\alpha$ and with equivalences  $\EQ_{\alpha}^!\mathcal{M}_{\EQ}^{(I)}\simeq \mathcal{M}_{\EQ}^{(J)}.$
In this case, one verifies that the family 
$$\mathcal{A}_{\EQ}^{Ran}:=\big\{\mathcal{A}_{\EQ}^{(I)}:=\rho_{\EQ,*}^{(I)}\mathcal{M}_{\EQ}^{(I)}\in \mathsf{Mod}(X^I)\big\},$$ is a factorization algebra module object.
Indeed, notice first that $$\Delta(\alpha)^!\rho_{\EQ,*}^{(I)}\mathcal{M}_{\EQ}^{(I)}\simeq \rho_{\mathcal{X},*}^{(J)}\EQ_{\alpha}^!\mathcal{M}_{\EQ}^{(I)}\xrightarrow{\simeq}\rho_{\EQ,*}^{(J)}\mathcal{M}_{\EQ}^{(J)}.$$

Then, the factorization equivalences are given by
\begin{eqnarray*}
j(\alpha)^*\rho_*^{(I)}\mathcal{M}_{\EQ}^{(I)}&\simeq & (\rho^{(I)})_*'(j')^*\mathcal{M}_{\EQ}^{(I)}
\\
&\simeq &(\rho^{(I)})_*'(\mathbf{fact}_{\alpha})_*(j'')^*\big(\bigboxtimes_{j\in J}\mathcal{M}_{\EQ}^{(I_j)}\big)
\\
&\simeq& \big(\prod_{j}\rho^{(I_j)}\big)_*''(j'')^*\bigboxtimes_{j}\mathcal{M}^{(I_j)}
\\
&\simeq& j(\alpha)^*\big(\prod_j\rho^{(I_j)}\big)_*\bigboxtimes_j\mathcal{M}^{(I_j)}
\\
&\simeq & j(\alpha)^*\bigboxtimes_{j\in J}\rho_*^{(I_j)}\mathcal{M}^{(I_j)}.
\end{eqnarray*}

\subsection{Internal Factorization Hom-Algebras}
\label{sssec: Internal Factorization Hom-Algebra $!$-Sheaves}
Let $\EQ\rightarrow X$ be affine, with algebra $\mathcal{A}^{\ell},$ with $\mathcal{F}(\mathcal{A})$ the associated factorization algebra via Proposition \ref{Multi-jet r-adjoint proposition}. Let $\mathcal{Z}_{Ran}\xrightarrow{p}Ran_{X_{dR}}$ be linearizable i.e. $\mathcal{B}_{\mathcal{Z}_{Ran}}=p_*\mathcal{O}_{\mathcal{Z}_{Ran}}$ is a commutative unital $\mathcal{D}_{Ran_X}$-algebra. Then,
\begin{eqnarray}
\label{eqn: RSol Fact}
\RS_{X}(\mathcal{Z}^{Ran},\mathcal{F}(\EQ)\big)&\simeq&\mathsf{Maps}_{/X_{dR}}(X_{dR}\times_{Ran_{X_{dR}}}\mathcal{Z}_{Ran},\mathcal{F}\EQ)
\\
&\simeq&\mathsf{Maps}_{\mathsf{CAlg}(\mathcal{D}_{Ran_X})}(\mathcal{F}\mathcal{A},p_*\mathcal{O}_{\mathcal{Z}_{Ran}}) \nonumber ,\end{eqnarray}
which is equivalent to $\mathbb{R}Sol_{\mathcal{D}}(\mathcal{A},\Delta^*p_*\mathcal{O}_{\mathcal{Z}_{Ran}}).$

Spaces of maps (\ref{eqn: RSol Fact}) are not necessarily compatible with factorization structures, therefore to understand (derived) global sections in a factorization compatible manner, we must formalize when taking factorization compatible maps out of a factorization compatible sheaf $\EuScript{A}_{\mathcal{X}}^{Ran}$ over a factorization $\mathcal{D}$-space into any other such object $\EuScript{B}_{\mathcal{X}}^{Ran}$ provides another factorization compatible object.
\vspace{1.5mm}

When this situation holds, following the terminology of \cite{Butson2}, we say that
$\EuScript{A}_{\mathcal{X}}^{Ran}\in \mathsf{F}\mathsf{Alg}\big[\mathrm{Shv}_{\mathcal{X}}^{Ran}\big]$ \emph{admits internal hom-sheaf objects over $Ran_{X_{dR}}$}. In other words, there exists a unital factorization functor
\begin{equation}
    \label{eqn: Internal !-factorization compatible sheaf hom object over XdR}
\underline{\EuScript{H}\mathrm{om}}_{\EuScript{S}hv_{\mathcal{X}}^{Ran}}\big(\EuScript{A}_{\mathcal{X}}^{Ran},-\big):\EuScript{S}hv_{\mathcal{X}}^{Ran}\rightarrow \EuScript{S}hv_{Ran_{X_{dR}}}^{Ran},
\end{equation}
which coincides upon taking global sections (of sheaves of categories) with the internal hom functor in sheaves over $X_{dR}^I$:
$$\underline{\mathrm{Hom}}_{\mathcal{D}_{X^I}}\big(\EuScript{A}_{\mathcal{X}}^{(I)},\EuScript{B}_{\mathcal{X}}^{(I)}\big).$$

A factorization algebra $\EuScript{A}_{\mathcal{X}}^{Ran}$ over $\mathcal{X}^{Ran}$ is accordingly said to \emph{admit internal-hom factorization algebras}
 if there exists a unital factorization functor 
\begin{equation*}\underline{\hat{\EuScript{H}}}:=\underline{\hat{\EuScript{H}}\mathrm{om}}(\EuScript{A}_{\mathcal{X}},-):\EuScript{S}hv_{\mathcal{X}}^{Ran}\rightarrow \EuScript{S}hv_{\mathcal{X}}^{Ran},\hspace{1mm}\text{ such that }
\adjustbox{scale=1.02}{
\begin{tikzcd}
\EuScript{S}hv_{\mathcal{X}}^{Ran}\arrow[dr,"(\ref{eqn: Internal !-factorization compatible sheaf hom object over XdR})"] \arrow[r,"\underline{\hat{\EuScript{H}}}"] & 
\EuScript{S}hv_{\mathcal{X}}^{Ran}\arrow[d,"\rho_{*}^{Ran}"]
\\
& \EuScript{S}hv_{Ran_{X_{dR}}}^{Ran},
\end{tikzcd}}
\end{equation*}
commutes (up to homotopy).

By Proposition \ref{Factorization Push-Forward Preserves Factorization Algebras}, the factorization push-forward functor for the structure map $\rho^{Ran}$ induces a map
of factorization compatible objects. Explicitly,
for any $\EuScript{B}_{\mathcal{X}}^{Ran}\in \mathsf{F}\mathsf{Alg}\big[\mathrm{Shv}_{\mathcal{X}}^{Ran}\big],$ the internal $!$-hom factorization sheaf over $\mathcal{X}^{Ran}$ is determined by the assignment
$$I\mapsto \underline{\hat{\EuScript{H}}\mathrm{om}}\big(\EuScript{A}_{\mathcal{X}},\EuScript{B}_{\mathcal{X}}\big)^{(I)}:=\rho_{*}^{(I)}\underline{\mathrm{Hom}}_{\mathsf{Shv}(\mathcal{X}^{(I)})}\big(\EuScript{A}_{\mathcal{X}}^{(I)},\EuScript{B}_{\mathcal{X}}^{(I)}\big)\in \mathsf{Mod}(\mathcal{D}_{X^I}).$$
One may readily verify the totality of these objects determine a factorization $!$-sheaf object in $\EuScript{S}hv_{Ran_{X_{dR}}}^{Ran}.$

\begin{proposition}
\label{prop: FComaptible}
Let $\EQ\in \PS_{\mathcal{D}}^{fact}(X)$ and suppose we $\Phi$ is a factorizable linearization functor for $\EQ.$
Then there is an induced functor 
$$\widetilde{\Phi}:\mathsf{FAlg}\big(\EuScript{QC}oh_{\EQ})\rightarrow \mathsf{FAlg}_{\mathcal{D}}(X),$$
from factorization compatible quasi-coherent sheaves with connection to factorization $\mathcal{D}$-modules.
\end{proposition}
Consider the functor of quasi-coherent direct image along the structure map $p_{\EQ}:\EQ\rightarrow Ran_{X_{dR}}$ on factorization compatible sheaves. Then by Proposition \ref{prop: FComaptible}, it is meaningful to say that if $\EQ$ has the property that the collection of $p_{\EQ^I,\bullet}(\mathcal{O}_{\EQ^I})\in \mathsf{Mod}(\mathcal{D}_{X^I}),$ determines an object of $\mathsf{FAlg}_{\mathcal{D}}(X),$ then $\EQ$ is \emph{linearizable}.

\subsection{Global Sections}
The functor of \emph{factorization homology}, on a 
 (not necessarily regular holonomic) $\D_{Ran_X}$-module $\mathcal{M}=(\mathcal{M}^I)$ is
$$\int_X(-):\mathsf{Mod}(\mathcal{D}_{Ran_X})\rightarrow \mathsf{Ind}\big(\mathsf{Perf}_{k}\big),\hspace{2mm} \int_X\mathcal{M}:=\underset{I\in fSet}{hocolim}\hspace{.5mm}p_*^{(I)}(\mathcal{M}^I),$$
where $p:X\rightarrow \mathrm{pt},p^I:X^I\rightarrow \mathrm{pt}$ are projections. In other words,
$$\Gamma_{dR}(X^I,\mathcal{M}^I)=p_*^{(I)}(\mathcal{M}^I)=\Gamma\big(X^I,DR_{X^I}(\mathcal{M}^I)\big),$$
with $DR$ the $\D$-module de Rham functor on $X^I.$

From the previous subsections we may extend this definition from $\mathsf{Mod}(\mathcal{D}_{Ran_X})$ 
to factorization compatible algebra objects internal to $\EuScript{QC}oh_{\EQ},$ for $\EQ\in \PS_{\mathcal{D}}^{fact}(X).$

\begin{definition}
    \label{definition: Linearized Homology}
\normalfont 
Consider a pair $(\EQ,\Phi)$ consisting of a factorization $\mathcal{D}$-space over $X$ and a factorizable linearization functor.
The \emph{$\Phi$-linearized factorization homology} of $\EQ$ is the composition
\begin{equation}
    \label{eqn: Linearized Fact Homology}
    \int_{X,\Phi}(-):=\int_X\circ \widetilde{\Phi}(-):\mathsf{FAlg}\big(\EuScript{QC}oh_{\EQ}\big)\rightarrow \mathsf{FAlg}_{\mathcal{D}}(X)\rightarrow \mathsf{Vect}_k,
\end{equation}
\end{definition}
Recall the induced functor $\widetilde{\Phi}$ as in Proposition \ref{proposition: PhiTilde}.

In particular, for $\EuScript{A}_{\EQ}\in \mathsf{FAlg}(\EuScript{QC}oh_{\EQ})$ then we obtain, by taking the quasi-coherent factorization push-forward 
$$\underset{I\in fSet}{hocolim}\hspace{.5mm}\Gamma_{dR}\big(X^I,(p_{\EQ^I,\bullet}^{\mathrm{QCoh}}(\EuScript{A}_{\EQ^I})\big).$$
If $\EQ$ is linearizable, then its \emph{linearized factorization homology}, denoted $\mathbb{H}(\EQ),$ is defined to be the functor (\ref{eqn: Linearized Fact Homology}) for factorization push-forward along $p:\EQ\rightarrow Ran_{X_{dR}}$ of its structure sheaf $\mathcal{O}_{\EQ}$ in $\mathsf{FAlg}(\EuScript{QC}oh_{\EQ}).$
\begin{remark}
From the point of view of PDE geometry, we start with an affine $\D$-prestack, rather than a factorization $\mathcal{D}$-space. In order to get a linearized cohomology which take the universal homotopical multi-jet construction Proposition \ref{Multi-jet r-adjoint proposition}. In this case, the resulting factorization $\D$-space is commutative and its linearized homology agrees with global sections of its solution stack $\RS_X(\EQ).$
\end{remark}
We always view vector spaces as particular instance of $\mathcal{D}$-modules on the point via the canonical map $a:X\rightarrow pt$ and the induced pull-back functor $a_{\mathcal{D}}^*:\mathsf{Mod}(\mathcal{D}_{pt})\rightarrow \mathsf{Mod}(\mathcal{D}_X).$ Precisely, $a_{\mathcal{D}}^*:=a_{dR}^*$ via quasi-coherent pull-back.

\begin{theorem}\emph{(\textcolor{blue}{Theorem B (i)})}
\label{theorem: Theorem B Main claim}
Let $p:\EQ\rightarrow X$ be a $\D$-finitary non-linear PDE and put $\mathcal{A}:=p_*(\mathcal{O}_{\EQ})$ with $\mathcal{A}_{Ran_X}:=\mathcal{F}(p_*\mathcal{O}_{\EQ}).$ Then the linearized homology of the associated multi-jet space is equivalent to the factorization homology of its derived $\D$-algebra of functions which admits a natural morphism
$$\gamma:\int_X^{fact} \mathcal{A}_{\EQ}\rightarrow \Gamma\big(\RS(\EQ),\mathcal{O}_{\RS(\EQ)}\big).$$
It is functorial with respect to morphisms of derived $\D$-algebras.
\end{theorem}
\begin{proof}
By our assumptions, we have that $\mathcal{A}\in \mathsf{CAlg}_{X}(\mathcal{D}_X).$
Consider (\ref{eqn: Point of RSol}) i.e. the diagram formed by $U$-parameterized solutions:
\begin{equation*}
\begin{tikzcd}
    U\times X_{dR}\arrow[d,"q_U"]\arrow[r,"\varphi_U\times id"] \arrow[rr, bend left, "j_{\infty}(\varphi_U)"] & \RS_X(\EQ)\times X_{dR}\arrow[d,"\rho"] \arrow[r,"ev_{\varphi_U}"] & \EQ
    \\
    U\arrow[r,"\varphi_U"] & \RS_X(\EQ) & 
\end{tikzcd}
\end{equation*}
and consider the morphism $\RS_X(\EQ)\times X_{dR}\rightarrow \EQ.$ Together with $p_{dR}:X\rightarrow X_{dR}$, this supplies a morphism of commutative monoids in (ind-coherent) sheaves on $X_{dR}$ i.e. commutative $\mathcal{D}$-algebras,
$$\mathcal{A}^{\ell}\rightarrow \mathcal{O}_{\RS}\otimes\mathcal{O}_X,$$
and consequently by Proposition \ref{Multi-jet r-adjoint proposition}, a morphism 
$\mathcal{F}(p_*\mathcal{O}_{\EQ})\rightarrow \mathcal{O}_{\RS}\otimes\omega_{Ran_X},$
in commutative monoids in $\mathcal{D}_{Ran_X}.$
Now consider the pull-back functor $a_{\mathcal{D}}^*,$ which induces a functor on commutative algebra objects,
$$\mathsf{CAlg}_k\rightarrow \mathsf{CAlg}_X(\mathcal{D}_X),V\mapsto a^*(V)\simeq V\otimes \mathcal{O}_X.$$
We prove the result by identifying its left-adjoint.

It is given by $\int_{X}^{fact}\circ \mathcal{F},$ 
with $\int_X^{fact}$ the functor of factorization homology in the presence of an operad\footnote{This was described for the Lie operad in \cite{FraGai} (chiral homology) and for the commutative operad in \cite{KRY}.} $\mathcal{P}$, specified to the case with $\mathcal{C}omm.$

Briefly, there are functors 
\begin{equation}
    \label{eqn: P-operad Chiral homology functor}
    \mathbb{R}\Gamma_{dR}\big(Ran_X,-\big)_{\mathcal{P}}:\mathcal{P}-\mathsf{Alg}\big[\mathsf{Mod}\big(\mathcal{D}_{Ran_X}\big)\big]\rightarrow \mathcal{P}-\mathsf{Alg}\big[\mathsf{Vect}_k\big],
\end{equation}
for an algebraic operad $\mathcal{P}.$ Moreover, the forgetful functors
$$\mathcal{P}-\mathsf{Alg}\big[\mathsf{Mod}\big(\mathcal{D}_{Ran_X}\big)\big]\rightarrow \mathsf{Mod}\big(\mathcal{D}_{Ran_X}\big),$$
 are conservative and commute with limits, thus possess left-adjoint functors $\mathrm{Free}_{\mathcal{P}}$, corresponding to taking free $\mathcal{P}$-algebras.

The embedding of the main diagonal $\Delta:X\rightarrow Ran_X$ induces a fully-faithful embedding $\Delta_*:\mathsf{Mod}\big(\mathcal{D}_X\big)\rightarrow \mathsf{Mod}\big(\mathcal{D}_{Ran_X}\big)$ which is moreover lax symmetric monoidal. It sends algebra objects to algebra objects and we denote the induced functor by
$$\Delta_*^{\mathcal{P}}:\mathcal{P}-\mathsf{Alg}\big[\mathsf{Mod}\big(\mathcal{D}_{X}\big)\big]\rightarrow \mathcal{P}-\mathsf{Alg}\big[\mathsf{Mod}\big(\mathcal{D}_{Ran_X}\big)\big].$$
Kan extension of (\ref{eqn: P-operad Chiral homology functor}) along the functor $\Delta_*^{\mathcal{P}}$ gives
\begin{equation}
    \label{eqn:Left-kan extension}
    \mathbb{R}\Gamma_{dR}\big(X\subset Ran_X,-\big)_{\mathcal{P}}:=\mathrm{Kan}_{\Delta_*^{\mathcal{P}}}:\mathcal{P}-\mathsf{Alg}\big[\mathsf{Mod}\big(\mathcal{D}_{X}\big)\big]\rightarrow \mathcal{P}-\mathsf{Alg}\big[\mathsf{Vect}\big].
\end{equation}
We have thus put
$$\int_X^{fact}:=\mathbb{R}\Gamma_{dR}(X\subset Ran_X,-)_{\mathcal{C}omm}.$$ 
For any $R\in \mathsf{CAlg}_k$ consider $a_{\mathcal{D}}^*$ and notice that
$$\mathsf{Maps}(Spec(R),\RS(\EQ))\simeq \mathsf{Maps}_{\mathsf{CAlg}_X(\mathcal{D}_X)}(\mathcal{A}^{\ell},a_{\mathcal{D}}^*(R))\simeq \mathsf{Maps}_{\mathsf{CAlg}_k}(\int_X^{fact}\mathcal{A}^{\ell},R).$$
Therefore any commutative $\mathcal{D}$-algebra morphism to the `constant' $\mathcal{D}_X$-algebra $a_{\mathcal{D}}^*(R)$ of the form $\mathcal{A}^{\ell}\rightarrow R\otimes\mathcal{O}_X,$ is equivalent to a morphism
$$\int_X^{fact}\mathcal{A}^{\ell}\rightarrow R.$$
One may observe from the identification
$\RS(\EQ)\simeq Spec\big(\int_X^{fact}\mathcal{A}^{\ell}\big),$
there is a natural equivalence 
$$\int_X^{fact}\mathcal{A}^{\ell}\rightarrow \Gamma\big(\RS(\EQ),\mathcal{O}_{\RS(\EQ)}\big).$$
Now let $f:\mathcal{A}^{\ell}\rightarrow \mathcal{B}^{\ell}$ be a morphism in $\mathsf{CAlg}_X(\mathcal{D}_X)$ and consider its image $\mathcal{F}(f):\mathcal{F}(\mathcal{A}^{\ell})\rightarrow \mathcal{F}(\mathcal{B}^{\ell}),$ via the fully-faithful adjoint functor in Proposition \ref{Multi-jet r-adjoint proposition}. 
One can show that this morphism is compatible with the homotopy colimits and thus yields a map
$$\int_X^{fact}\mathcal{A}^{\ell}\rightarrow \int_X^{fact}\mathcal{B}^{\ell},\hspace{1mm}\text{ in } \mathsf{CAlg}_k.$$

Consider the base-change $\mathcal{B}^{\ell}\otimes_{\mathcal{A}^{\ell}}^{\mathbb{L}}a_{\mathcal{D}}^*(R)$ as an object in
$$\mathsf{CAlg}(\mathsf{QCoh}(X_{dR})\otimes \mathsf{Mod}(R)\big)\simeq \mathsf{CAlg}_X(\mathcal{D}_X)\otimes \mathsf{CAlg}_{R/}.$$ Again by application of $\mathcal{F}(-)$, there is an object
$$\mathcal{F}\big(\mathcal{B}^{\ell}\otimes_{\mathcal{A}}^{\mathbb{L}}a_{\mathcal{D}}^*(R)\big)\simeq \mathcal{F}(\mathcal{B}^{\ell})\otimes^{\mathbb{L}}a_{\mathcal{D}}^*(R)\in \mathsf{FAlg}_{\mathcal{D}}(X)_{comm}\otimes\mathsf{CAlg}_{R/}.$$
We can now take global de Rham sections along the first component and sheaf-theoretic global sections along $Spec(R)$. In particular, since there is a natural inclusion 
$$\mathcal{F}(\mathcal{B}^{\ell})\rightarrow \mathcal{F}(\mathcal{B}^{\ell})\otimes a_{\mathcal{D}}^*(R),$$ 
one can verify that by the above adjunctions, that there is a commutative diagram in $\mathsf{CAlg}_k$:
\[
\begin{tikzcd}
    \underset{I}{hocolim}\hspace{1mm}\Gamma_{dR}(X^I,\mathcal{F}(\mathcal{A})^I)\arrow[d]\arrow[r] & R\arrow[d]
    \\
    \underset{I}{hocolim}\hspace{1mm}\Gamma_{dR}(X^I,\mathcal{F}(\mathcal{B})^I)\arrow[r] & \big(\underset{I}{hocolim}\Gamma_{dR}(X^I,-)\otimes a_{*}\big)\mathcal{F}\big(\mathcal{B}^{\ell}\otimes_{\mathcal{A}}^{\mathbb{L}}a_{\mathcal{D}}^*(R)\big).
\end{tikzcd}
\]
In other words, setting
$$\int_{X\otimes Spec(R)}^{fact}:\mathsf{FAlg}_{\mathcal{D}}(X)_{comm}\otimes \mathsf{CAlg}_{/R}\rightarrow \mathsf{CAlg}_{R/},$$
then one has 
$$\int_{X\otimes Spec(R)}^{fact}\mathcal{B}\simeq \int_X^{fact}\mathcal{B}\otimes_{\mathcal{A}}^{\mathbb{L}}a_{\mathcal{D}}^*(R)\simeq \int_X^{fact}\mathcal{B}\otimes_{\int_X^{fact}\mathcal{A}^{\ell}}R.$$
\end{proof}
\begin{proposition}[\textcolor{blue}{Theorem B (ii)}]
\label{proposition: Thm B (ii)}
Consider the map $\gamma$ in Theorem \emph(\ref{theorem: Theorem B Main claim}), and the full-subcategory of $\mathsf{DAff}_X(\mathcal{D}_X)$ whose spaces are of the form
\begin{equation}
\label{eqn: Eventually Coconnective}
\{Spec_{\mathcal{D}}(\mathcal{A})|H_{\mathcal{D}}^{i}(\mathcal{A})=0, i<<n,\text{ for some } n<0\},
\end{equation}
and $H_{\mathcal{D}}^0(\mathcal{A})$ is a commutative $\mathcal{D}$-algebra of finite type.
Then the functorial assignment defined by $\gamma$ induces an equivalence between \emph{(\ref{eqn: Eventually Coconnective})} and the
category of eventually coconnective derived affine schemes.
\end{proposition}
\begin{proof}
    We need to show that restricting $\gamma$ gives an equivalence of monoids. This follows by \ref{ssec: Finiteness}, and the fact that since 
$$\int_X^{fact}\mathcal{A}^{\ell}= \underset{I}{hocolim}\Gamma_{dR}(X^I,\mathcal{F}(\mathcal{A})^I),$$ via the composition of forgetful functors 
$$\mathsf{CAlg}_X(\mathcal{D}_{X^I})^{\leq 0}\hookrightarrow \mathsf{CAlg}_X(\mathcal{D}_{Ran_X})\rightarrow \mathsf{Mod}(\mathcal{O}_{X^I})^{\leq 0}\subset \mathsf{Mod}(\mathcal{O}_{Ran_X}),$$
we get $U(\mathcal{F}\mathcal{A}^I)\in \mathsf{Mod}(\mathcal{O}_{X^I})^{\leq 0}.$ It follows that each component of $hocolim\Gamma_{dR}(X^I,\mathcal{F}\mathcal{A}^I)$ is an object of $\mathsf{Vect}_k^{\leq 0}.$
In a slight abuse of notation, the equivalence is obtained then by putting $R^{\bullet}:=\Gamma(\RS(\EQ),\mathcal{O}_{\RS(\EQ)})\in \mathsf{CAlg}_k,$ with corresponding affine derived stack 
$\mathbf{Spec}_k(R^{\bullet}).$
It is thus eventually coconnective, since $\mathcal{A}^{\ell}$ is an object in (\ref{eqn: Eventually Coconnective}) by assumption. 
\end{proof}

\begin{corollary}
\label{Thm B corollary}
   Suppose that $\mathcal{I}\rightarrow \mathcal{A}\rightarrow \mathcal{B}$ is a $\mathcal{D}$-algebraic non-linear PDE with $\EQ$ the corresponding derived $\mathcal{D}$-space $\mathbb{R}Spec_{\mathcal{D}}(\mathcal{B})$, that is moreover finitely $\mathcal{D}$-presented with multi-jet space $\EQ_{Ran_X}.$ The space of global sections in Theorem \ref{theorem: Theorem B Main claim} are equivalent to the derived stack
$\Gamma(Ran_{X_{dR}},\EQ_{Ran_X})$ with the property that its classical truncation (i.e. $\mathbf{Spec}_k(R^{\bullet})^{cl}$) coincides with
a non-local solution space on sections $Sect(X,E).$
\end{corollary}
This follows from Theorem \ref{theorem: Theorem B Main claim} and some general facts about derived spaces of flat sections of prestacks over $X_{dR},$ in \cite{KSYI}. For the sake of being self-contained we present a brief discussion.

Prestacks over $X_{dR}$ for instance, those of the form (\ref{eqn: Derived Local solution D space}) associated to sequences (\ref{eqn: SES}), can be extended to $Ran_{X_{dR}},$ so we must show via $\gamma,$ their space of global sections coincide as a usual derived stack with the corresponding derived space of solution-sections. To this end, consider the natural inclusion 
$$\mathbb{R}\mathsf{Maps}(X_{dR},\J(E))\hookrightarrow \mathbb{R}Sect(X,\JetX E),$$
identifying horizontal sections. Then, there is an equivalence
$$\mathbb{R}Sect(X,E)\simeq \RS(\J E),$$ and since the 
 multi-jet space of a map $\pi:E\rightarrow X$ is the by definition the multi-jet space of the stack $\J(E)\rightarrow X_{dR},$ our claim follows by noting that for $\EQ$ of the form $\mathbb{R}Sol_{\mathcal{D}},$ the alluded to object $\Gamma(Ran_{X_{dR}},\EQ_{Ran_X}),$ is the global flat sections. The result can be seen by noticing the natural equivalence 
$$\mathbb{R}Sect(X,E)\simeq \RS(\mathsf{Jets}_{dR}^{\infty}(E)^{Ran}).$$ 
At the level of classical truncations $\EQ=Sol_{\mathcal{D}}(\mathcal{B}),$ with $\mathcal{B}=\mathcal{O}(\mathrm{Jets}(E))/\mathcal{I},$ as in (\ref{eqn: SES}), so $^{cl}(\mathbb{R}Sol)$ as a subspace of $Sect(X,E),$ consists of those sections $s$ for which $j_{\infty}(s)^*:\JetX(E)\rightarrow X$ factors through $\EQ$. That is, $j_{\infty}(s)^*(\mathcal{I})\equiv 0,$ for the equation $\D_X$-ideal $\mathcal{I}.$

\begin{observation}
\normalfont
Suppose all objects admit internal homs in the sense of Subsect. \ref{sssec: Internal Factorization Hom-Algebra $!$-Sheaves}. Then one may express factorization homology in terms of a suitable functor of factorizable solutions. Write
$$\mathbb{R}Sol_{Ran_X}(-):=\mathsf{Maps}_{\mathcal{D}(Ran_X)}(-,\mathcal{O}_{Ran_X}),$$
for the functor sending $\mathcal{D}_{Ran_X}$-modules to $k_{Ran_X}$-modules. It is compatible under diagonal embeddings $\Delta:X^J\hookrightarrow X^I$ since
$$\mathbb{L}\Delta^*\mathbb{R}Sol_{X^J}(\mathcal{M}^J,\mathcal{O}_{X^J})\simeq \mathbb{R}Sol_{X^I}(\Delta^*\mathcal{M}^J,\mathcal{O}_{X^I})\simeq \mathbb{R}Sol_{X^I}(\mathcal{M}^I,\mathcal{O}_{X^I}).$$
Then usual
Verdier-duality (over each $X^I$) defined by 
$$\mathbb{D}_{Ran_X}(\mathcal{M}):=\{\mathbb{D}_{X^I}(\mathcal{M}^I)\}=\{\omega_{X^I}[d_{X^I}]\otimes\mathbb{R}\mathcal{H}om_{\mathcal{D}_{X^I}}(\mathcal{M}^I,\mathcal{D}_{X^I})\},$$
allows one to view $\mathbb{R}Sol_{Ran_X}$ as a suitable homotopy colimit of objects of the form 
$\mathbb{R}\Gamma_{dR}(X^I,\mathbb{D}_{X^I}(\mathcal{M}^I)[-d_{X^I}].$
\end{observation}

\section{Application: Variational Calculus}
\label{sec: Applications 1}
We recall how variational calculus arises in the language of $\mathcal{D}$-geometry \cite{Pau02,Pau01} and elaborate several important new homological aspects related to the tricomplex.
\subsection{Variational Calculus and the BV Formalism in $\mathcal{D}$-Geometry}
\label{ssec: Variational Calculus and BV Formalism in D Geometry}
Suppose that $\mathcal{A}$ is the $\mathcal{D}_X$-algebra corresponding to the infinite jet-bundle of a given vector bundle $\pi:E\rightarrow X,$ and recall $h(\mathcal{A})$ identifies with top forms modulo the $\D$-action. 

For example, if $dim(X)=2,$ sections of $\Omega_{\mathcal{A}}^{2,s}\otimes_{\mathcal{D}}\omega_X$ may be viewed as global sections $\Gamma(Spec_{\mathcal{D}}(\mathcal{A}),\Omega^{2,s}/d_h\Omega^{1,s}).$ Then, variational $2$-forms $\omega_{\mathcal{S}}$ associated to the action are global sections 
$\omega_{\mathcal{S}}\in \Gamma\big(Spec_{\mathcal{D}}(\mathcal{A}),\Omega^{1,2}/d^h(\Omega^{0,2})\big)$ for which $d^v\omega_{\mathcal{S}}=0.$

\begin{remark}
One makes sense of homotopy pre-symplectic forms as $(d_X-1,2)$-forms $\omega\in H^0(\EQ,\Omega_{\EQ}^{d-1,2})$ on $\EQ$, such that 
$d^v\in H^0(\EQ,d_h(\Omega^{d-2,1})).$
For lack of space, we mention only in brief the possibility of encoding also homotopy Hamiltonian vector fields (c.f. Example \ref{ex: Hamiltonian Operators}) i.e. elements $\theta\in H^0(\EQ,\Theta_{\EQ}^{\ell}),$ such that $\mathcal{L}_{\theta}(\omega)\in d_h\Omega_{\EQ}^{d-2,1},$ by naturally extending the local Lie derivative (Subsect. \ref{ssec: Local Calculus}) to the entire de Rham algebra as a derivation.
\end{remark}

Apply $h$ to the $\D$-linear universal derivation (\ref{sssec: D de Rham diff}) to obtain a map
$h\big(d\big):h(\mathcal{A})\rightarrow h(\Omega_{\mathcal{A}}^1)$ which we may then evaluate on the action by Proposition \ref{prop: 0,1 var h} (2) as $h(d)(\mathcal{S})\in h\big(\Omega_{\mathcal{A}}^1\big)$ inducing an \emph{insertion map}, 
\begin{equation}
\label{eqn:Insertion map}
i_{d\mathcal{S}}:\Theta_{\mathcal{A}}^{\ell}\rightarrow \mathcal{A}.
\end{equation}
Operation (\ref{eqn:Insertion map}) realize variational calculus in $\D$-geometry via the Euler-Lagrange ideal $\mathcal{I}_{EL}:=\mathrm{Im}\big(i_{d\mathcal{S}}\big)\subset \mathcal{A},$ and space of 
Noether identities $\mathcal{N}_{\mathcal{S}}:=\ker\big(i_{d\mathcal{S}}\big)\subset \Theta_{\mathcal{A}}.$ A Noether symmetry is a class $\xi\in h(\Theta_{\mathcal{A}})$ such that $Lie_{\xi}\mathcal{S}=0,$ via the $\mathrm{Lie}^*$ $\mathcal{A}$-algebroid action, descended on $H^0.$ 

That is, the local Lie derivative $\Theta_{\mathcal{A}}\boxtimes \mathcal{A}\rightarrow \Delta_*\mathcal{A}$ induces a map
$$h(\Theta_{\mathcal{A}}\boxtimes\mathcal{A})\simeq h(\Theta_{\mathcal{A}})\boxtimes h(\mathcal{A})\rightarrow h(\Delta_*\mathcal{A})\simeq \Delta_*h(\mathcal{A}),$$
as $h$ commutes with $\D$-module push-forwards.

In the dg-setting, one must consider the appropriate homotopical operations (\ref{eqn: Holocals}). In particular, (\ref{eqn: Derived Inner Verdier Dual}) and (\ref{eqn: Derived Local solution D space}) and taking the cone and cocone of the morphism in $Ho(\DG(\mathcal{A}))$ generalizing the map (\ref{eqn:Insertion map}).
This is based on the weight $1$ component (\ref{eqn: Weight p}) of $\mathsf{DR}^{var}(\mathcal{A}^{\ell}),$ yielding an equivalence,
\begin{equation}
    \label{eqn: Homotopy Insertion}
    \mathsf{DR}^{var}(\mathcal{A}^{\ell})(1)=p_{\infty}^*\omega_X\otimes^{\mathbb{L}}\mathbb{L}_{\mathcal{A}}\simeq \mathbb{R}\mathcal{H}om_{\mathcal{A}^{\ell}[\mathcal{D}]}(\omega_X^{\vee}\otimes\mathbb{D}_{\mathcal{A}}(\mathbb{L}_{\mathcal{A}}),\mathcal{A}^{\ell}).
    \end{equation}
Roughly speaking, (\ref{eqn: Homotopy Insertion}) is obtained via a derived double-duality statement (c.f. \ref{sssec: D de Rham diff}); the natural morphism,
$$\mathbb{L}\Omega_{\mathcal{A}^{\ell}}^1\rightarrow \mathbb{R}\mathcal{H}om_{\mathcal{A}^r\otimes\mathcal{D}_X^{op}}(\mathbb{D}_{\mathcal{A}}^{ver}(\mathbb{L}\Omega_{\mathcal{A}}^1),\mathcal{A}^r[\mathcal{D}_X^{op}]),$$
is an isomorphism in $Ho(\DG(\mathcal{A}^{\ell}))$.
Therefore, the object in homotopy category of dg-$\mathcal{A}[\D]$-modules we want to consider is
$$Cone(\omega_X^{\vee}\otimes \mathbb{D}_{\mathcal{A}}^{ver}(\mathbb{L}_{\mathcal{A}})\rightarrow \mathcal{A}),$$
obtained from the equivalence (\ref{eqn: Homotopy Insertion}).

\subsubsection{The $\EL$ $\D_X$-Space}
The quotient $\mathcal{D}_X$-algebra of on-shell functions is
\begin{equation}
    \label{eqn:On shell function algebra}
    \mathcal{A}_{EL}:=\mathcal{A}/\mathcal{I}_{EL},
\end{equation}
where for $F\in\mathcal{A}$, its image under the quotient $\mathcal{D}$-algebra map $q:\mathcal{A}\rightarrow \mathcal{A}_{EL}$ is denoted by $[F].$ Thus $P\bullet [F]:=\big[P\bullet F\big],$ where $\mathcal{A}_{\mathrm{EL}}\otimes \mathcal{A}_{\mathrm{EL}}\rightarrow \mathcal{A}_{\mathrm{EL}}$ is the obvious one $[F][G]:=[FG].$
Therefore, $\mathcal{A}_{EL}$ is a commutative $\mathcal{D}_X$-$\mathcal{A}$-algebra, via $q.$

The Euler-Lagrange critical $\D_X$-space is the (étale local) algebraic $\D_X$-space 
\begin{equation}
    \label{eqn: Euler Lagrange D Space}
\underline{Sol}_{EL}:=\underline{\mathrm{Spec}}_{\mathcal{D}_X}(\mathcal{A}_{\mathrm{EL}}):\mathrm{CAlg}_{\mathcal{A}/}(\mathcal{D}_X)\rightarrow \mathrm{Sets}.
\end{equation}
The $\D$-space (\ref{eqn: Euler Lagrange D Space}) is the one naturally arising (see Sect. \ref{sec: Geometry of D-Prestacks}) from the sequence
\begin{equation}
\label{eqn: Short exact Euler-Lagrange sequence}
0\rightarrow \mathcal{I}_{EL}\rightarrow \mathcal{A}\rightarrow \mathcal{A}_{EL}\rightarrow 0,
\end{equation}

\begin{proposition}
\label{Euler-Lagrange SES is D submanifold proposition}
The short exact sequence \emph{(\ref{eqn: Short exact Euler-Lagrange sequence})} defines a $\D$-algebraic PDE in the sense of \emph{(\ref{eqn: SES})}.
\end{proposition}
\begin{example}
\label{example: EL}
\normalfont Consider the $\D_X$ algebra associated to  $\pi:\mathbb{R}^{n+m}\rightarrow \mathbb{R}^n,$ with $\mathcal{S}=[L\omega]$ for $L=L(x,u^{\alpha},u_{\sigma}^{\alpha})\in \mathcal{A}=\mathrm{Jet}^{\infty}(\mathcal{O}_E)$ with volume form $\omega\in\omega_X.$ Then, the vector field in $\Theta_{\mathcal{A}}$ associated to the vertical derivation $\frac{\partial}{\partial u^{\alpha}}$ is sent by (\ref{eqn:Insertion map}) to the generator $i_{d\mathcal{S}}(\partial_{u^{\alpha}})=\mathcal{E}_{\alpha}(L).$ This is called the \emph{Euler-Lagrange function} on $\mathcal{A}$, and is represented by the expression
\begin{equation}
\label{EulerLagrangeFunction}
\mathcal{E}_{\alpha}(L):=\sum_{\sigma}(-1)^{|\sigma|}D_{\sigma}\big(\frac{\partial L}{\partial u_{\sigma}^{\alpha}}\big), \hspace{3mm}\alpha =1,\ldots,m.
\end{equation}
The image of the $\mathcal{D}_X$-algebra surjection in (\ref{eqn: Short exact Euler-Lagrange sequence}) is the quotient $\mathcal{A}_{EL}:=\mathcal{A}/\big<\mathcal{E}_{\alpha}(L):\alpha=1,\ldots,m\big>,$ defined by the relations $\mathcal{E}_{\alpha}(L)=0.$
\end{example}
Via an appropriate cofibrant resolution (c.f. Subsect.\ref{sssec: Variational tricomplexes}), one has the derived \EL-$\D$-stack as in (\ref{eqn: Local solution D space}), denoted by 
$\mathbb{R}Sol_{EL}.$

\begin{proposition}
\label{prop: EL shifted structure}
Suppose that the Euler-Lagrange critical $\mathcal{D}$-algebra $\mathcal{A}_{\mathrm{EL}}$ is $\mathcal{D}$-smooth and the $\mathcal{D}$-space of solutions identifies with a critical locus of some action functional. Then there exists a canonical cohomology class $\omega \in h^{-1}(\Omega_{\mathcal{A}_{\mathrm{EL}}}^2)$ and $\mathbb{R}Sol_{EL}$ carries a canonical $(-1)$-shifted $(d_X,2)$-form.
Moreover, there is a canonical morphism 
    $$\mathcal{V}ar_{cl}^{(d_X-1,2)}(\mathbb{R}Sol_{EL},\mathbf{0})\rightarrow \mathcal{A}_{cl}^2\big(\mathbb{R}Sect_0(X,E); \mathbf{-1}),$$
    of simplicial sets.
\end{proposition}
\begin{proof}[Sketch of the proof.]
Consider the insertion map (\ref{eqn:Insertion map}) and its restriction to the critical algebra $\mathcal{A}_{\mathrm{EL}}$. This restriction coincides with the zero map $\Theta_{\mathcal{A}_{\mathrm{EL}}}^{\ell}\rightarrow \mathcal{A}_{\mathrm{EL}}$ and consequently, since $\mathcal{A}_{\mathrm{EL}}$ is smooth by hypothesis, this map identifies with the cohomology class of $d\overline{S}$ in $h\big(\Omega_{\mathcal{A}_{\mathrm{EL}}}^1\big)$ where we view $\overline{S}$ as living in the image of $h$ in the $\mathcal{D}$-algebra $\mathcal{A}_{\mathrm{EL}}.$ Since this class is zero in $h(\mathcal{A}_{\mathrm{EL}})=H^0\big(\mathsf{DR}(\mathcal{A}_{\mathrm{EL}})\big)$, there exists a form $\theta\in\Omega_X^{n-1}\otimes_{\mathcal{O}_X}\Omega_{\mathcal{A}_{\mathrm{EL}}}^1$. 
We let $\omega:=d \theta\in\Omega_X^{n-1}\otimes \Omega_{\mathcal{A}_{\mathrm{EL}}}^2$ and note that this form is closed for the differential induced by the de Rham differential $d_X$ on $X$, since they are anti-commuting:
$$d_X\circ d\pm d\circ d_X=0.$$ Therefore, it defines a natural class in $h^{-1}\big(\Omega_{\mathcal{A}_{\mathrm{EL}}}^2\big).$
By Proposition \ref{prop: Jets is pre-symplectic} let $j:Sol_{\mathcal{D}_X}^{EL}\hookrightarrow \mathbb{R}Sol_{EL},$ denote the the natural derived enhancement. There is a natural morphism 
$\mathbb{T}_{\mathbb{R}Sol_{EL}/X}\wedge^*\mathbb{T}_{\mathbb{R}Sol_{EL/X}}\rightarrow \mathcal{O}_{\mathbb{R}Sol_{EL}}[-1],$ induced by a $D_{Tot}$-cocycle $\gamma$ in the complex 
    $\mathcal{V}ar^{(d_X-1,2),cl}(\mathbb{R}Sol_{EL}).$
    Such an element is determined by the $(-1)$-shifted $(d_X,2)$-form $\omega_{-1}$ as well as the pre-symplectic form $\omega_{\JetX}$ on $\JetX.$ Indeed, one may verify it satisfies the necessary conditions Proposition \ref{prop: Closed degree N (p,q) form}. In particular, the pre-symplectic form is $\delta$-closed.
\end{proof}
\vspace{1mm}

Following Theorem \ref{theorem: Theorem B Main claim} and the discussion of Corollary \ref{Thm B corollary}, we give a helpful pictorial summary of the spaces involved in Proposition \ref{prop: EL shifted structure} is given by:
\[
\begin{tikzcd}
  Sol_{\mathcal{D}}^{EL}\arrow[d,"\simeq"] \arrow[r,"j_{\mathcal{D}}"] & \mathbb{R}Sol_{\mathcal{D}}^{EL}\arrow[d,"Thm.(\ref{theorem: Theorem B Main claim})"]
  \\
 Sect(X,E)\supseteq  \begin{cases}\equalto{Sol^{non-loc}(\mathcal{I}_{EL})}{Crit_X(S)}
  \end{cases}\arrow[r,"j"] & \begin{rcases}\equalto{\mathbb{R}Sol^{non-loc}(\mathcal{I}_{EL})}{\mathbb{R}Crit_X(S)}\end{rcases}\subseteq \mathbb{R}Sect(X,E).
\end{tikzcd}
\]

The identification of the classical critical locus with the non-local solution space is given for $S$ a functional with associated Lagrangian $L$ (c.f. \ref{example: EL}) via
$$\{x\in H|d_xS=0\}\simeq \{x\in H|\mathcal{E}_{\alpha}(L)\circ j_{\infty}(x)=0\}.$$
\vspace{1mm}

\noindent\textit{\textbf{Symmetries in $\D$-Geometry.}}
There is another naturally arising short-exact sequence of $\mathcal{A}[\mathcal{D}_X]$-modules
\begin{equation}
    \label{eqn:Short exact Noether symmetry sequence}
    0\rightarrow \mathcal{N}_{\mathcal{S}}\rightarrow \Theta_{\mathcal{A}}\rightarrow \mathcal{I}_{EL}\rightarrow 0.
\end{equation}

Recall that by Proposition \ref{prop: Integration pairing}, elements of $h(\Theta_{\mathcal{A}})$ are interpreted as vector fields on the space of solutions. Denote the $\mathcal{A}[\mathcal{D}_X]$-submodule of Noether symmetries by
$\mathcal{S}\mathrm{ym}_{\mathrm{Noeth}}(\mathcal{S})$ and note there is an isomorphism of Lie algebras $\mathcal{S}\mathrm{ym}_{\mathrm{Noeth}}(\mathcal{S})\xrightarrow{\simeq} h\big(\Theta_{\mathcal{A}_{\mathrm{EL}}}\big)$ coming from a natural operation in $P_{\mathcal{A},1}^*\big(\{\Theta_{\mathcal{A}}\},\Theta_{\mathcal{A}/\mathcal{I}_{\mathrm{EL}}}\big)$ c.f. (\ref{eqn: LocalOps}) and (\ref{eqn: Holocals}).

In the $\D$-geometric setting other objects of physicial significance can be identified:
\begin{itemize}
    \item \emph{pre gauge symmetries} are elements $g\in \mathcal{P
}_{\mathcal{A},[1]}^*(\{\mathcal{Q}\},\Theta_{\mathcal{A}}),$ for $\mathcal{Q}\in \mathrm{Mod}(\mathcal{A}^{\ell}[\mathcal{D}_X]);$
\item A pre gauge-symmetry $g$ is a \emph{gauge symmetry} if the associated linear map $G:=h(g):h(\mathcal{Q})\rightarrow h(\Theta_{\mathcal{A}})$ has the property that $g(\epsilon)$ is a Noether symmetry for all $\epsilon.$ For $\mathcal{Q}$ arbitrary, one says that $G$ is a Noether symmetry depending on arbitrary parameters $\epsilon$.

\item A \emph{$\D$-Noether symmetry} is a gauge symmetry $g$ depending on parameters coming from the underlying $\D$-algebra $\mathcal{A}$, rather than arbitrary ones i.e. they are maps $g:\mathcal{A}\rightarrow h(\Theta_{\mathcal{A}}).$ 
\end{itemize} 
In particular, a pre gauge symmetry is a gauge symmetry if and only if the linear map $G:=h(g):h(\mathcal{Q})\rightarrow Sym_{Noeth}(\mathcal{S})$ is surjective. Locally, they are given by
$$G(\epsilon)=\sum_{\sigma}\sum_{\alpha=1}^mR^{\alpha,\sigma}(\mathcal{D}_X\bullet \epsilon )_{\sigma}\partial_{u^{\alpha}}.$$

\subsection{The BV-$\D_X$-space}
\label{ssec: BV-D-Space}
In a theory that possess gauge symmetries, the true physical observables (before gauge-fixing) will be the functionals not on (\ref{eqn: Euler Lagrange D Space}) or its derived analog (c.f. \ref{eqn: Derived Local Solutions}), but a related \BV-space. In particular, the \BV-functional will be a class $\mathcal{S}_{BV}\in h(\mathcal{A}_{BV}).$ 

In the $\D$-geometric setting, various regularity conditions introduced in \cite{Pau02,Pau01} are given -- what are called regular generating space of Noether symmetries $\mathfrak{g}^{\bullet},$ whose components in \emph{loc.cit.} are supposed to be finitely generated projective $\mathcal{A}^{\ell}[\mathcal{D}_X]$-modules.
Then, one has a completed bi-graded symmetric \BV-$\D$-algebra (c.f. Subsect. \ref{ssec: HoPDE and Quantization}) defined by 
$$\widehat{Sym}_{\mathcal{A}^{\ell}}\big(\underbrace{\mathfrak{g}[2]\oplus \mathbb{T}_{Sol}\oplus \mathcal{O}_{Sol}\oplus \mathbb{D}^{ver}\mathfrak{g}[-1]}_{\mathcal{V}_{\mathrm{BV}}}\big).$$

From the perspective of homological algebra and sheaf theory we may relax the projectivity requirements, but to ensure dualizability, we will have to replace our objects by homologically finite objects i.e. cofibrant $\mathcal{O}_{Sol}[\D]$-modules. 
\begin{definition}
\label{BVbundle}
\normalfont 
A \emph{Batalin-Vilkovisky sheaf} is a (possibly multi-graded) sheaf $\mathcal{E}_{BV}$ together with an isomorphism on global sections of (graded) 
$\mathcal{O}_{Sol}[\D_X]$-modules
$\mathcal{A}[\mathcal{D}_X]\otimes_{\mathcal{O}_E} \mathcal{E}_{\mathrm{BV}}^{*}\simeq\mathcal{V}_{BV},$ such that the components $\mathcal{V}_{\mathrm{BV}}^i$ are arbitrary $\mathcal{O}_{Sol}[\D]$-modules. If all components are finitely generated and projective $\mathcal{O}_{Sol}[\D_X]$-modules (in particular, vector $\D$-bundles over $Sol_{\mathcal{D}}$ 
c.f. \ref{sssec: Horizontal Jets}), we say that it is a \emph{BV-bundle}.
\end{definition}
Considering a \BV-sheaf, proceed by considering the space of symmetries as an $\mathcal{O}_{Sol}[\D]$-modules $\mathcal{C}$ given by the cone above. Considering a cofibrant replacement of this not necessarily projective module\footnote{In physics, one says there are reducible symmetries.}, in $\DG_{\D}(\mathcal{O}_{Sol})$, say $\mathcal{R}_{\mathcal{C}}^{\bullet},$ we may consider the complex
$$E_{Sol}^{\bullet}:=\big[\mathcal{R}_{\mathcal{C}}^{\bullet}\rightarrow \mathbb{T}_{Sol}[-1]\rightarrow \mathcal{O}_{Sol}\big],$$
with $\mathbb{T}_{Sol}$ in degree $-1$ and $\mathcal{O}_{Sol}$ in degree $0.$ Considering the symmetric algebra in $\mathcal{O}_{Sol}$-dg modules (\ref{ssec: De Rham Algebras}), then one has a quasi-isomorphism:
$$\mathbb{R}\mathcal{P}_{\mathcal{O}_{Sol},[1]}^*\big(\mathcal{C};\mathcal{O}_{Sol}\big)\xrightarrow{\simeq} \mathcal{P}_{\mathcal{O}_{Sol},[1]}^*\big(\mathcal{R}_{\mathcal{C}}^{\bullet};\mathcal{O}_{Sol}\big),$$
of complexes of dg-operations and therefore an induced evaluation pairing 
$$\mathcal{R}_{\mathcal{C}}^{\bullet}\boxtimes^{\mathbb{L}}\mathbb{R}\mathcal{P}_{\mathcal{O}_{Sol},[1]}^*(\mathcal{C};\mathcal{O}_{Sol})\rightarrow \mathbb{R}\Delta_*\mathcal{O}_{Sol},$$
since this is a finitely presented $\mathcal{O}_{Sol}\otimes\D_X$-resolution and the following is therefore immediate.

\begin{proposition}
\label{BV bundle and BV D algebra proposition}
Suppose that $\mathcal{E}_{BV}$ is a \BV-sheaf. Then 
$$\mathcal{A}_{\mathrm{BV}}\simeq \mathcal{S}ym_{\mathcal{O}_{Sol}}\big(\mathbb{R}\mathcal{P}_{\mathcal{O}_{Sol,[1]}}^*(\mathcal{C};\mathcal{O}_{Sol})\big)\otimes_{\mathcal{O}_X}Sym_{\mathcal{O}_{Sol}-dg}^*(E^{\bullet}),$$
is a cofibrant bi-differential $\mathcal{O}_{Sol}[\D]$-algebra. If $\mathcal{E}_{BV}$ is moreover the sheaf of sections of a \BV-bundle $E_{BV}\rightarrow E,$ there are natural isomorphisms of bi-graded $\mathcal{D}_X$-algebras
$\mathrm{Jet}^{\infty}\big(\mathcal{O}_{E_{BV}}\big)\cong \mathcal{A}_{BV}=Sym_{\mathcal{A}}\big(\mathcal{V}_{BV}\big).$
\end{proposition}
It is also convenient in practice to introduce for a given \BV-bundle the notion of a \emph{BV-decomposition}-- a chosen splitting of the total space $$E_{\mathrm{BV}}=E_{\mathfrak{t}(1)}\times_X E_{\mathfrak{t}(2)}\times_X\ldots\times_X E_{\mathfrak{t}(n)},$$
by (graded or super) field bundles $\pi_{\mathfrak{t}(i)}:E_{\mathfrak{t}(i)}\rightarrow X$ of a certain type $\mathfrak{t}(i).$

\begin{remark}
    \BV-decompositions exists but there is no preferred way to choose them. This is the content of the Batchelor theorem for graded bundles which states that every graded bundle is non-canonically isomorphic to a split graded bundle and this split form is canonical, however the isomorphism implementing the splitting is non-canonical.
\end{remark}

By Proposition \ref{BV bundle and BV D algebra proposition} write $\mathbf{Spec}_{\mathcal{D}_X}(\mathcal{A}_{BV})$ for the associated algebraic $\D_X$-space and call this the BV-$\mathcal{D}_X$-theory. It étale local on $X$ and identifies with a pre-sheaf of fields $\underline{Sect}(X,E_{BV}).$ The image of the natural map
\begin{equation}
    \label{eqn:Local functionals map}
\mathcal{F}:h\big(\mathcal{A}_{BV}\big)= h\big(\mathrm{Jet}^{\infty}(\mathcal{O}_{E_{BV}})\big)\rightarrow \mathrm{Maps}\big(\underline{Sect}(X,E_{BV}),\underline{\mathbf{k}}\big),
\end{equation}
are general functionals of field and antifields for a BV-bundle $E_{\mathrm{BV}}$.
There is an evident pairing (c.f. Prop \ref{prop: Integration pairing})
$$\int:h^*\big(\mathsf{DR}^{var}(\mathcal{A}_{\mathrm{BV}})\big)\times H_{*,c}(X)\rightarrow \mathrm{Maps}\big(\underline{Sect}(X,E_{BV}),\underline{k}\big).$$
To properly formalize interactions $-$ especially in the terminology of a Feynman diagram with $V$ vertices $-$ one should extend the BV-$\mathcal{D}$-dg algebra by scalars to a coupling constants algebra i.e. 
$$\mathcal{A}_{BV}^{\lambda}:=\mathbb{C}[\![\lambda_{v_1},\ldots,\lambda_{V}]\!]\otimes\mathcal{A}_{BV}.$$ 
For the sake of exposition, just consider the extension by the algebra generated by formal variables $\hbar,g,\mathfrak{j}$ playing the role of Planck's constant, a coupling constant and a source field, respectively.  
     The left dg-$\mathcal{D}_X$-algebra and its right analog are
$$\mathcal{A}_{\mathrm{BV}}[\![\hbar,g,\mathfrak{j}]\!]:=\prod_{\alpha_1,\alpha_2,\alpha_3\in\mathbb{N}}\mathcal{A}_{\mathrm{BV}}\big<\hbar^{\alpha_1},g^{\alpha_2},\mathfrak{j}^{\alpha_3}\big>,\hspace{2mm}\text{ and }\hspace{1mm} \mathcal{A}_{\mathrm{BV}}^r[\![\hbar,g,\mathfrak{j}]\!]\simeq \Omega^{n,0}\big(\pi_{\mathrm{BV}}\big)[\![\hbar,g,\mathfrak{j}]\!],$$ 
so that elements $\mathcal{L}_{int}$ of the latter represent local interaction Lagrangian densities.
It follows that interaction functionals $\mathcal{S}_{int}$ are defined by the image of the map\footnote{This is the transgression of the compactly supported horizontal differential $n$-form $\mathcal{L}_{\mathrm{int}}\in \overline{\Omega}^n(\pi_{\mathrm{BV}})[\![\hbar,g,\mathfrak{j}]\!]$ to the mapping space $\mathrm{Maps}_{\mathrm{PShf}}\big(\underline{\Gamma}(X,E_{\mathrm{BV}}),\underline{\mathbf{k}}\big)[\![\hbar,g,\mathfrak{j}]\!]$ in the pre-sheaf category.} induced from (\ref{eqn:Local functionals map}),
$$h\big(\mathcal{A}_{\mathrm{BV}}\big)[\![\hbar,g,\mathfrak{j}]\!]\rightarrow \mathrm{Maps}\big(\underline{Sect}(X,E_{BV}),\underline{\mathbf{k}}\big)[\![\hbar,g,\mathfrak{j}]\!].$$
\subsubsection{Factorization approach to $\D$-geometric \BV}
Geometrically, the space $$\EQ_{BV}=\mathbb{R}\underline{Spec}_{\mathcal{D}}(\mathcal{A}_{BV}),$$ is understood as the quotient $$\mathbb{R}Spec(\mathcal{A}/\mathcal{I}_{\mathcal{S}})/ (\mathcal{N}_{\mathcal{S}}^r/\mathcal{I}_{\mathcal{S}}^r),$$ by the Lie algebroid of on-shell gauge symmetries. The $\D$-geometric algebroid structure is encoded by the bracket obtained via the restriction of the local bracket on vector fields, which factors as
$$[-,-]^*:ker(i_{d\mathcal{S}})^r\boxtimes ker(i_{d\mathcal{S}})^r\rightarrow \Delta_*ker(i_{d\mathcal{S}})^r\hookrightarrow \Delta_*\Theta_{\mathcal{A}}.$$
The algebroid action then reads $ker(i_{d\mathcal{S}})^r\boxtimes \mathcal{A}_{EL}^r\rightarrow \Delta_*\mathcal{A}_{EL}^r.$
\vspace{1.5mm}

Suppose that $\mathcal{E}_{\mathrm{BV}}$ is a \BV sheaf of $\mathcal{O}_X$-modules with the structure of a Lie algebroid over $X$.

Then $\mathcal{V}_{\mathrm{BV}}$ is the corresponding $\mathrm{Lie}^*$-algebroid i.e. 
$\mathcal{V}_{\mathrm{BV}}\simeq \mathrm{ind}_{\mathcal{A}[\mathcal{D}_X]}\big(\mathcal{E}_{\mathrm{BV}}^{\vee}\big).$
It follows that $\mathcal{V}_{\mathrm{BV}}^{Ran}:=\Delta_*\mathcal{V}_{\mathrm{BV}},$ with $\Delta:X\hookrightarrow Ran_X$ determines a Lie algebra object in $\mathsf{Mod}(\D_{Ran_X})$ for the local tensor product $\otimes^*.$ It is also a factorization $\mathcal{O}_{\mathcal{A}_{EL}}[\D_{Ran_X}]$-module and if $\mathcal{E}_{BV}$ is moreover a \BV-bundle, it is a factorization $\D$-bundle over canonical the factorization $\D$-space obtained from $Spec_{\mathcal{D}}(\mathcal{A}_{EL}).$ 
Notice that $\mathcal{V}_{\mathrm{BV}}^{Ran}=\{\mathcal{V}_{\mathrm{BV}}^{(I)}\}$ is such that $\mathcal{V}_{\mathrm{BV}}^{(I)}:=\Delta_*^{(I)}\mathrm{ind}_{\mathcal{A}[\mathcal{D}_X]}(\mathcal{E}_{\mathrm{BV}}^{\vee}).$  Consider the corresponding co-commutative coalgebra object
\begin{equation}
\label{eqn: BV-CE}
\mathrm{CE}_{\otimes^*}^{\bullet}\big(\mathcal{V}_{\mathrm{BV}}\big):=\mathrm{Sym}^{\bullet,\otimes^*}\big(\mathbb{D}^{ver}(\mathcal{V}_{\mathrm{BV}})[-1]\big),
\end{equation}
via derived local duality (\ref{eqn: Derived Inner Verdier Dual}). Since the underlying $\D$-module object is supported on the main diagonal, 
(\ref{eqn: BV-CE}) is specified by a collection
$$\mathrm{CE}_{\otimes^{\star}}^{\bullet}(\mathcal{V}_{\mathrm{BV}})^{Ran}:=\big\{\mathrm{CE}_{\otimes^{\star}}(\mathcal{V}_{\mathrm{BV}})^{(I)}\},\hspace{1mm}\text{ where }\hspace{1mm}\mathrm{CE}_{\otimes^{\star}}^{\bullet}(\mathcal{V}_{\mathrm{BV}})^{(I)}\simeq\big(\bigoplus_{\alpha:I\rightarrow J}\bigboxtimes_{j\in J}\mathcal{V}_{\mathrm{BV}}^{(I_j)}[| J|]\big)_{\Sigma_{|I|}}.$$
One may then observe:
$$\Delta^{*}\big(CE_{\otimes^{\star}}^{\bullet}(\Delta_*\mathcal{V}_{\mathrm{BV}})\big)^{(I)}\simeq \bigoplus_{k\in \mathbb{Z}_{>0}}\Delta^{*}Sym^{k,\otimes^{\star}}\big(\Delta_*\mathcal{V}_{\mathrm{BV}}\big)\simeq \bigoplus_k Sym^{k,!}(\mathcal{V}_{\mathrm{BV}}),$$
via the $\otimes_{\mathcal{O}_X}$-symmetric algebra.

A homotopy Lie-algebroid object $\mathcal{L}^{\bullet}$ in $(\infty,1)$-category of $\mathcal{A}^{\ell}[\mathcal{D}_X]$-modules determines an associated formal derived algebraic $\mathcal{D}_X$-stack \cite{KSYI} via
$$\EQ_{\mathcal{L}}:=\mathbf{Spec}_{\mathcal{D}_X}\big(\mathrm{CE}_{\mathcal{D}}(\mathcal{L})\big)=\mathbf{Spec}_{\mathcal{D}_X}\big(\mathrm{Sym}_{\mathcal{A}}^*(\mathcal{L}^{\circ}[-1])\big).$$

\begin{proposition}
\label{Derived Quotient as BV Theorem}
$\mathcal{V}_{\mathrm{BV}}^{(1)}$ is a $\mathrm{Lie}^*$-algebroid over $\mathrm{Spec}_{\mathcal{D}_X}(\mathcal{A}_{\mathrm{EL}}),$ such that there is a natural map
$$\mathcal{\mathbf{Spec}}_{\mathcal{D}_X}\big(\Delta^{!}\mathrm{CE}_{\otimes^*}^{\bullet}(\mathcal{V}_{\mathrm{BV}})\big)\rightarrow \big[\mathbf{Spec}_{\mathcal{D}_X}(\mathcal{A}_{\mathrm{EL}}/\mathcal{V}_{\mathrm{BV}})\big],$$ yielding by abuse of notation, an equivalence
$$\Gamma\big(\big[\mathbf{Spec}_{\mathcal{D}_X}/\mathcal{V}_{\mathrm{BV}}\big],\mathcal{O}_{[\mathbf{Spec}_{\mathcal{D}}(\mathcal{A}_{\mathrm{EL}})/\mathcal{V}_{\mathrm{BV}}]}\big)\xrightarrow{\simeq}\mathrm{CE}_{\otimes^*}^{\bullet}(\mathcal{V}_{\mathrm{BV}}).$$

\end{proposition}
With additional hypothesis, it is possible to further express the factorization homology as the Lie algebroid hypercohomology of a BV-sheaf. For example, if it is an elliptic local Lie algebra \cite{CosGwi02,CosGwi01} with an invariant pairing, then putting $\mathcal{A}_{\mathrm{BV}}^{\bullet}:=\Delta^{!}\mathrm{CE}_{\otimes^*}^{\bullet}\big(\mathcal{V}_{\mathrm{BV}}\big)\in \mathsf{CAlg}_X(\D_X),$ it inherits a shifted local Poisson structure and there is an equivalence of monoids in commutative $k$-algebras,
$$\int_X\mathcal{A}_{\mathrm{BV}}^{\bullet}\xrightarrow{\simeq} \mathrm{CE}\big(\mathbb{R}\Gamma\big(X,\mathcal{E}_{\mathrm{BV}})\big).$$

\subsection{Intrinsic symmetries via $\D$-geometry.}
\label{sssec: Intrinsic formulation via D-geometry}
To conclude, we sketch a possible candidate space of $\BV$ gauge symmetries using global derived $\D$-geometry. A more complete treatment via the de Rham space approach based on (\ref{eqn: MainEquiv}) will be given elsewhere.

Consider a differentially generated $\mathcal{D}_X$-ideal sheaf
$\mathcal{I}=Span\big\{\mathsf{F}_1,\ldots,\mathsf{F}_N\big\}$
where, we suppose that $\mathcal{A}=Sym_{\mathcal{O}_X}(\mathcal{D}_X^n),$ and consider the defining sequences for its quotient $\mathcal{D}$-algebra of functions $\mathcal{B}$,
\begin{equation}
    \label{eqn: SES Example}
span\{\mathsf{F}_1,\ldots,\mathsf{F}_N\}\rightarrow Sym_{\mathcal{O}_X}(\mathcal{D}_X^n)\rightarrow \mathcal{B},
\end{equation}
whose differential polynomials $\mathsf{F}_i(x,\partial_x),i=1,\ldots,N,$ are of fixed $ord(\mathsf{F}_i):=d_i$ for $i=1,\ldots,N.$
Sequences (\ref{eqn: SES Example}) obviously define a $\mathcal{D}$-space $$Sol_{\mathcal{D}_X}(\mathcal{I})\simeq \{(f_1,\ldots,f_n)|\mathsf{F}_i(f_1,\ldots,f_n)=0\},$$ for some $f_i\in\mathcal{O}_X,$ and we consider the affine $\mathcal{D}$-geometry of the inclusion of this $\D_X$-space into a derived $\D_X$-space as a homotopical thickening
$$j:Sol_{\mathcal{D}_X}\hookrightarrow \mathbb{R}Sol_{\mathcal{D}_X}.$$ 
Consider the naturally arising (homotopy fiber) sequence of tangent $\D$-complexes:
\begin{equation}
    \label{eqn: Tangent D-Complex Fiber Sequence for RCRIT}
\mathbb{T}\big(Sol_{\mathcal{D}_X}/\mathbb{R}Sol_{\mathcal{D}_X}\big)\rightarrow \mathbb{T}\big(Sol_{\mathcal{D}_X}\big)\rightarrow j^*\mathbb{T}\big(\mathbb{R}Sol_{\mathcal{D}_X}\big).
\end{equation}
The intrinsic definition of $\mathcal{D}$-\emph{gauge symmetries} might be posed in terms of the relative tangent complex $\mathbb{T}\big(Sol_{\mathcal{D}_X}/\mathbb{R}Sol_{\mathcal{D}_X}\big).$ Forgetting the $\mathcal{D}$-action this object has the structure of a Lie-algebroid object in quasi-coherent sheaves on derived (Artin) stacks i.e. it is an $L_{\infty}$-algebroid which we view as a Lie algebra object in the $\infty$-category of homotopically finitely-presented $\mathcal{D}_X$-modules over $\mathcal{A}$.
The point now is to construct a derived reduction of $\mathbb{R}Sol_{\mathcal{D}}$ by a suitable Hamiltonian action.\footnote{Supposing we have a retraction of $j$ such that 
$\mathbb{T}\big(\mathbb{R}Sol_{\mathcal{D}_X}/Sol_{\mathcal{D}_X}\big)\rightarrow \mathbb{T}\big(\mathbb{R}Sol_{\mathcal{D}_X}\big)$ induces a Hamiltonian flow, the desired reduction of $\mathbb{R}Sol_{\mathcal{D}_X}$ by this Hamiltonian action is just the so-called homotopical $\mathcal{D}$-Poisson reduction studied in \cite{Pau01}, essentially corresponding to the explicit presentation recalled in Sect.\ref{ssec: Variational Calculus and BV Formalism in D Geometry}.}

Assume that: $j$ is an $n$-shifted coisotropic morphism in a suitable $\D$-sense, which may be obtained for instance, by transporting the notion of $n$-shifted coisotropic morphism of derived stacks via the equivalence (\ref{eqn: MainEquiv}), and that $\mathbb{R}Sol_{\mathcal{D}}$ is reasonably $(-1)$-shifted as an object of $\mathsf{DStk}_{/X_{dR}},$ under (\ref{eqn: MainEquiv}).

Under these assumptions, one has a map of fiber sequences of dg-$\mathcal{A}[\mathcal{D}_X]$ modules

\begin{equation}
    \label{eqn: GaugeSymmetryDiagram}
\begin{tikzcd}
   \mathbb{L}_{Sol_{\mathcal{D}}/\mathbb{R}Sol}[-1]\arrow[d] \arrow[r] & j^*\mathbb{L}_{\mathbb{R}Sol_{\mathcal{D}}}\arrow[d] \arrow[r] & \mathbb{L}_{Sol_{\mathcal{D}}}\arrow[d]
   \\
   \mathbb{T}_{Sol_{\mathcal{D}}}[-n]\arrow[r] & j^*\mathbb{T}_{\mathbb{R}Sol_{\mathcal{D}}}^{\ell}[-n]\arrow[r]& \underbrace{\mathbb{T}_{Sol_{\mathcal{D}}/\mathbb{R}Sol}^{\ell}[1-n]}_{\textcolor{blue}{\mathcal{D}-\text{Geometric gauge symmetries}}}
\end{tikzcd}
\end{equation}

Now consider the formal completion of the solutions inside its homotopical enhancement, denoted $\mathbb{R}\mathcal{C}_{\mathcal{D}_X}.$ There is a natural morphism $Sol_{\mathcal{D}_X}\rightarrow \mathbb{R}\mathcal{C}_{\mathcal{D}_X},$ and if it satisfies base-change along reduced affine $\mathcal{D}$-schemes mapping to $\mathbb{R}\mathcal{C}_{\mathcal{D}}$ in classical $\D$-prestacks, that is, if
$$\mathbb{R}Sol_{\mathcal{D}_X}\rightarrow \mathbb{R}\mathcal{C}_{\mathcal{D}_X},$$
is an equivalence of reduced objects i.e. on reduced classical $\mathcal{D}_X$-prestacks, the relative tangent $\mathcal{D}$-space $\mathbb{T}_{Sol_{\mathcal{D}_X/\mathbb{R}\mathcal{C}_{\mathcal{D}}}}$ is a dg-Lie algebroid in $\mathcal{D}$-modules satisfying 
$$\mathbb{T}(Sol_{\mathcal{D}_X}/\mathbb{R}\mathcal{C}_{\mathcal{D}})\simeq \mathbb{L}(Sol_{\mathcal{D}_X}/\mathbb{R}Sol_{\mathcal{D}_X})[-2].$$ 

Its dg-$\mathcal{O}_{Sol_{\mathcal{D}_X}}$-algebra in $\mathcal{D}_X$-modules is
$Sym_{\mathcal{O}_{Sol_{\mathcal{D}}}}\big(\mathbb{D}^{loc}\mathbb{T}(Sol_{\mathcal{D}}/\mathbb{R}\mathcal{C}_{\mathcal{D}})[-1]\big),$
under the appropriate derived local $\mathcal{D}$-module Verdier duality (\ref{eqn: Derived Inner Verdier Dual}). 
We thus obtain a complex of $p$-forms (décalage), as $Sym_{\mathcal{O}}^p(\mathbb{D}^{loc}(\mathbb{T}(Sol_{\mathcal{D}}/\mathbb{R}\mathcal{C}_{\mathcal{D}})[-1])[p].$
\vspace{2mm}

In the presence of a boundary the situation is more involved, and to conclude the paper we allocate the remaining space to highlight the mathematical nature of including boundary conditions and comment on some challenges. Our main result in the following section proposes a complex of $\D$-geometric gauge symmetries in the presence of a boundary in Proposition \ref{prop: Boundary Symmetries} of Subsect. \ref{sssec: Shifted Geometry of Boundary Constraints} 

\section{Application: BV-Boundary Conditions}
\label{sec: Applications 2}
Consider $\underline{Adm}^{\partial}\subset \underline{Sect}(X,E),$ the sub presheaf of admissiable boundary conditions defined as sub-functors $\underline{L}\rightarrow \underline{E},$ such that $dim(L)=k<dim(E)$ and the intersection $\underline{L}\times \underline{\partial E}\neq \emptyset.$ Any reasonable notion of a Lagrangian variational problem with free boundary conditions should include the natural transformation 
$\mathcal{S}:\underline{Adm}^{\partial}\rightarrow \underline{k}$. 

If $\partial X$ is a boundary of $X$, we will show it induces a boundary $\mathcal{D}$-scheme $\partial \mathrm{Jet}^{\infty}$ of the space of jets and for any $u\in Sect(X,E)$ the map $j_{\infty}(u)$ sends $\partial X$ into $\partial\mathrm{Jet}^{\infty}.$  So, if $\omega$ is the differential of some form vanishing on $\partial\mathrm{Jet}^{\infty}$, then $j_{\infty}^*(u)(\omega)$ will be the differential of some form vanishing on $\partial X.$ By Stokes formula, the action
determined by $\omega$, evaluated on $X$ , is zero.
In other words, the action of $\omega$ is given only by its cohomology class modulo $\partial\mathrm{Jet}^{\infty}.$ 

Therefore we consider the forms on the jet space vanishing on the boundary jet-scheme i.e. relative forms defined by the differential ideal $ker(\iota_{\partial}^*)$ in $\Omega(\mathrm{Jet}^{\infty}(E)).$ 

One may show there is a short-exact sequence of complexes for $q\geq 0,$
$$0\rightarrow \Omega^{\bullet,q}(\mathrm{Jet}^{\infty}/\partial\mathrm{Jet}^{\infty})\rightarrow\Omega^{\bullet,q}\big(\mathrm{Jet}^{\infty}(E)\big)\rightarrow \Omega^{\bullet,q}(\partial\mathrm{Jet}^{\infty})\rightarrow 0,$$
inducing a long exact cohomology sequence
$$0\rightarrow \overline{H}^0(\mathrm{Jet}^{\infty}/\partial\mathrm{Jet}^{\infty})\xrightarrow{H^0(i)}\overline{H}^0(\mathrm{Jet}^{\infty})\xrightarrow{\overline{H}^0(r)}\overline{H}^0(\partial\mathrm{Jet}^{\infty})\rightarrow$$
$$\xrightarrow{\overline{\partial}}\overline{H}^1(\mathrm{Jet}^{\infty}/\partial\mathrm{Jet}^{\infty})\xrightarrow{H^1(i)}\overline{H}^1(\mathrm{Jet}^{\infty})\xrightarrow{\overline{H}^1(r)}\overline{H}^1(\partial\mathrm{Jet}^{\infty})\rightarrow\cdots$$
With additional regularity hypothesis on the systems of differential equations (e.g. $\ell$-normality),
we obtain sequences for all $q\geq 0,$
$$0\rightarrow h(\Omega^q_{\partial\mathrm{Jet}^{\infty}/X})\rightarrow h\big(\Omega_{\mathrm{Jet}^{\infty}/\partial\mathrm{Jet}^{\infty}}^q\big)\rightarrow h(\Omega_{\mathrm{Jet}^{\infty}/X}^q)\rightarrow 0.$$
Let $\mathcal{S}$ be an element of $\overline{H}^n\big(\mathrm{Jet}^{\infty}(\pi),\pi_{\infty}^{-1}(\partial E)\big)$ and $\mathcal{S}_{\partial}\in\overline{H}^{n-1}\big(\mathrm{Jet}^{\infty}(\pi)\big).$
They correspond (c.f. Sect. \ref{ssec: Applications to Field Theory}) to natural transformations
\begin{equation}
\underline{\mathcal{S}}:\underline{\mathcal{A}\mathrm{dm}(\pi)}\ni L\rightarrow \underline{\mathcal{S}}(L)=j_{\infty}^*(L)\in H^n(L,\partial L)\cong \mathbb{R}
\end{equation}
\begin{equation}
\underline{\mathcal{S}_{\partial}}:\Sigma\in \underline{\partial\mathcal{A}\mathrm{dm}(\pi)}:=\{\partial L|L\in\mathcal{A}\text{dm}(\pi)\}\longmapsto j_{\infty}(\Sigma)^*\in H^{n-1}(\Sigma).
\end{equation}
When $\Sigma$ is orientable and oriented as an $(n-1)$-dimensional sub-variety, then $H^{n-1}(\Sigma)\cong \mathbb{R}.$
One then considers the \emph{total functional}
$$\mathcal{S}_{Tot}:L\rightarrow \mathcal{S}_{Tot}(L):=\underline{\mathcal{S}}(L)+\underline{\mathcal{S}_{\partial}}(\partial L)\in \mathbb{R}.$$
To make sense of $\mathcal{S}_{Tot}$ via an integration map (\ref{eqn: Integration pairing}) one invokes the use of flag varieties, which permit us to simultaneously integrate over an $n$-dimensional and $(n-1)$-dimensional space.

\subsection{$\mathcal{D}$-Geometric Boundary Schemes}
Suppose $X$ is a $d$-dimensional manifold with $\partial X\neq \emptyset.$ Denote the inclusion of the boundary $i:\partial X\hookrightarrow X.$

Consider the transfer $\mathcal{D}$-module,
$\mathcal{D}_{\partial X\hookrightarrow X}:=\mathcal{O}_{\partial X}\otimes_{i^{-1}\mathcal{O}_X}i^{-1}\mathcal{D}_X,$ as a sheaf of $i^{-1}\mathcal{D}_X^{op}$-modules on $\partial X.$
The usual $\mathcal{D}$-module operations are defined:
$$i_*:\mathsf{Mod}(\mathcal{D}_{\partial X}^{op})\rightarrow \mathsf{Mod}(\mathcal{D}_X^{op}),\hspace{1mm} i_*\mathcal{M}^{\bullet}:=Ri_*(\mathcal{M}^{\bullet}\otimes_{\mathcal{D}_X}^{\mathbb{L}}\mathcal{D}_{\partial X\hookrightarrow X}),$$ where the $\mathcal{D}_X$-module structure comes from the natural map $\mathcal{D}_X\rightarrow i_*i^{-1}\mathcal{D}_X.$ Similarly, put 
$$Li^*\mathcal{M}^{\bullet}:=\mathcal{D}_{\partial X\hookrightarrow X}\otimes_{i^{-1}\mathcal{D}_X}^{\mathbb{L}}i^{-1}\mathcal{M}^{\bullet}.$$
Note that we have
$i_*\mathcal{D}_{\partial X}\simeq i_*(\mathcal{D}_{\partial X\hookrightarrow X}), Li^*(\mathcal{D}_X)\simeq \mathcal{D}_{\partial X\hookrightarrow X}.$

Analogously to the usual Spencer resolution $DR_X^{\bullet}(\mathcal{D}_X)$ we consider $$DR_{\partial X}^{\bullet}(\mathcal{D}_X)\simeq i_{\partial *}\big(\Omega_{\partial X}^{\bullet}\otimes_{i_{\partial}^{-1}\mathcal{O}_X}^{\mathbb{L}}i_{\partial}^{-1}\mathcal{D}_X\big),$$ and note the existence of a morphism of sheaves on $X,$
$$\widetilde{i}_{\partial}^*:DR_X^{\bullet}(\mathcal{D}_X)\rightarrow DR_{\partial X}^{\bullet}(\mathcal{D}_X),$$
induced by the pull-back of forms along $i_{\partial}.$ 

\begin{proposition}
The sheaf of smooth densities $Dens(X/\partial X)$ on $X$ relative to the boundary $\partial X$, has a Spencer resolution as a $\mathcal{D}_X^{op}$-module and there is a quasi-isomorphism $Cone(\widetilde{i}_{\partial}^*)[d_X-1]\xrightarrow{\simeq} Dens(X/\partial X).$
\end{proposition}
A convenient way to encode constrained systems of PDEs is by exploiting the geometry of algebraic foliations $\mathcal{F}$ as they arise naturally in $\mathcal{D}$-module theory.
\vspace{1mm}

\noindent\textit{\textbf{$\mathcal{F}$-Jets.}}
Consider $(X,\mathcal{F})$ and the left $\mathcal{D}_X$-ideal $\mathcal{D}_X\bullet\mathcal{F}$ generated by $\mathcal{F}$. The space $\mathcal{D}_{\mathcal{F}}=\mathcal{D}_X/\mathcal{D}_X\bullet \mathcal{F},$ is the sheaf of differential operators in normal directions with an induced order filtration 
$$F^k(\mathcal{D}_{\mathcal{F}}):=\mathcal{D}_{X}^{\leq k}/(\mathcal{D}_{X}^{\leq k}\cap \mathcal{D}_X\bullet \mathcal{F}).$$
Passing to graded objects, $Gr(\mathcal{D}_X\bullet \mathcal{F})\subseteq Gr(\mathcal{D}_X)$ is a graded ideal.

Suppose that $C\rightarrow (X,\mathcal{F})$ is a vector bundle over $X$. For all $k\geq 0,$ there is an $\mathcal{O}_X$-module 
$\mathcal{J}_{\mathcal{F}}^k(C)$ of sections $Sect((X,\mathcal{F}),C).$ In other words, $s,s'\in Sect((X,\mathcal{F}),C)$ are $k$-equivalent over $\mathcal{F}$ if 
$$(\theta_1\circ\ldots\circ\theta_{\ell}(s))(x)=(\theta_1\circ\ldots\circ\theta_{\ell}(s')(x),\ell\leq k,\forall \theta_i\in \Theta_{\mathcal{F}}.$$
For every $k>\ell$ there are projections $\mathcal{J}_{\mathcal{F}}^k(C)\rightarrow \mathcal{J}_{\mathcal{F}}^{\ell}(C)$ and we denote the projective limit $\mathcal{J}_{\mathcal{F}}^{\infty}(C).$ There is a canonical morphism $$j_k^{\mathcal{F}}:Sect((X,\mathcal{F}),C)\rightarrow Sect((X,\mathcal{F}),\mathcal{J}_{\mathcal{F}}^{\infty}(C)\big).$$
In other words for each leaf $\mathcal{F}_x$ of $\mathcal{F}$ at $x\in X$, two sections $s,t\in Sect(X,E)$ are $\mathcal{F}$-tangent of order $\leq k$ if $s|_{\mathcal{F}_x}$ and $t|_{\mathcal{F}_x}$ are $k$-th order tangent in the usual sense i.e. have the same $k$-jet. 
\begin{proposition}
The sheaf $\mathcal{J}_{\mathcal{F}}^{\infty}(C)$ has the structure of a commutative $\mathcal{D}_{\mathcal{F}}$-algebra and defines a functor 
$\mathcal{J}_{\mathcal{F}}^{\infty}(-):\mathcal{O}_{(X,\mathcal{F})}-CAlg\rightarrow \mathcal{D}_{\mathcal{F}}-CAlg_X,$
left adjoint to the obvious forgetful functor.
\end{proposition}
The canonical $\mathcal{D}_{\mathcal{F}}$-module structure on $\mathcal{J}_{\mathcal{F}}^{\infty}(C)$ can be written explicitly via a flat connection (along $\mathcal{F}$),
$$\nabla^{\mathcal{F}}:\mathcal{F}\rightarrow \mathcal{E}nd(\mathcal{J}_{\mathcal{F}}^{\infty}(C)).$$
It is determined for each point $\theta=[s]_x^{\infty}$ and each $v\in \mathcal{F}$ by its values on a function $\mathsf{F}\in \mathcal{O}\big(\mathcal{J}_{\mathcal{F}}^{\infty}(C)\big)$ by setting
$$\nabla_v^{\mathcal{F}}(\mathsf{F})|_{\theta}=v\big(j_{\infty}(s)^*\mathsf{F}\big)(x).$$

A $\mathcal{D}_{\mathcal{F}}$-ideal in $\mathcal{J}_{\mathcal{F}}^{\infty}(C)$ generates an algebraic non-linear PDE $\mathcal{J}_{\mathcal{F}}^{\infty}(C)/\mathcal{I}_{\mathcal{F}},$ endowed with its induced order filtration. In this way, we can make sense of $k$-th order $\mathcal{F}$-equations (or $k$-th order non-linear PDEs along $\mathcal{F}$) in terms of the sub-manifolds
$\EQ\subset Jets_{\mathcal{F}}^k(C).$ A usual, we are interested in the infinite jet space $\mathrm{Jets}_{\mathcal{F}}^{\infty}(C).$ Note if $\mathcal{F}$ is just $\Theta_X,$ then $\mathrm{Jet}_{\Theta_X}^{\infty}(C)$ is the usual jet-space $\mathrm{Jet}_X^{\infty}(C).$
\vspace{1mm}

Consider the commutative $\mathcal{D}_X$-algebra $\mathcal{A}:=\mathcal{O}\big(\mathrm{Jets}_X^{\infty}(E)\big),$ and set $\partial E:=\pi^{-1}(\partial X),$ with 
\begin{equation}
    \label{eqn: Boundary D Scheme}
    p_{\partial}:\partial \mathrm{Jet}_{[\partial X]}^{\infty}(E):=p_{\infty}^{-1}(\partial X)\hookrightarrow \mathrm{Jet}^{\infty}(E).
    \end{equation}
Let $\mathcal{F}\subseteq \Theta_X$ be a foliation with leaf $F_x$ and assume $\partial X\subseteq \mathcal{F}_x.$ Put $p_{\infty}^{\mathcal{F}}:\mathrm{Jet}_{\mathcal{F}}^{\infty}(E)\rightarrow X.$

\begin{proposition}
\label{prop: Boundary F-Jets}
Let $E\rightarrow (X,\mathcal{F})$ be as above with $\partial X\neq \emptyset$ such that $\partial X\subseteq \mathcal{F}_x$ for some $x\in X.$ We have:
\begin{itemize}
    \item A canonical injection $\mathrm{Jet}^{\ell}(\ell_x^*E)\hookrightarrow \mathrm{Jet}_{\mathcal{F}}^{\infty}(E)$ seding $j_{\infty}(s)(x)$ to $j_{\infty}^{\mathcal{F}}(s')(x)$ with $s'$ the extension of $s$ from $\mathcal{F}_x$ to all $X$;

    \item A canonical surjection $\mathrm{Jet}^{\infty}(E)\rightarrow \mathrm{Jet}_{\mathcal{F}}^{\infty}(E)$ given by restriction.
\end{itemize}
\end{proposition}
Let $\overline{\mathcal{F}}$ denote the complimentary distribution to $\mathcal{F}$ and consider the \emph{infinite normal jet bundle of $\pi$ with respect to $\partial X$}, $$n_{\partial,\overline{\mathcal{F}}}:\mathrm{Jet}_{\overline{\mathcal{F}}}^{\infty}(\iota_{\partial}^*E)\rightarrow \partial X.$$
Essentially due to the fact that partial derivatives commute, there are equivalences
$$\mathrm{Jet}_{\mathcal{F}}^{\infty}\big(\mathrm{Jet}_{\overline{\mathcal{F}}}^{\infty}(E)\big)\simeq \mathrm{Jet}_{\overline{\mathcal{F}}}^{\infty}\big(\mathrm{Jet}_{\mathcal{F}}^{\infty}(E)\big),$$
further identifying with $\mathrm{Jet}^{\infty}(E).$
By Proposition \ref{prop: Boundary F-Jets} there is a map 
$\mathrm{Jet}^{\infty}(\mathrm{Jet}_{\overline{\mathcal{F}}}^{\infty}(\iota_{\partial}^*E)\big)\hookrightarrow \mathrm{Jet}_{\mathcal{F}}^{\infty}\big(\mathrm{Jet}_{\overline{\mathcal{F}}}^{\infty}(E)\big).$
\begin{proposition}
 There is an equivalence of $\mathcal{D}_{\partial X}$-schemes 
    $$\partial \mathrm{Jet}_{[\partial X]}^{\infty}(E)\simeq \mathrm{Jet}^{\infty}\big(\mathrm{Jet}_{\overline{\mathcal{F}}}^{\infty}(\iota_{\partial}^*E)\big).$$
    \end{proposition}

Let $\partial X\subset X$ and put $\mathcal{A}=\mathcal{O}(\mathrm{Jet}^{\infty}(E)).$ A \emph{boundary $\mathcal{D}$-algebra} of $\mathcal{A}$ is a $\mathcal{D}$-algebraic PDE of the form (\ref{eqn: Boundary D Scheme}).
Let us show it may be written in the form (\ref{eqn: SES}).
Suppose $dim(X)=n$ and $\partial X=\{x_n=0\}.$ In this case, 
$\theta\in Sect\big(\partial X,\mathrm{Jet}_{\overline{\mathcal{F}}}^{\infty}(\iota_{\partial}^*E)\big),$ is determined by a tuple
$$(x_1,\ldots,x_{n-1},\cdots,\frac{\partial^{|\sigma|} u_i^{\alpha}}{\partial x_{\sigma}}(x_1,\ldots,x_{n-1}),\cdots),\alpha=1,\ldots,rank(E),$$
and $\sigma$ a multi-index of length $dim(\partial X).$ 
There is an induced morphism of $\mathcal{D}$-schemes
$i_{\partial}:\partial \mathsf{X}\hookrightarrow \mathsf{X},$ and denoting by
$Ker(i_{\partial}^*)$ the ideal of functions vanishing on $\partial \mathsf{X},$ we can indeed exhibit a boundary $\mathcal{D}$-algebra as a quotient of a suitable $\mathcal{D}$-ideal, as required:
$$0\rightarrow Ker(i_{\partial}^*)\hookrightarrow\mathcal{A}\rightarrow \mathcal{B}_{\partial}\rightarrow 0.$$

\subsubsection{Shifted Geometry of Boundary Constraints.}
\label{sssec: Shifted Geometry of Boundary Constraints}
Consider $X$ compact with boundary $\partial X.$ 
Including boundary conditions must be homotopy-compatible with the $(-1)$ shifted symplectic geometry of classical field theory i.e. one must impose boundary conditions so that the obstructions given by boundary terms vanish in a suitable sense.
\vspace{1mm}

More precisely, let $\mathbb{R}Sol_{\partial X}$ be the $\D_{\partial X}$-prestack of germs of solutions on $\partial X$. 
There is a natural map 
$$res:\mathbb{R}Sol_X\rightarrow i_{\partial *}\mathbb{R}Sol_{\partial X},$$
restricting an infinite jet of a solution to its infinite normal jet on the boundary. There is a corresponding cofiber sequence
$$res^*\mathbb{L}_{\mathbb{R}Sol_X}\rightarrow \mathbb{L}_{i_{\partial *}\mathbb{R}Sol_{\partial X}}\rightarrow \mathbb{L}_{\mathbb{R}Sol_{X}/i_{\partial *}\mathbb{R}Sol_{\partial X}}.$$
From Proposition \ref{prop: DRVar is a prestack} we have a relative variational tricomplex
$$\mathcal{V}ar_{/\partial X}\simeq\bigoplus_{q\geq 0}\mathcal{V}ar_{/\partial X}^q\simeq \bigoplus_{q\geq 0}DR_{\partial X}^{\bullet}\big(Sym^q(\mathbb{L}_{\mathbb{R}Sol_X/i_{\partial *}\mathbb{R}Sol_{\partial X}}[-1])\big).$$

A \emph{$\D$-geometric local boundary condition} refers to sheaves of derived Lagrangians $f:\mathcal{L}\rightarrow \mathbb{R}Sol_{\partial X},$ for the induced symplectic structure. In other words,
$$\mathbb{T}_{\mathcal{L}}\simeq hofib\big(\mathbb{T}_{\mathcal{L}}^*\rightarrow f^*\mathbb{T}_{\mathbb{R}Sol_{\partial X}}\big).$$
The moduli space of solutions on the boundary satisfying the condition $\mathcal{L}$ is then the following homotopy fibre product in $\mathcal{D}_X$-prestacks, which as a derived Lagrangian intersection in a $0$-shifted space is itself $(-1)$-shifted symplectic:
\begin{equation}
\label{eqn: RSolBoundary}
\begin{tikzcd}
\mathbb{R}Sol_{X,\partial X}^{\mathcal{L}}\arrow[d,"\beta"]\arrow[r,"\alpha"] & \mathbb{R}Sol_X\arrow[d,"res"]
\\
i_{\partial*}\mathcal{L}\arrow[r,"i_{\partial*}f"] & i_{\partial *}\mathbb{R}Sol_{\partial X}.
\end{tikzcd}
\end{equation}

Since (\ref{eqn: RSolBoundary}) is Cartesian, there are natural equivalences 
$$\mathbb{T}_{\mathbb{R}Sol_{X,\partial X}^{\mathcal{L}}/i_{\partial *}\mathbb{R}Sol_{\partial X}}\simeq \alpha^*\mathbb{T}_{\mathbb{R}Sol_X/i_{\partial *}\mathbb{R}Sol_{\partial X}}\oplus \beta^*\mathbb{T}_{i_{\partial *}\mathcal{L}/i_{\partial *}\mathbb{R}Sol_{\partial X}}.$$

Let us conclude by giving a candidate for the gauge symmetries in this context, following the discussion of Subsect. \ref{sssec: Intrinsic formulation via D-geometry}, and diagram (\ref{eqn: GaugeSymmetryDiagram}). To this end,
suppose $\mathbb{R}Sol_{X,\partial X}^{\mathcal{L}}$ is the derived $\mathcal{D}$-enhancement of a classical prestack 
$$j_{\mathcal{L},\partial X}:Sol_{X,\partial X}^{\mathcal{L}}\rightarrow \mathbb{R}Sol_{X,\partial X}^{\mathcal{L}}.$$
By naturally extending diagram (\ref{eqn: RSolBoundary}), one obtains 
two fiber sequences of complexes of $\mathcal{O}_{Sol_{X,\partial X}^{\mathcal{L}}}[\mathcal{D}_{\partial X}]$-modules
$$\mathbb{T}_{Sol_{X,\partial X}^{\mathcal{L}}/\mathbb{R}Sol_{X}}\rightarrow \mathbb{T}_{Sol_{X,\partial X}^{\mathcal{L}}/i_{\partial *}\mathbb{R}Sol_{\partial X}}\rightarrow j_{\mathcal{L},\partial X}^*\alpha^*\mathbb{T}_{\mathbb{R}Sol_X/i_{\partial *}\mathbb{R}Sol_{\partial X}},$$
$$\mathbb{T}_{Sol_{X,\partial X}^{\mathcal{L}}/i_{\partial*}\mathcal{L}}\rightarrow \mathbb{T}_{Sol_{X,\partial X}^{\mathcal{L}}/i_{\partial *}\mathbb{R}Sol_{\partial X}}\rightarrow j_{\mathcal{L},\partial X}^*\beta^*\mathbb{T}_{i_{\partial *}\mathcal{L}/i_{\partial *}\mathbb{R}Sol_{\partial X}}.$$
\begin{proposition}
\label{prop: Boundary Symmetries}
    There is an equivalence between $\mathbb{T}_{Sol_{X,\partial X}^{\mathcal{L}}/\mathbb{R}Sol_{X,\partial X}^{\mathcal{L}}}$ and
$$hofib\big(\mathbb{T}_{Sol_{X,\partial X}^{\mathcal{L}}/i_{\partial *}\mathbb{R}Sol_{\partial X}}\rightarrow j_{\mathcal{L},\partial X}^*\alpha^*\mathbb{T}_{\mathbb{R}Sol_X/i_{\partial *}\mathbb{R}Sol_{\partial X}}\oplus j_{\mathcal{L},\partial X}^*\beta^*\mathbb{T}_{i_{\partial *}\mathcal{L}/i_{\partial *}\mathbb{R}Sol_{\partial X}}\big).$$
\end{proposition}

If $\partial X$ is of codimension $\geq 1$, there is a possibility to extend these ideas as well as define analogs of transgression maps.

Let $f:Z\rightarrow X$ be an embedding of a closed sub-variety\footnote{In more general settings, we may have that $f$ is monomorphism of derived Artin stacks.} and let $\EuScript{E}\in \PS_{/X}.$ Then consider the pull-back diagrams in $\PS$ defining $X_Z^{\wedge}$ and $\iota_Z^*(\EuScript{E}),$
\begin{equation}
\label{eqn: Formal Nbhd}
\begin{tikzcd}
    X_Z^{\wedge}\arrow[d]\arrow[r,"\iota_Z"] & X\arrow[d,"p_{dR}^X"]
    \\
    Z_{dR}\arrow[r,"f_{dR}"] & X_{dR}
\end{tikzcd}\hspace{3mm}\text{ and, }\hspace{1mm} \begin{tikzcd}
    \iota_Z^*(\EuScript{E})\arrow[d] \arrow[r] & \EuScript{E}\arrow[d]
\\
    X_Z^{\wedge}\arrow[r,"\iota_Z"] & X
\end{tikzcd}
\end{equation}
Consider 
$f^*X_Z^{\wedge}:=f^*(Z_{dR}\times_{X_{dR}}X),$
defined by an analogous pull-back square:
\begin{equation}
\label{eqn: Formal Normal}
\begin{tikzcd}
    f^*X_{Z}^{\wedge}\arrow[d]\arrow[r] & X_Z^{\wedge}\arrow[d]
    \\
    Z\arrow[r] & X
\end{tikzcd}
\end{equation}
Note that $id_X^*X_X^{\wedge}\simeq X.$
\begin{proposition}
 \label{proposition: Formal mapping restrictions}
    Let $f:Z\rightarrow X$ be a monomorphism of derived Artin stacks with $\EuScript{E}\in \PS_{/X}^{laft}.$ Then we have:
    \begin{itemize}
        \item an equivalence $(p_{dR}^Z)^*(p_{dR!}^Z)(Z)\times_Z f^*(X_Z^{\wedge})\simeq Z\times_{X_{dR}}X;$
        \vspace{1mm}
        
        \item a canonical epimorphism 
    $\mathsf{Maps}(X,\EuScript{E})\rightarrow \M(X_Z^{\wedge},\iota_Z^*\EuScript{E}).$
    \vspace{1mm}
    
        \item a canonical morphism
        $res_Z:\RS_X(\EQ)\rightarrow \RS_Z(f^*(X_Z^{\wedge}),\iota_Z^*\EQ),$ for all $\EQ$ satisfying $(\star)$, with $\iota_Z:f^*(X_Z^{\wedge})\rightarrow X$ the natural map.
    \end{itemize} 
\end{proposition}
Considering (\ref{eqn: Formal Nbhd}) and (\ref{eqn: Formal Normal}), Proposition \ref{proposition: Formal mapping restrictions} is seen to be true by considering the pull-back diagram
\[
\begin{tikzcd}
    p_{dR}^*p_{dR!}(Z)\times_Z f^*(X_Z^{\wedge})\arrow[d]\arrow[r] & f^*(X_Z^{\wedge})\arrow[d]\arrow[r] & X\arrow[d,"id_X"]
    \\
    p_{dR}^*p_{dR!}Z\arrow[d]\arrow[r] & Z\arrow[d,"p_{dR}"]\arrow[r,"f"] & X\arrow[d,"p_{dR}"]
    \\
    Z\arrow[r,"p_{dR}"]& Z_{dR}\arrow[r,"f_{dR}"] & X_{dR}
\end{tikzcd}
\]

One can think of $f^*(X_Z^{\wedge})$ as normal jets of solutions on $Z$ in $X.$

\section{Appendix.}

\subsection{Factorization}
In this section, all categories are viewed as DG-categories and then as stable $(\infty,1)$-categories in the usual manner.

\subsubsection{Factorization Objects}
We recall  convenient way to construct a category of factorization objects associated to a lax monoidal functor $F.$ Roughly, let $\PS_{\mathbf{Q}}^{\mathbf{P}}$ be the sub-category of $\PS$ with $\mathbf{P}$ a class of `properties' on morphisms defining a colimit presentation of a prestack $\EQ$, and $\mathbf{Q}$ a property of the objects $\EQ$ themselves. For instance, take $\mathbf{P}$ as closed-embeddings and $\mathbf{Q}$ to be laftness, then objects of $\PS_{laft}^{cl-emb}$ are pseudo ind-schemes. 

One first defines $F$-objects on a category of affine spaces $\mathsf{Aff},$ and Kan-extends along the Yoneda embedding.

Omitting the explicit details, we put
\begin{equation}
    \label{eqn: F-Object Functor}
\underline{F}_{\mathrm{PreStk}_{\mathbf{Q}}^{\mathbf{P}}}^!:=\mathrm{Res}\big(\mathrm{Kan}_{j}^{\mathrm{R}}(\underline{F}_{\mathsf{Aff}}^!)\big):\big(\mathrm{PreStk}_{\mathbf{Q}}^{\mathbf{P}}\big)^{\mathrm{op}}\rightarrow \mathrm{Cat}_{\mathrm{pres},\mathrm{cont}}^{\infty,\mathrm{st}},\end{equation}
where $\mathrm{Res}$ is the restriction along the functor $\mathrm{PreStk}_{\mathbf{Q}}^{\mathbf{P}}\hookrightarrow \mathrm{PreStk}.$ Right Kan-extension here is along the Yoneda, $j$. It can be calculated explicitly: if $\EQ\in \mathrm{PreStk}_{\mathbf{Q}}^{\mathbf{P}}$, represented as $\EQ\simeq \mathrm{colim} Z_i$ then
$$\underline{F}^!(\EQ)\simeq \mathrm{colim} \underline{F}^!(Z_i),$$
taken with respect to the diagram formed
by functors $f_{F}^!:=\overline{F}^!(f):\overline{F}^!(S_2)\rightarrow \overline{F}^!(S_2)$ for $f:S_1\rightarrow S_2.$
Typical examples might be: $$\underline{\mathrm{QCoh}}^*(-),\underline{\mathrm{IndCoh}}^!(-),\underline{\mathcal{D}}^*(-),$$ or
$\underline{\mathrm{Crys}}^r(-)=\underline{\mathrm{IndCoh}}^!(-)\circ (-)_{dR},$ or
$\underline{\mathrm{Crys}}^{\ell}=\underline{\mathrm{QCoh}}^*(-)\circ (-)_{dR}.$

Some of these agree when evaluated on classical smooth geometric objects $X.$ More generally, they might coincide when $X$ is presentable as a colimit whose structure maps are proper i.e. $X$ is an object of $\mathrm{PreStk}_{\mathrm{proper}}^{\mathrm{laft}}.$

Consider the Grothendieck construction in correspondences for (\ref{eqn: F-Object Functor}),
$$\EuScript{G}^{\mathrm{corr}}\big(\underline{F}_{\mathrm{PreStk}_{\mathbf{Q}}^{\mathbf{P}}}^!\big),$$and the $(\infty,1)$-category of commutative monoids $\mathrm{CAlg}\big[\EuScript{G}^{\mathrm{corr}}\big(\underline{F}_{\mathrm{PreStk}_{\mathbf{Q}}^{\mathbf{P}}}^!\big)\big],$ together with a natural forgetful functor to commutative monoids in correspondences in prestacks
$$\rho:\mathrm{CAlg}\big[\EuScript{G}^{\mathrm{corr}}\big(\underline{F}_{\mathrm{PreStk}_{\mathbf{Q}}^{\mathbf{P}}}^!\big)\big]\rightarrow \mathrm{CAlg}\big[\mathrm{Corr}(\mathrm{PreStk}_{\mathbf{Q}}^{\mathbf{P}})\big].$$
Its action on objects $\big(\EQ;\mathcal{E}\in F(\EQ)\big)$ is $\rho\big(\EQ;\mathcal{E}\in F(\EQ)\big):=\EQ.$

Recall that $\EQ\in \mathrm{CAlg}[Corr(PreStk)]$ consists of an object $\EQ\in PreStk,$ together with correspondences (multiplication and unity):
    \[
    \begin{tikzcd}
    & \arrow[dl,"m_1"]\mathrm{mult}_{\EQ}\arrow[dr,"m_2"]& 
    \\
    \EQ\times \EQ && \EQ
    \end{tikzcd}\hspace{2mm}
    \begin{tikzcd}
    & \arrow[dl,"\eta_1"]\mathrm{unit}_{\EQ}\arrow[dr,"\eta_2"]& 
    \\
    \mathrm{pt} && \EQ
    \end{tikzcd}
    \]
with higher arity analogs of the multiplication maps satisfying relations of a unital commutative monoid object (in the $\infty$-categorical sense).
Both $Ran_{\mathcal{X}},Ran_{\mathcal{X}}^{\mathrm{un}}\in \EuScript{CA}\mathrm{lg}\big(\EuScript{C}\mathrm{orr}(\EuScript{PS}\mathrm{tk})\big)$ i.e.
the multiplication is defined by the operation of disjoint union and the unit operation is determined by the map $\iota_{\emptyset}:\mathrm{pt}\hookrightarrow Ran_{\mathcal{X}}$ corresponding to inclusion of the empty set.
\begin{definition}
A \emph{weak multiplicative $F$-object} (on $\EQ$) is a section of $\rho.$
\end{definition}
Namely, it is the datum of an object $\mathcal{E}\in F(\EQ)$ 
such that there are maps
$$\nu:m_1^{!}(\mathcal{E}\boxtimes\mathcal{E})\rightarrow m_2^!(\mathcal{E}),$$
defining the multiplication, together with 
$\mu:e_1^!\mathbf{1}_{F(\mathrm{pt})}\rightarrow e_2^!(\mathcal{E})$ with higher arity analogues of the multiplications $\nu$ satisfying the relations of a unital commutative monoid. When required, we denote a given arity $n\geq 2$ map as $\nu^{(n)}:m^!(\mathcal{E}^{\boxtimes n})\rightarrow m_2^!(\mathcal{E}),$ where $m:\EQ^{\times n}\rightarrow \EQ$ is the map induced from $m_1$ and $m^!=F^!(m)$ is the functorial pull-back under our lax monoidal functor $F.$ We denote by $\mathbf{1}_{F(\mathrm{pt})}$ the unit for the canonical symmetric monoidal structure on $F(\mathrm{pt}).$

The $\infty$-category of weak multiplicative $F$-objects over $\EQ$ is denoted by $\mathrm{Mult}F(\EQ).$ It admits a full-subcategory of \emph{multiplicative $F$-objects}
$$\mathrm{Mult}F(\EQ)\hookrightarrow \mathrm{Mult}F(\EQ)^{\mathrm{weak}},$$
consisting of those for which $\nu$ (and hence all higher arity operations) and $\mu$ are homotopy equivalences.

 One can ask what it means for this multiplication to `factorize.'
\begin{definition}
\label{Definition: Factorization F-Object}
\normalfont 
A \emph{unital factorization $F$-object} on $X$ is an object of $\mathrm{Mult}F(Ran_X^{\mathrm{un}}).$
\end{definition}
Explicitly, this Definition \ref{Definition: Factorization F-Object} gives an object $\mathcal{E}=(\mathcal{E}^I)\in F(Ran_X^{\mathrm{un}})$, together with isomorphisms
\begin{equation}
    \label{eqn: Factorization Isomorphisms}
    \alpha:j_{\mathrm{disj}}^!\big(\mathcal{E}\boxtimes\mathcal{E}\big)\rightarrow\bigsqcup^{!}\mathcal{E},\hspace{3mm}\text{ in }\hspace{2mm} F^!\big((Ran_X^{\mathrm{un}},\times Ran_X^{\mathrm{un}})_{\mathrm{disj}}\big),
\end{equation}
together with an isomorphism
$\beta:\mathbf{1}_{F(\mathrm{pt})}\rightarrow \iota_{\emptyset}^!(\mathcal{E})$ in $F(\mathrm{pt}),$ correpsonding to the inclusion (as a map of prestacks) $\iota_{\emptyset}:\mathrm{pt}\hookrightarrow Ran_X^{\mathrm{un}},$ together with all higher-arity analogs for the maps (\ref{eqn: Factorization Isomorphisms}).
Let $(-)^{\mathrm{op}}$ denote the endo-functor taking the opposite category and put $F_{\mathrm{op}}:=(-)^{\mathrm{op}}\circ F.$ Then a
co-unital factorization $F$ object on $X$ is a unital $F_{\mathrm{op}}$-object.

Denote the categories of factorization objects by 
$$\mathrm{Fact}F_{\mathrm{un}}^!(X):=\mathrm{Mult}F^!(Ran_{X}^{\mathrm{un}}),\hspace{2mm} \text{ and }\hspace{2mm} \mathrm{Fact}F_{\mathrm{co-un}}^!(X):=\mathrm{Mult}F_{\mathrm{op}}^!(Ran_X^{\mathrm{un}}).$$

We give some examples of factorization objects taking the  opportunity to fix notations.
\begin{example}
\label{example: Factorization F-Objects}
\normalfont 
\noindent (a)\hspace{2mm} $F$ is $\underline{\mathrm{QCoh}}^*(-):$ Gives the category of factorization quasi-coherent sheaves on $X$, to be denoted $\mathrm{QCoh}^{\mathrm{fact},\bullet}(X),$ with respect to the quasi-coherent pull-back operation, denoted for a morphism $f$ by $f_{\mathrm{QCoh}}^*.$
\vspace{2mm}

\noindent (b)\hspace{2mm} $\underline{\mathrm{IndCoh}}^!(-):$ Gives the category of factorization ind-coherent sheaves on $X$ (see \cite{GaiRoz01,GaiRoz02}), to be denoted $\mathrm{IndCoh}^{\mathrm{fact},!}(X),$ with respect to the $!$-pull-back of ind-coherent sheaves $f_{\mathrm{IndCoh}}^!.$
\vspace{2mm}

\noindent (c)\hspace{2mm} $\underline{\D}^*(-):$ Gives the category of factorization $\D$-modules with respect to the $\D$-module pull-back denoted $\mathsf{FAlg}_{\mathcal{D}}^{*}(X).$ They are `usual' factorization algebras of \cite{BeiDri,FraGai}.
\vspace{1.5mm}

\noindent (d)\hspace{2mm} 
For $\underline{\mathrm{PreStk}}_{/(-)}$ given by slicing over a pre-stack, we obtain unital factorization pre-stacks (also called unital factorization spaces) $\mathrm{PreStk}_{\mathrm{un}}^{\mathrm{fact}}(X)$.
\end{example}
More importantly for us, are those factorization $F$-objects over not simply $X$, but its de Rham space $Z=X_{dR}.$ In this situation item (e) in Example \ref{example: Factorization F-Objects} gives \emph{unital factorization $\mathcal{D}$ spaces over $X$} i.e.
$$\mathrm{PreStk}_{\mathcal{D},\mathrm{un}}^{\mathrm{fact}}(X):=\mathrm{PreStk}_{\mathrm{un}}^{\mathrm{fact}}(X_{dR}).$$
Over $Z=X_{dR}$, items (a) and (c) are equivalent 
$\mathrm{QCoh}^{\mathrm{fact},\mathrm{un}}(Z)\simeq \mathsf{FAlg}_{\mathcal{D}}^!(X)^{\mathrm{un}}.$

\subsubsection{Homotopical Multi-jets}
\label{sssec: Homotopical Multi-jets}

Consider $\PS_{/Ran_X}$ and the canonical map $q_{dR}:=p_{Ran_X,dR}:Ran_X\rightarrow Ran_{X_{dR}}$ which defines an adjunction
\begin{equation}
    \label{eqn: Adjunction for multijets}
    \begin{tikzcd}
    q_{dR}^*\colon \PS_{/ Ran_{X_{dR}}} \arrow[r,shift left=.5ex]& \PS_{/Ran_X}\arrow[l,shift left=.5ex]\colon q_{dR,*}.
    \end{tikzcd}   
    \end{equation}
The essential image of the pushforward in adjunction (\ref{eqn: Adjunction for multijets}) is called a \emph{multi-jet $\mathcal{D}$-space}, which is denoted by 
$$\J(\EuScript{E})^{Ran}:=q_{dR,*}(\EuScript{E})\rightarrow Ran_{X_{dR}},$$
for some $\EuScript{E}\rightarrow Ran_X$ prestack relative to the prestack $Ran_X.$

The main result we use throughout the paper is presented here as a collection of various facts pertaining to commutative $\D$-algebras, factorization algebras and factorization $\D$-spaces.

\begin{proposition}
\label{Multi-jet r-adjoint proposition}
Let $X$ be a smooth affine variety. There is a natural functor 
$$U:\mathsf{FAlg}_{\mathcal{D}}(X)_{comm}\rightarrow \mathsf{CAlg}(\mathcal{D}_{Ran_X})\rightarrow \mathsf{CAlg}_X(\mathcal{D}_X),$$
which has a right-adjoint,
$$\mathcal{F}:\mathsf{CAlg}_X(\mathcal{D}_X)\rightarrow \mathsf{FAlg}_X(\mathcal{D}_X)_{comm},$$
that roughly speaking, views a derived $\mathcal{D}$-algebra $\mathcal{A},$ with the natural maps of $n$-fold products $\beta_n:j_*j^*(\mathcal{A}\boxtimes\cdots\boxtimes \mathcal{A})\rightarrow \Delta_*(\mathcal{A}\otimes \cdots\otimes\mathcal{A}),$ in terms of a fibration sequence, $$\mathcal{A}\boxtimes\cdots\boxtimes\mathcal{A}\rightarrow j_*j^*(\mathcal{A}\boxtimes\cdots\boxtimes \mathcal{A})\rightarrow \Delta_*(\mathcal{A}\otimes \cdots\otimes\mathcal{A}).$$

Moreover, $\mathcal{F}$ extends to a functor
$$\mathsf{FAlg}_{\mathcal{D}}(X)_{comm}\rightarrow \mathsf{PreStk}_{\mathcal{D}}^{fact}(X)_{comm}^{\mathrm{co-un},\mathrm{aff}},$$ to affine co-unital factorizaion $\mathcal{D}$-spaces.
The natural forgetful $\infty$-functor 
$$U_{fact}:\mathsf{PreStk}_{\mathcal{D}}^{fact}(X)\rightarrow \PS_{X_{dR}}, Z^{Ran}\mapsto Z^{(1)},$$ also has a right-adjoint -- the homotopical multi-jet $\infty$-functor. 
\end{proposition}
The forgetful functor $U_{fact}$ is the induced functor on pre-stacks given by restriction along the diagonal i.e.
$\mathsf{PreStk}_{Ran_{X_{dR}}}\xrightarrow{\Delta_{dR}^*}\mathsf{PreStk}_{X_{dR}},$ restricted to the sub-category of multiplicative objects $\PS_{\mathcal{D}}^{fact}(X)\subset \PS_{Ran_{X_{dR}}},$ as in Example \ref{example: Factorization F-Objects}.
\vspace{2mm}

\noindent\textit{Sheaves of Categories.}
The appropriate categorification of structure sheaves of derived (factorization) spaces we use is given by the theory of quasi-coherent sheaves of categories \cite{Gai1Aff}.

If $X$ is affine $\mathsf{ShvCat}(X):=\mathsf{QCoh}(X)-\mathsf{Mod}\big(\mathsf{DGCat}_{cont}\big)$ and one extends this to prestacks $\EQ$ via $f^*\EuScript{QC}oh_{\EQ}=\EuScript{QC}oh_X$ for every $f:X\rightarrow \EQ.$

Sheaves of categories $\EuScript{QC}oh_{\EQ}$ indeed categoryify the structure sheaf in the sense that we may take their `global sections categories'
$$\mathbf{\Gamma}(\EQ,-):\mathsf{ShvCat}(\EQ)\rightarrow \mathsf{DGCat}_{cont},$$
such that $$\mathbf{\Gamma}(\EQ,\EuScript{QC}oh_{\EQ})=\mathsf{QCoh}(\EQ).$$

The main example of a (unital) sheaf of factorization category we use is described as follows.

Let $\mathcal{X}^{Ran}$ be a unital factorization $\mathcal{D}$-space. In particular for all $\alpha$ we have maps $
\mu_{\alpha}:X_{\mathrm{dR}}^J\times_{X_{\mathrm{dR}}^I}\EQ^{(I)}\rightarrow \EQ^{(J)}.$

The unital $\mathcal{D}$-factorization category of sheaves on $\mathcal{X}^{Ran}$, is $$\mathrm{Shv}_{\mathcal{X}}^{Ran}:=\big\{\rho_!^{(I)}\mathrm{Shv}_{\mathcal{X}^{(I)}}\big\},$$
specified by the assignment $\mathrm{fSet}\ni I\mapsto \rho_!^{(I)}\mathrm{Shv}_{\mathcal{X}^{(I)}}\in \mathrm{Shv}\EuScript{C}\mathrm{at}(X_{dR}^I),$ whose structure maps are of the form
\begin{equation}
\label{eqn: Unital Fact Cat Structure Maps}
    \Phi_{\alpha}:=\rho_!^{(J)}(\nu_{\alpha\bullet})\circ \widetilde{\rho}_!^{(I)}(\mu_{\alpha}^{\bullet}):\Delta_{\alpha}^!\rho_!^{(I)}\mathrm{Shv}_{\mathcal{X}^{(I)}}\rightarrow \rho_{!}^{(J)}\mathrm{Shv}^{(J)}.
\end{equation}
Remark the natural equivalences 
$$\Delta_{\alpha}^!\rho_!^{(I)}\mathrm{Shv}_{\mathcal{X}^{(I)}}\simeq \widetilde{\rho}_!^{(I)}\mathrm{Shv}_{X_{dR}^J\times_{X_{dR}^I}\mathcal{X}^{(I)}}.$$
From base-change along  $X_{dR}^J\times_{X_{dR}^I}\mathcal{X}^{(I)},$ one can see that the unital structure maps (\ref{eqn: Unital Fact Cat Structure Maps}) are well-defined and in the case when $\alpha$ is a partition, are even equivalences of sheaves of categories.

Indeed, notice that the unit of the above equivalence gives 
$$\mu_{\alpha}^{\bullet}:\mathrm{Shv}_{X_{dR}^J\times_{X_{dR}^I}\mathcal{X}^{(I)}}\rightarrow \mu_{\alpha !}\mu_{\alpha}^!\mathrm{Shv}_{X_{dR}^J\times_{X_{dR}^I}\mathcal{X}^{(I)}},$$ and the right-adjoint $\nu_{\alpha\bullet}$ to $\nu_{\alpha}^{\bullet},$ that the map $\Phi_{\alpha}$ we are looking for is indeed given by the composition
\begin{eqnarray*}
\widetilde{\rho}_!^{(I)}\mathrm{Shv}_{X_{dR}^J\times_{X_{dR}^I}\mathcal{X}^{(I)}}&\xrightarrow{\rho_!^{(I)}(\mu_{\alpha}^{\bullet})}& \widetilde{\rho}_!^{(I)}\mu_{\alpha !}\mu_{\alpha}^!\mathrm{Shv}_{X_{dR}^J\times_{X_{dR}^I}\mathcal{X}^{(I)}}
\\
&\simeq & \rho_!^{(J)}\nu_{\alpha !}\nu_{\alpha}^!\mathrm{Shv}_{\mathcal{X}^{(J)}}
\\
&\xrightarrow{\rho_!^{(J)}(\nu_{\alpha\bullet})}& \rho_!^{(J)}\mathrm{Shv}_{\mathcal{X}^{(J)}}.
\end{eqnarray*}
Note the factorizaion equivalences are of the form 
$$j(\alpha)^!\rho_!^{(I)}\mathrm{Shv}_{\mathcal{X}^{(I)}}\xrightarrow{\simeq}j(\alpha)^!\bigboxtimes_{j\in J}\rho_!^{(I_j)}\mathrm{Shv}_{\mathcal{X}^{(I_j)}},$$
and taking the image under the functor $(\rho^{Ran})_!^{\mathrm{Shv}},$ of the units $\mu_{\alpha}^{\bullet}$ and the right-adjoint $\nu_{\alpha\bullet},$ we have the factorization unit maps
$$\rho_!^{(I)}\mathrm{Shv}_{\mathcal{X}^{(I)}}\boxtimes \mathrm{Shv}_{X_{dR}^{I_{\alpha}}}\rightarrow \rho_!^{(J)}\mathrm{Shv}_{\mathcal{X}^{(J)}},\hspace{5mm} \mathrm{Shv}_{X_{dR}^J}\rightarrow \rho_!^{(J)}\mathrm{Shv}_{\mathcal{X}^{(J)}}.$$

\subsubsection{Factorization pull-back.}
Consider an admissible morphism of unital factorization $\mathcal{D}$-spaces $f^{Ran}:\mathcal{X}^{Ran}\rightarrow \EQ^{Ran}.$
There is an induced (unital) functor of unital $\mathcal{D}$-factorization categories
\begin{equation}
    \label{eqn: Pull-Back Factorization Functor}
\big(f^{Ran}\big)^*:\mathrm{Shv}_{\EQ}^{Ran}\rightarrow \mathrm{Shv}_{\mathcal{X}}^{Ran},
\end{equation}
determined by the following assignments:
\begin{itemize}
    \item For each $I\in \mathrm{fSet}$ there is a morphism in $\mathrm{Shv}\EuScript{C}\mathrm{at}\big(X_{dR}^I\big)$, given by $f^{(I),*}:\rho_{\EQ!}^{(I)}\mathrm{Shv}_{\EQ^{(I)}}\rightarrow \rho_{\mathcal{X}!}^{(I)}\mathrm{Shv}_{\mathcal{X}^{(I)}};$
    
    \item For each $\alpha:I\rightarrow J$ there is a natural transformation of functors defined by the (homotopically) commuting diagram
    \[
    \begin{tikzcd}
    \widetilde{\rho}_{\EQ!}^{(I)}\mathrm{Shv}_{X_{dR}^J\times_{X_{dR}^I}\EQ^{(I)}}\arrow[d,"\widetilde{f}^{(I),*}"] \arrow[r] & \rho_{\EQ,!}^{(J)}\mathrm{Shv}_{\EQ^{(J)}}
    \arrow[d,"f^{(J),*}"]
    \\
     \widetilde{\rho}_{\mathcal{X}!}^{(I)}\mathrm{Shv}_{X_{dR}^J\times_{X_{dR}^I}\mathcal{X}^{(I)}}\arrow[r] & \rho_{\mathcal{X}!}^{(J)}\mathrm{Shv}_{\mathcal{X}^{(J)}}
    \end{tikzcd}
    \]
    whose horizontal arrows are determined by the unital structure maps (\ref{eqn: Unital Fact Cat Structure Maps}).
\end{itemize}
\subsubsection{Factorization push-forward.}
There is an induced (unital) functor of unital $\mathcal{D}$-factorization categories
\begin{equation}
 \label{eqn: Push-Forward Factorization Functor}
 \big(f^{Ran}\big)_*:\mathrm{Shv}_{\mathcal{X}}^{Ran}\rightarrow \mathrm{Shv}_{\EQ}^{Ran},
 \end{equation}
 determined by the following:
 \begin{itemize}
     \item For each $I\in \mathrm{fSet}$ there is a morphism in $\mathrm{Shv}\EuScript{C}\mathrm{at}(X_{dR}^I),$ given as
     $f_*^{(I)}:\rho_{\mathcal{X}!}^{(I)}\mathrm{Shv}_{\mathcal{X}^{(I)}}\rightarrow \rho_{\EQ!}^{(I)}\mathrm{Shv}_{\EQ}^{(I)}$;
     
     \item For each $\alpha,$ there is a natural transformation of morphisms of sheaves of categories defined by the homotopy-commuting diagram
     
     \[
    \begin{tikzcd}
    \widetilde{\rho}_{\mathcal{X}!}^{(I)}\mathrm{Shv}_{X_{dR}^J\times_{X_{dR}^I}\mathcal{X}^{(I)}}\arrow[d,"\widetilde{f}_*^{(I)}"] \arrow[r] & \rho_{\mathcal{X},!}^{(J)}\mathrm{Shv}_{\mathcal{X}^{(J)}}
    \arrow[d,"f_*^{(J)}"]
    \\
     \widetilde{\rho}_{\EQ!}^{(I)}\mathrm{Shv}_{X_{dR}^J\times_{X_{dR}^I}\EQ^{(I)}}\arrow[r] & \rho_{\EQ!}^{(J)}\mathrm{Shv}_{\EQ^{(J)}}
    \end{tikzcd}
    \]
 \end{itemize}

\bibliography{Bibliography}
  \bibliographystyle{plain}
  
\end{document}